\documentclass[notitlepage,reqno,11pt]{amsart}
\usepackage{latexsym,amssymb, epsfig, amsmath,amsfonts, subfigure,amsthm}

\usepackage{rotating}
\usepackage[toc,page]{appendix}
\usepackage{color}
\usepackage{multirow}
\usepackage{relsize}
\usepackage{microtype}
\usepackage{dsfont}
\usepackage{cases}
\usepackage{enumitem}

\usepackage[colorlinks, allcolors=blue]{hyperref}

\usepackage[margin=1in]{geometry}

\numberwithin{equation}{section}
\newtheorem{theorem}{Theorem}[section]
\newtheorem{lemma}{Lemma}[section]
\newtheorem{assumption}{Assumption}[section]

\newtheorem{coro}{Corollary}[section]
\newtheorem{prop}{Proposition}[section]
\newtheorem{remark}{Remark}[section]

\newtheorem{example}{Example}[section]

\newlength{\defbaselineskip}
\newcommand{\setlinespacing}[1]%
{\setlength{\baselineskip}{#1 \defbaselineskip}}

\def\E{\mathbb{E}}
\def\P{\mathbb{P}}

				\newcommand{\R}  {\mathbb{R}}
				\newcommand{\N}  {\mathbb{N}}

				\newcommand{\non}{\nonumber}

				\newcommand{\baa}{\begin{eqnarray*}}
					\newcommand{\eaa}{\end{eqnarray*}}
				
				\newcommand{\fF}{\overline{\mathfrak F}}
				\newcommand{\fS}{\overline{\mathfrak S}}

				\newcommand{\ttl}{\Large
					Stochastic epidemic models with pulse vaccination, varying infectivity and waning immunity}
			
			\newcommand{\indic}[1]{\mathds{1}_{#1}}
			\renewcommand{\bar}[1]{\overline{#1}}

\begin{document}
				
				\title[]{\ttl}
				
				%
				%
				%
				%
				\author[Arsene Brice Zotsa Ngoufack]{Arsene Brice Zotsa Ngoufack}
				\address{Vanderbilt University, Nashville, TN, USA}
				\email{bricezotsa@gmail.com}

				\date{\today}
				
					\begin{abstract}
				We introduce a fully stochastic, non‑Markovian SIRS‑type epidemic model that incorporates varying infectivity, waning immunity, and a pulse vaccination strategy that may not confer permanent immunity. The model is constructed at the individual level, where each person is characterized by random infectivity and susceptibility functions, and vaccination campaigns occur at the jump times of a Poisson random measure with arbitrary intensity. We rigorously derive the epidemic dynamics as the large-population limit of an interacting stochastic particle system, leading to a system of nonlinear Volterra-type integral equations governing the average susceptibility and total force of infection. We establish a functional law of large numbers(FLLN) for the empirical processes and provide explicit expressions for the limiting compartmental proportions. The long-term behavior of the system is analyzed: we prove that the infection-free solution is globally asymptotically stable when the basic reproduction number falls below a critical threshold, and that the disease persists when this threshold is exceeded. The threshold is given by the harmonic mean of the maximal susceptibility across individuals and generalizes previous results by incorporating pulse vaccination. Our framework provides a probabilistically grounded extension of classical deterministic pulse vaccination models and offers new insights into the control of epidemics through scheduled immunization policies.
			\end{abstract}

				\subjclass[2020]{60F17; 35Q92; 60K35; 35B40; 92D30}
				
				\keywords{Stochastic epidemic model, age-structured model, age of infection, age of vaccination, varying infectivity, varying immunity/susceptibility; Pulse vaccination, Periodical vaccine, Vaccination campaign, endemicity, local stability, Volterra integral equation, Poisson random measure, Measure-valued process.}
				
				\maketitle
				\allowdisplaybreaks
				
				\section*{Introduction}

				Vaccination remains one of the most effective strategies for controlling infectious diseases. Traditional epidemic models often assume that vaccination is deployed reactively, in response to case detection or outbreak dynamics \cite{britton2019stochastic}. However, in real-world settings, particularly under resource constraints or public health protocols, vaccination is frequently administered according to predetermined schedules, age groups, or mass campaigns that operate independently of the current epidemic state \cite{hadjipanayis2019compliance, who2016grisp}.
				
				To reflect such practices, researchers have developed epidemic models incorporating periodic vaccination. In \cite{shulgin1998pulse, stone2000theoretical}, the authors studied an SIR model with pulse vaccination, showing that repeatedly vaccinating the susceptible population at discrete time intervals can lead to disease eradication. They formulated the model as a system of ordinary differential equations (ODEs) and analyzed its long-term behavior. In particular, they identified the existence of an infection-free solution, defined as a solution in which the number of infected individuals remains identically zero over time. Under a suitable epidemic threshold condition, they proved that this infection-free solution is asymptotically stable. Furthermore, they demonstrated that the epidemic dynamics converge asymptotically toward this solution. In such cases, the infection-free solution is said to be globally attractive in the literature.
				
				This framework was extended in \cite{meng2010delay, gao2006analysis, gao2007analysis}, where the authors introduced time delays to account for incubation periods and distributed infection dynamics. In particular, \cite{gao2007analysis} proposed an SIR model with pulse vaccination and distributed time delay, showing that the infection-free solution remains globally attractive when the vaccination rate exceeds a critical threshold. Conversely, they established that the disease becomes uniformly persistent when the vaccination rate falls below this threshold. In \cite{gao2006analysis}, the same authors developed an SEIRS model with time delay and pulse vaccination, further enriching the deterministic modeling framework.
				
				Additional deterministic models with pulse vaccination and delay effects were proposed in \cite{jin2008pulse, nagy2011epidemic, liu2017pulse}. However, these models remain deterministic and do not capture the inherent randomness of epidemic spread. A notable exception is \cite{wang2014pulse}, where the authors introduced stochasticity by adding Gaussian white noise to the classical ODE-based SIR model with pulse vaccination. They analyzed the resulting stochastic differential equation and studied the stability and global attractiveness of the infection-free solution. 
				
				Despite these advances, to the best of our knowledge, a stochastic epidemic model that rigorously incorporates a flexible pulse vaccination strategy has not yet been developed. Most existing models with pulse vaccination remain deterministic, while the literature on more advanced stochastic epidemic models, particularly those without pulse vaccination has evolved significantly. For instance, in \cite{forien-Zotsa2022stochastic,forien2026stochastic}, the authors introduced a stochastic SIRS-type model with varying infectivity and waning immunity. In this framework, recovered individuals gradually become susceptible again, and each individual is characterized by random infectivity and susceptibility functions, drawn independently at each infection event. The epidemic dynamics are governed by a system of nonlinear Volterra-type integral equations describing the evolution of average susceptibility and the total force of infection. The authors analyzed the long-term behavior of the system and showed that the epidemic threshold for disease extinction depends on the population’s immunity profile, thereby generalizing classical threshold results. The model is non-Markovian and captures memory effects in the infection process. The fluctuations of this model were studied in \cite{ngoufack2025functional}. A particular case focusing on vaccination policies was considered in \cite{foutel2025optimal}, where individuals are vaccinated after a random delay following recovery. This model was further extended in \cite{guerin2025stochastic} by incorporating memory of past infections, and its fluctuations were characterized in \cite{CLT_zotsa2025stochastic}. For foundational literature on classical stochastic epidemic models, we refer to \cite{britton2019stochastic}, and for non-Markovian models, see \cite{pang2022functional, forien_epidemic_2021}.
				
				\medskip
				In this paper, we introduce a novel stochastic epidemic model that incorporates varying infectivity, waning natural immunity, and a vaccination strategy independent of the epidemic state. This framework is designed to reflect practical interventions such as routine immunization programs and mass vaccination campaigns. 
				Following the methodology in \cite{forien-Zotsa2022stochastic}, we assume that each individual is characterized by random infectivity and susceptibility functions, depending on their age of infection or vaccination and drawn independently at each new infection or vaccination event. Unlike the existing pulse‑vaccination literature, which generally assumes that immunization confers permanent and complete protection, our model incorporates a more realistic scenario in which vaccination may or may not provide permanent immunity. Thus, at each vaccination event, we draw a new random function, now depending on the time elapsed since vaccination to represent the individual current susceptibility, while their infectivity becomes zero. 
				
				The vaccination policy is modeled as a point process, with campaign times corresponding to the jump times of a Poisson random measure. The intensity of this measure is given by an arbitrary positive sigma-finite measure, allowing for flexible specification of both periodic and irregular vaccination schedules. When the intensity measure is identically zero, our model reduces to the one studied in \cite{forien-Zotsa2022stochastic}. The main distinction between our stochastic model and that of \cite{forien-Zotsa2022stochastic} lies in the definitions of current infectivity and current susceptibility. In our framework, these functions depend on the age of infection, the age of vaccination, and the infection age at the time of vaccination. This contrasts with \cite{forien-Zotsa2022stochastic}, where these functions depend solely on the age of infection. More specifically, we use the infection age at the time of vaccination defined as the time elapsed between the infection event and the vaccination campaign to determine an individual’s eligibility for vaccination. A transition to the vaccinated state occurs only if this duration exceeds the total length of the current infectious period.
				
				\medskip
				The epidemic dynamics are governed by a system of nonlinear integral equations describing the evolution of the average susceptibility and the total force of infection. From this system, we derive explicit representations for the limiting proportions of individuals in each epidemiological compartment (Theorem~\ref{Th:FLLN}). In Proposition~\ref{prop-FLLN}, we present a specific case of this theorem corresponding to a compartmental non-Markovian SIRS model. Furthermore, in Example~\ref{Ex-Model-SIRS}, we retrieve the classical ODE model by assuming that all durations follow exponential distributions. This example is of particular interest as it constructs a hybrid model incorporating both continuous and discrete vaccination; such a hybrid approach was emphasized in \cite{bolatova2024mathematical}.
				We then analyze the long-term behavior of the system and characterize the conditions for disease extinction or persistence. Specifically, we demonstrate that when the basic reproduction number $R_0$ is below unity, the disease-free equilibrium is globally asymptotically stable. In the specific case where vaccination confers permanent immunity, we show that the epidemic threshold is determined by the harmonic mean of the maximal susceptibility across individuals (Theorem~\ref{Th-LT-free}), coinciding with the threshold identified in~\cite{forien-Zotsa2022stochastic}. Consequently, the disease-free solution is globally attractive, and the epidemic becomes extinct (Corollary~\ref{coro-endemic-Th-d-free}). Conversely, when this threshold is exceeded, the disease persists in the population( Theorem~\ref{end-th-persistence}). In particular, for the compartmental non-Markovian SIRS model introduced in Proposition~\ref{prop-FLLN}, we establish that the disease-free steady state is asymptotically stable if $R_0<1$, whereas the disease persists if $R_0>1.$

				\medskip
				We observe that the system of Volterra integral equations obtained via our FLLN depends explicitly on both the age of infection and the age of vaccination. Under the assumption that infected individuals cannot be reinfected during their infectious period, the average total force of infection in our model coincides with that of the model in~\cite{forien-Zotsa2022stochastic}. However, the average susceptibility differs significantly, reflecting the structural influence of the vaccination policy. In the general case, where the assumption preventing reinfection during the infectious period is removed, both the total force of infection and the average susceptibility diverge from those in~\cite{forien-Zotsa2022stochastic}. These discrepancies arise from additional exponential decay terms that capture the cumulative impact of the vaccination pulses. Specifically, the vaccination intensity dictates the rate at which individuals transition to the immunized state, thereby reducing both individual susceptibility and the overall transmission potential within the population.
				
				\medskip
				The core of the proof relies on a coupling argument between the process counting the number of infections in the population and a family of i.i.d. counting processes in which the renormalized force of infection is replaced by its mean. Although certain parts of the argument build on techniques developed in \cite{forien-Zotsa2022stochastic}, the analysis is considerably more delicate here due to the introduction of a vaccination policy modeled by a Poisson random measure. This is not a straightforward extension: both the total force of infection and the average susceptibility must be decomposed into two distinct contributions. The first corresponds to individuals who will not be vaccinated at the time of vaccination, and the second to those who will be vaccinated.
				A major challenge of this work lies in obtaining the explicit expression of the limit for our stochastic model (see Section~\ref{sec-exp-F-G}). To compute the expectations of the two contributing terms, it is necessary to handle carefully the indicator functions involving the infection age at the time of vaccination. These indicators determine the appropriate reinfection rate, which depends on whether the individual’s current susceptibility is governed by natural immunity or by vaccine‑induced immunity.
				It was particularly difficult to compute the first‑order moment of the average susceptibility for vaccinated individuals. A direct computation would require determining the exponential moments of a Poisson stochastic integral with a non‑predictable integrand, which is a notoriously complex task. To overcome this difficulty, we establish two theorems in Section~\ref{sec-poiss-integral} that provide formulas for the moment exponents of Poisson stochastic integrals. Compared with the results of \cite{breton2014factorial, decreusefond2014moment, privault2012moments}, these formulas are significantly more tractable. However, we ultimately adopted an alternative approach: computing the probability that an individual is still infected at the time of vaccination. This method allows us to avoid the difficulties associated with Poisson integrals with non‑predictable integrands.
				The existence of a solution to the resulting system of integral equations is established via the Leray–Schauder fixed point theorem. Proving the existence of a solution is non‑trivial, as we seek a càdlàg (right‑continuous with left limits) solution, which is uncommon in the existing literature for this type of problem. To apply the Leray–Schauder theorem, one must carefully examine the topology of the functional space under consideration. Since the theorem requires a Banach space, we cannot work directly in the space of càdlàg functions equipped with the Skorokhod topology. Instead, we first establish the existence of a solution in $L^\infty$ (Space of essential bounded functions), which is a Banach space, and then conclude that the solution is indeed càdlàg based on its explicit integral representation.
				Furthermore, a central requirement of the Leray–Schauder theorem is compactness. Establishing compactness in $L^\infty$, is far from standard; nevertheless, we satisfy this requirement by relying on Cherkas’s $1970$ characterization (see \cite{cherkas1970compactness}). 
				The remainder of the proof follows the methodology established in \cite{forien-Zotsa2022stochastic}.

				\medskip
				This work opens the door to the development of more effective computational models and new theoretical approaches to support epidemic control through vaccination campaigns. In particular, we can leverage this idea to extend the model developed in \cite{baghdadi2026stochastic} to propose a periodic vaccination plan to eradicate malaria.
				
				\subsection*{Organization of the paper} The remainder of the paper is organized as follows. In Section~\ref{sec-model}, we present a detailed description of the model, state the main assumptions and we present the statement of the functional law of large numbers (FLLN), along with a discussion of how our results relate to existing models. The analysis of the endemic equilibrium is carried out in Section~\ref{sec-endemic-eq}. The proofs of the FLLN are provided in Section~\ref{proof-FLLN}, while those concerning the endemic equilibrium are given in Section~\ref{sec-LT}.
				
				\subsection*{Notation}
				Throughout the paper, all the random variables and processes are defined on a common complete probability space $(\Omega, \mathcal{F},\P)$.  
				We use $\xrightarrow[N\to+\infty]{\mathbb{P}}$ to denote convergence in probability as the parameter $N\to \infty$.
				Let $\N$ denote the set of natural numbers and $\R^k (\R^k_+)$ the space of $k$-dimensional vectors with real (nonnegative) numbers, with $\R(\R_+)$ for $k=1$.  We use $\indic{\{\cdot\}}$ for the indicator function. Let $D=D(\R_+; \R)$ be the space of $\R$-valued c{\`a}dl{\`a}g functions defined on $\R_+$, with convergence in $D$ meaning convergence in the Skorohod $J_1$ topology (see, e.g., \cite[Chapter 3]{billingsley1999convergence}).  Also, we use $D^k$ to denote the $k$-fold product with the product $J_1$ topology. Let $C$ be the subset of $D$ consisting of continuous functions and $D_+$ the subset of $D$ of c{\`a}dl{\`a}g functions with values on $\R_+$

				\bigskip 
				\section{Model Description}\label{sec-model}
				\subsection{Definition of the model}\label{sec-model-descr}
				We consider a population of fixed size $N$. 
				Let $(\lambda_0^I,\gamma_0^I)$ and $(\lambda^I,\gamma^I)$ be two random variables taking values in $D(\R_+, \R_+)^2$, and $\gamma^V$ a random variable taking values in $D(\R_+,\R_+)$. Let $ \lbrace (\lambda_{k,0}^I, \gamma_{k,0}^I), 1 \leq k \leq N \rbrace $ be a family of i.i.d. copies of $ (\lambda_0^I, \gamma_0^I) $ and $ \lbrace (\lambda_{k,i}^I, \gamma_{k,i}^I), i \geq 1, 1 \leq k \leq N \rbrace $ be a family of i.i.d. copies of $ (\lambda^I, \gamma^I) $, independent of the previous family. Let $\lbrace (\gamma^V_{k,j}),\,1\leq k\leq N, j\geq0\rbrace$ be a family of copies of $\gamma^V$, independent of the previous two families.
				The function $\lambda_{k,0}^I$ (resp. $\gamma_{k,0}^I$) is the  infectivity (resp. susceptibility) of the $ k $-th individual between time $0$ and the time of their first (re)-infection, and $\lambda_{k,i}^I(t)$ denotes the infectivity of the $k$-th individual, at time $t$ after their $i$-th infection (where we count here only the infections after time $0$), $\gamma_{k,i}^I(t)$ denotes the susceptibility of the $k$-th individual, at time $t$ after their $i$-th infection, given that this is their most recent infection, and $\gamma^V_{k,j}(t)$ denotes the susceptibility of individual $k$, at time $t$ after their $j$-th vaccination.
				
				We can think of the law of $(\lambda_0^I,\gamma_0^I)$ as being a mixture of the law for the initially infected individuals, who have been infected before time $0$ and for which $\lambda_0^I(0)\ge0$ and $\gamma_0^I(0)=0$, and the law for the initially susceptible individuals, for which $\lambda_0^I(0)=0$ and $\gamma_0^I(0)>0$ (possibly $\gamma_0^I(0)=1$).
				
				We will assume that $\gamma^V_{k,0}=0.$
				
				We assume that $(\lambda_0^I,\gamma_0^I)$, $(\lambda^I,\gamma^I)$ and $\gamma^V$ satisfy the following assumption.
				\begin{assumption}\label{AS-lambda}
					We assume that:
					\begin{enumerate}
						\item $ 0 \leq \gamma_0^I(t)\leq1,\,0\leq\gamma^I(t) \leq 1$ and $0\leq\gamma^V(t) \leq 1 $  almost surely. 
						\item There exists a deterministic constant $ \lambda_* < \infty $ such that for all $ t \geq 0,\, 0\leq\lambda_0^I(t)\leq\lambda_*,\,$ $0\leq\lambda^I(t) \leq \lambda_* $ almost surely.
					\end{enumerate}
				\end{assumption}
				
				For $ i \geq 0 $ and $ 1 \leq k \leq N $, we define the duration of infection: 
				\begin{align*}
					\eta_{k,0} := \sup \lbrace t \geq 0 : \lambda_{k,0}^I(t) > 0 \rbrace, \quad \text{ and }\quad
					\eta_{k,i} := \sup \lbrace t \geq 0 : \lambda_{k,i}^I(t) > 0 \rbrace.
				\end{align*}
				We will also use the notations
				\begin{equation} \label{def:eta}
					\eta_0=\sup\{t>0,\ \lambda_0^I(t)>0\} \quad \text{ and } \quad \eta=\sup\{t>0,\ \lambda^I(t)>0\}.
				\end{equation}
				\begin{figure}[htb]
					\centering
					\includegraphics[width=\textwidth]{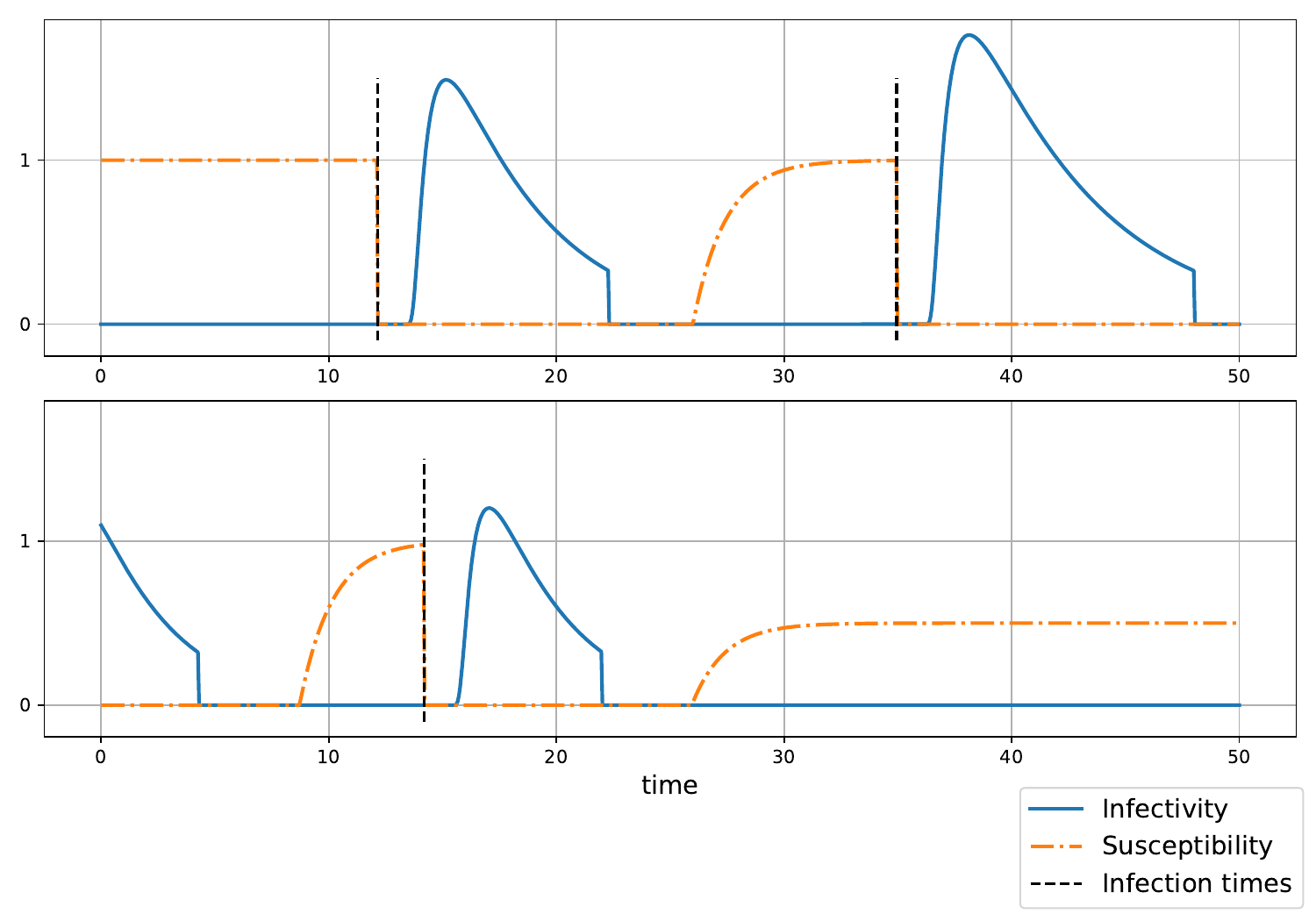}
					\caption{Illustration of the evolution of an individual's infectivity and natural susceptibility through time. Each graphic shows the dynamics of an individual's infectivity (blue) and susceptibility (orange). The top graphic corresponds to an individual which is initially susceptible, and the bottom one to an initially infectious individual. Note that, after being reinfected, the second individual remains partially immune even a long time after infection.} \label{fig:infection_cycles}
				\end{figure}
				
				\bigskip
				Let $Q^V$ be a Poisson random process on $\R_+$ with intensity a sigma-finite measure $\nu$ on $\R_+$ independent of $(\lambda_0^I,\gamma_0^I)$, $(\lambda^I,\gamma^I)$ and $\gamma^V$. Then for $t\geq0,$ we introduce $\mathcal{N}_t=Q^V((0,t])$ the Poisson process that describes the number of vaccination campaign that takes place in the time interval $(0,t]$. 
				
				Let  now $ (Q_k^V, 1 \leq k \leq N) $ be a family of i.i.d copies of $Q^V$ independent of $ \lbrace (\lambda_{k,i}^I, \gamma_{k,i}^I), i \geq 0, 1 \leq k \leq N \rbrace $ and $(\gamma^V_{k,j})_{k,j}$. We then introduce $\mathcal{N}_t^k=Q^V_k((0,t])$ and denote by $(T^V_{k,i})_i$ the associated jump times, which correspond to the vaccination times campaign of individual $k$.
				
				Let $A^N_k(t)$ be the number of times that the individual $k$ has been (re)--infected on the time interval $(0,t]$ and $(\tau^N_{k,i})_i$ the associated jump times. Literally for $i\geq1,\,\tau^N_{k,i}$ corresponds to the time of the $i$-th infection of individual $k$. 
				
				At any time $t$, the state of individual $k$ is determined by the chronology of the most recent infection $\tau_{k,A^N_k(t)}^N$ and the most recent vaccination $T^V_{k,\mathcal{N}_t^k}$. We distinguish two situations:
				\begin{enumerate}
					\item If $T^V_{k,\mathcal{N}_t^k}-\tau_{k,A^N_k(t)}^N \leq \eta_{k,A^N_k(t)}$, individual $k$ is infected at the time of vaccination. In this case, the individual is not vaccinated, and their infectivity $\lambda^I$ and susceptibility $\gamma^I$ remain unchanged.
					\item If $T^V_{k,\mathcal{N}_t^k} - \tau_{k,A^N_k(t)}^N > \eta_{k,A^N_k(t)}$,  individual $k$ is not infected at the time of vaccination. In this case, the individual is vaccinated, a new random function  $\gamma^V$ is drawn to determine their current susceptibility, and their infectivity becomes zero. 
				\end{enumerate}
				
				Consequently, the current infectivity is given by:
				\begin{equation}
					\begin{aligned}
						\lambda_{k,A^N_k(t),\mathcal{N}^k(t)}(t) &:= \lambda_{k,A^N_k(t)}^I(t-\tau_{k,A^N_k(t)}^N) \mathds{1}_{\{ T^V_{k,\mathcal{N}_t^k} - \tau_{k,A^N_k(t)}^N \leq \eta_{k,A^N_k(t)} \}},
					\end{aligned}
				\end{equation}
				and the current susceptibility is given by
				\begin{equation}
					\begin{aligned}
						\gamma_{k,A^N_k(t),\mathcal{N}^k(t)}(t) &:= \gamma_{k,A^N_k(t)}^I(t-\tau_{k,A^N_k(t)}^N) \mathds{1}_{\{ T^V_{k,\mathcal{N}_t^k} - \tau_{k,A^N_k(t)}^N \leq \eta_{k,A^N_k(t)} \}} \\
						&\hspace{2cm}+ \gamma_{k,\mathcal{N}_t^k}^V(t-T^V_{k,\mathcal{N}_t^k}) \mathds{1}_{\{ T^V_{k,\mathcal{N}_t^k} - \tau_{k,A^N_k(t)}^N > \eta_{k,A^N_k(t)} \}}.
					\end{aligned}
				\end{equation}
				
				Let us now define $\overline{\mathfrak F}^N(t)$ and $ \overline{\mathfrak S}^N(t) $ as the average infectivity and susceptibility in the population, i.e.,
				
				\begin{align*}
					\overline{\mathfrak{F}}^N(t) := \frac{1}{N} \sum_{k=1}^N \lambda_{k,A^N_k(t),\mathcal{N}^k(t)}(t), && \overline{\mathfrak{S}}^N(t):= \frac{1}{N} \sum_{k=1}^N \gamma_{k,A^N_k(t),\mathcal{N}^k(t)}(t). 
				\end{align*}
				The instantaneous rate at which the $ \ell $-th individual infects the $k$-th individual is
				\begin{align*}
					\frac{1}{N} \lambda_{\ell,A^N_\ell(t),\mathcal{N}^\ell(t)}(t) \gamma_{k,A^N_k(t),\mathcal{N}^k(t)}(t) 
				\end{align*}
				where the $ \frac{1}{N} $ factor comes from the probability that the $ k $-th individual is chosen as the target of the infectious contact.
				Summing over the index $\ell$, the instantaneous rate at which the $ k $-th individual is (re)infected is given by
				\begin{align} \label{def_Upsilon}
					\Upsilon_k^N(t) := \gamma_{k,A^N_k(t),\mathcal{N}^k(t)}(t)\, \overline{\mathfrak{F}}^N(t).
				\end{align}
				\bigskip
				
				Let  now $ (Q_k, 1 \leq k \leq N) $ be an i.i.d. family of standard Poisson random measures (PRMs) on $ \R_+^2 $,  independent of the  sequence $(\lambda_{k,i}^I,\gamma_{k,i}^I)_{1\le k\le N, i\ge0}$, $(\gamma^V_{k,j})_{1\leq k\leq N, j\geq1}$ and $(Q_k^V, 1 \leq k \leq N) $.
				The family of counting processes $(A^N_k(t), t \geq 0, 1 \leq k \leq N)$ is then defined as the solution of
				\begin{align} \label{def_A_N_k}
					A^N_k(t) = \int_{[0,t] \times \R_+} \indic{u\leq\Upsilon^N_k(r^-)}Q_k(dr,du)\,.
				\end{align}

				Note that we construct $A^N_k$ by induction on the jumps times. Assumption \ref{AS-lambda}  implies that the rate $\Upsilon_k^N(t)$ is bounded almost surely
				by the constant $\lambda_\ast$. Consequently the jump times do not accumulate, and the above induction defines $A^N_k(t)$ for all $t\geq0$.
				
				\subsection{Compartmental model}Let $\varphi^I_0\in D_+$, $\varphi^I\in D_+$ and $\varphi\in D_+$ a random functions and let $(\varphi_{k,0}^I)_k$ be a family of copies of $\varphi_0^I$, $((\varphi_{k,i}^I))_k$ be a family of copies of $\varphi^I$ and $((\varphi_{k,i}^V))_k$ be a family of copies of $\varphi^V$ independent of the family $(Q_k,Q^V_k)_k$. We assume that $\varphi^I_0$, $\varphi^I$ and $\varphi^V$ are bounded by a deterministic constant.
				
				We set 
				\begin{equation}
					\begin{aligned}
						\varphi_{k,A^N_k(t),\mathcal{N}^k(t)}(t)&:=\varphi_{k,A^N_k(t)}^I(t-\tau_{k,A^N_k(t)}^N)\indic{ T^V_{k,\mathcal{N}_t^k}-\tau_{k,A^N_k(t)}^N\leq\eta_{k,A^N_k(t)}}\\
						&\hspace{2cm}+\varphi_{k,A^N_k(t)}^V(t-T^V_{k,\mathcal{N}_t^k})\indic{T^V_{k,\mathcal{N}_t^k}-\tau_{k,A^N_k(t)}^N>\eta_{k,A^N_k(t)}},
					\end{aligned}
				\end{equation}
				and we define for $t\geq0$,
				\[\mathfrak{U}^N_t(\varphi)=\frac{1}{N}\sum_{k=1}^{N}\varphi_{k,A^N_k(t),\mathcal{N}^k(t)}(t).\]
				In particular, 
				\begin{itemize}
					\item For $\varphi_0^I(t)=\indic{\eta_0\geq t},\,\varphi^I(t)=\indic{\eta\geq t},$ and $\varphi^V(t)=0,$
					$\overline{I}^N(t):=\mathfrak{U}^N_t(\varphi)$ will be the proportion of infected individuals at time $t$.
					\item If we denote by $\varsigma_0$ and $\varsigma$ the random variable which represent  the duration  where individuals are full immunized, for $\varphi_0^I(t)=\indic{\eta_0< t\leq \eta_0+\varsigma_0}$, $\varphi^I(t)=\indic{\eta< t\leq\eta+\varsigma},$ and $\varphi^V(t)=0,$
					$\overline{R}^N(t):=\mathfrak{U}^N_t(\varphi)$ will be the proportion of recovered individuals at time $t$.
				\end{itemize}
				%
				We also refer to \cite{forien-Zotsa2022stochastic,ngoufack2025functional} for more examples of the function $\varphi$.				
				\subsection{Functional law of large numbers}
				In this Section we will present the convergence of the process $\fS^N,\,\fF^N$ and $\mathfrak{U}^N(\varphi)$.
				
				We define by $\mu^V$ the law of $\gamma^V$ and $\mu^I$ the law of the pair $(\gamma^I,\lambda^I)$.
				
				For $(x,y)\in D^2_+,$ and $v\in\R_+,$ we set,
				\begin{multline}\label{eq-Psi-0}
					\begin{aligned}
						\mathcal{K}(x,y)(v)&=1-\E\Big[\indic{ v\leq\eta_{0}}\exp\left(-\int_{0}^{v}\gamma_{0}^I(r)y(r)dr\right)\Big]\\
						&\hspace{1cm}-\int_{0}^{v}\E\Big[\indic{v-s\leq\eta}\exp\left(-\int_{s}^{v}\gamma^I(r-s)y(r)dr\right)\Big]x(s)y(s)ds
					\end{aligned}
				\end{multline}
				
				Note that, $\mathcal{K}(x,y)(v)$ is the probability that individual is successful vaccinated at time $v$.
				
				We now consider the following system of integral equations, for which we look a solution $(x,y)\in D^2_+$:
				\begin{equation}\label{eq:F-G-x}
					\begin{aligned}
						x(t)&=\E\left[\gamma_0^I(t)\indic{t\leq\eta_{0}}\exp\left(-\int_{0}^{t}\gamma_{0}^I(r)y(r)dr\right)\right]\\
						&\hspace{2cm}+\E\left[\gamma_0^I(t)e^{-\nu((\eta_{0},t])}\indic{t>\eta_{0}}\exp\left(-\int_{0}^{t}\gamma_{0}^I(r)y(r)dr\right)\right]\\
						&\hspace{1cm}+\int_{0}^{t}\E\left[\gamma^I(t-s)\indic{t-s\leq\eta}\exp\left(-\int_{s}^{t}\gamma^I(r-s)y(r)dr\right)\right]x(s)y(s)ds\\
						&\hspace{1cm}+\int_{0}^{t}\E\left[\gamma^I(t-s)e^{-\nu((s+\eta,t])}\indic{t-s>\eta}\exp\left(-\int_{s}^{t}\gamma^I(r-s)y(r)dr\right)\right]x(s)y(s)ds\\	&\hspace{1cm}+\int_0^t \mathcal{K}(x,y)(v)e^{-\nu((v,t])}\E\left[\gamma^V(t-v)\exp\left(-\int_{v}^{t}\gamma^V(r-v)y(r)dr\right)\right]\nu(dv)
					\end{aligned}
				\end{equation}
				and 
				\begin{equation}\label{eq:F-G-y}
					\begin{aligned}
						y(t)&=\E\left[\lambda_0^I(t)\indic{t\leq\eta_{0}}\exp\left(-\int_{0}^{t}\gamma_{0}^I(r)y(r)dr\right)\right]\\
						&\hspace{2cm}+\E\left[\lambda_0^I(t)e^{-\nu((\eta_{0},t])}\indic{t>\eta_{0}}\exp\left(-\int_{0}^{t}\gamma_{0}^I(r)y(r)dr\right)\right]\\
						&\hspace{1cm}+\int_{0}^{t}\E\left[\lambda^I(t-s)\indic{t-s\leq\eta}\exp\left(-\int_{s}^{t}\gamma^I(r-s)y(r)dr\right)\right]x(s)y(s)ds\\
						&\hspace{1cm}+\int_{0}^{t}\E\left[\lambda^I(t-s)e^{-\nu((s+\eta,t])}\indic{t-s>\eta}\exp\left(-\int_{s}^{t}\gamma^I(r-s)y(r)dr\right)\right]x(s)y(s)ds.							
					\end{aligned}
				\end{equation}
				We first establish the conservation of the mass. The proof is in Section~\ref{sec:mass-conserv}.
				\begin{prop}\label{Lem:mass-conserv}Under Assumption~\ref{AS-lambda}, if $(x,y)\in D^2_+$ is a solution of the system of equations~\eqref{eq:F-G-x}-\eqref{eq:F-G-y}, then,
					\begin{multline}
						\begin{aligned}
							&\E\left[\indic{t\leq\eta_{0}}\exp\left(-\int_{0}^{t}\gamma_{0}^I(r)y(r)dr\right)\right]+\E\left[e^{-\nu((\eta_0,t])}\indic{t>\eta_{0}}\exp\left(-\int_{0}^{t}\gamma_{0}^I(r)y(r)dr\right)\right]\\
							&\hspace{1cm}+\int_{0}^{t}\E\left[\indic{t-s\leq\eta}\exp\left(-\int_{s}^{t}\gamma^I(r-s)y(r)dr\right)\right]x(s)y(s)ds\\
							&\hspace{1cm}+\int_{0}^{t}\E\left[e^{-\nu((s+\eta,t])}\indic{t-s\geq\eta}\exp\left(-\int_{s}^{t}\gamma^I(r-s)y(r)dr\right)\right]x(s)y(s)ds\\
							&\hspace{1cm}+\int_{0}^{t}\E\left[e^{-\nu((\eta_0,a])}\indic{a\geq\eta_{0}}\exp\left(-\int_{0}^{a}\gamma_{0}^I(r)y(r)dr\right)\right]\nu(da)\\
							&\hspace{1cm}+\int_{0}^{t}\int_{s}^{t}\E\left[e^{-\nu((s+\eta,a])}\indic{a-s\geq\eta}\exp\left(-\int_{s}^{a}\gamma^I(r-s)y(r)dr\right)\right]\nu(da)x(s)y(s)ds\\
							&\hspace{1cm}+\int_{0}^{t}e^{-\nu((v,t])}\mathcal{K}(x,y)(v)\E\left[\exp\left(-\int_{v}^{t}\gamma^V(r-v)y(r)dr\right)\right]\nu(dv)\\
							&\hspace{1cm}+\int_{0}^{t}\mathcal{K}(x,y)(v)\int_{v}^{t}e^{-\nu((v,a])}\E\Bigg[\exp\left(-\int_{v}^{a}\gamma^V(r-v)y(r)dr\right)\Bigg]\nu(da)\nu(dv)\\
							&=1+\int_{0}^{t}\mathcal{K}(x,y)(v)\nu(dv).
						\end{aligned}
					\end{multline}
				\end{prop}
				
				In the special case where, $\gamma^V=0,$ individuals acquire permanent immunity upon successful vaccination. By noting that,
				\[e^{-\nu((v,t])}+\int_{v}^{t}e^{-\nu((v,a])}\nu(da)=1,\]
				the terms in Proposition~\ref{Lem:mass-conserv} containing $\mathcal{K}$ simplify, allowing us to recover a consistent mass balance between the infectious and vaccinated histories.
				
				The following result establishes the existence and uniqueness of the solution $(x,y)$ of \eqref{eq:F-G-x}-\eqref{eq:F-G-y}. The proof is given in Section~\ref{exist}.
				\begin{theorem}\label{Th:exists}
					Under Assumption~\ref{AS-lambda}, the set of equations \eqref{eq:F-G-x}-\eqref{eq:F-G-y} has a unique solution $(\fS,\fF)\in D^2(\R_+,\R_+).$ 
				\end{theorem}
				We now state the main result of this article: our functional law of large numbers. The proof is provided in Section~\ref{proof-FLLN}.
				\begin{theorem}\label{Th:FLLN}
					Under Assumption~\ref{AS-lambda},  
					\begin{equation*}
						\left(\fS^N,\fF^N\right)\xrightarrow[N\to\infty]{\mathbb{P}}\left(\fS,\fF\right),\text{ in }D^2,
					\end{equation*}
					where $\left(\fS,\fF\right)$ is the unique solution of the system of equation~\eqref{eq:F-G-x}-\eqref{eq:F-G-y}.
					
					Given the solution $\left(\fS,\fF\right),$
					\[\mathfrak{U}^N(\varphi)\xrightarrow[N\to\infty]{\mathbb{P}}\mathfrak{U}(\varphi),\text{ in }D\]
					where for $t\geq0,$
					\begin{align}\label{cor:compartiment}
						&\mathfrak{U}_t(\varphi)=\E\left[\varphi_{0}^I(t)\indic{t\leq\eta_0}\exp\left(-\int_{0}^{t}\gamma_{0}^I(r)\fF(r)dr\right)\right]\non\\
						&\hspace{2cm}+\E\left[\varphi_{0}^I(t)e^{-\nu((\eta_0,t])}\indic{t>\eta_0}\exp\left(-\int_{0}^{t}\gamma_{0}^I(r)\fF(r)dr\right)\right]\non\\
						&\hspace{3cm}+\int_{0}^{t}\E\left[\varphi^I(t-s)\indic{t-s\leq\eta}\exp\left(-\int_{s}^{t}\gamma^I(r-s)\fF(r)dr\right)\right]\fF(s)\fS(s)ds\non\\
						&\hspace{1.5cm}+\int_{0}^{t}\E\left[\varphi^I(t-s)e^{-\nu((\eta+s,t])}\indic{t-s>\eta}\exp\left(-\int_{s}^{t}\gamma^I(r-s)\fF(r)dr\right)\right]\fF(s)\fS(s)ds\non\\
						&\hspace{1cm}+\int_0^t\mathcal{K}(\fS,\fF)(v)e^{-\nu((v,t])}\E\left[\varphi^V(t-v)\exp\left(-\int_{v}^{t}\gamma^V(r-v)\fF(r)dr\right)\right]\nu(dv).
					\end{align}
				\end{theorem}
				We first note that as for $\varphi_0^I(t)=\indic{t\leq \eta_0}$, $\varphi^I(t)=\indic{t\leq \eta},$ and $\varphi^V(t)=0,$ the proportion of infected individuals $\bar{I}^N=\mathfrak{U}^N(\varphi)$ converges in $D$ to $\mathfrak{U}(\varphi):=\bar{I}$ which reduces to:
				\begin{equation}\label{eq-I}
					\begin{aligned}
						\bar{I}(t)&=\E\Big[\indic{ t\leq\eta_{0}}\exp\left(-\int_{0}^{t}\gamma_{0}^I(r)y(r)dr\right)\Big]\\
						&\hspace{1cm}+\int_{0}^{t}\E\Big[\indic{t-s\leq\eta}\exp\left(-\int_{s}^{t}\gamma^I(r-s)y(r)dr\right)\Big]x(s)y(s)ds,
					\end{aligned}
				\end{equation}
				for $t\geq0$. Therefore $\mathcal{K}(\fS,\fF)=1-\bar{I}.$
				We then rewrite \eqref{eq:F-G-x}-\eqref{eq:F-G-y} as,
				\begin{equation}\label{eq-inf-F-sf-S}
					\begin{aligned}
						&\fS(t)=
						\E\left[\gamma_{0}^I(t)\indic{t\leq\eta_{0}}\exp\left(-\int_{0}^{t}\gamma_{0}^I(r)\fF(r)dr\right)\right]\\
						&\hspace{2cm}+\E\left[\gamma_{0}^I(t)e^{-\nu((\eta_{0},t])}\indic{t>\eta_{0}}\exp\left(-\int_{0}^{t}\gamma_{0}^I(r)\fF(r)dr\right)\right]\\
						&\hspace{1cm}+\int_{0}^{t}\E\left[\gamma^I(t-s)\indic{t-s\leq\eta}\exp\left(-\int_{s}^{t}\gamma^I(r-s)\fF(r)dr\right)\right]\fS(s)\fF(s)ds\\
						&\hspace{1cm}+\int_{0}^{t}\E\left[\gamma^I(t-s)e^{-\nu((s+\eta,t])}\indic{t-s>\eta}\exp\left(-\int_{s}^{t}\gamma^I(r-s)\fF(r)dr\right)\right]\fS(s)\fF(s)ds\\
						&\hspace{1cm}+\int_0^t\big(1-\bar{I}(v)\big)e^{-\nu((v,t])}\E\left[\gamma^V(t-v)\exp\left(-\int_{v}^{t}\gamma^V(r-v)\fF(r)dr\right)\right]\nu(dv)							
					\end{aligned}
				\end{equation}
				and 
				\begin{equation}\label{eq-inf-F-sf-F}
					\begin{aligned}
						&\fF(t)=
						\E\left[\lambda_{0}^I(t)\indic{t\leq\eta_{0}}\exp\left(-\int_{0}^{t}\gamma_{0}^I(r)\fF(r)dr\right)\right]\\
						&\hspace{2cm}+\E\left[\lambda_{0}^I(t)e^{-\nu((\eta_{0},t])}\indic{t>\eta_{0}}\exp\left(-\int_{0}^{t}\gamma_{0}^I(r)\fF(r)dr\right)\right]\\
						&\hspace{1cm}+\int_{0}^{t}\E\left[\lambda^I(t-s)\indic{t-s\leq\eta}\exp\left(-\int_{s}^{t}\gamma^I(r-s)\fF(r)dr\right)\right]\fS(s)\fF(s)ds\\
						&\hspace{1cm}+\int_{0}^{t}\E\left[\lambda^I(t-s)e^{-\nu((s+\eta,t])}\indic{t-s>\eta}\exp\left(-\int_{s}^{t}\gamma^I(r-s)\fF(r)dr\right)\right]\fS(s)\fF(s)ds.							
					\end{aligned}
				\end{equation}
				The random terms in \eqref{eq-inf-F-sf-S} and \eqref{eq-inf-F-sf-F} are $\lambda_0^I,\,\gamma_0^I,\,\eta_0,\,\lambda^I,\,\gamma^I,$, $\eta$ and $\gamma^V$.
				
				Note that when the measure $\nu$ is equal to zero, the set of equations~\eqref{eq-inf-F-sf-S}-\eqref{eq-inf-F-sf-F} coincide with the one in ~\cite{forien-Zotsa2022stochastic}. This yields a generalization of the model introduced by Forien, Pang, Pardoux and Zotsa-Ngoufack~\cite{forien-Zotsa2022stochastic}. Consequently our model also generalizes the one in \cite{guerin2025stochastic} when memory is removed, the vaccination model in \cite{foutel2025optimal} and the classical epidemic model in~\cite{britton2019stochastic}.
				
				\begin{remark}\label{RQ-bound-phi}
					From Proposition~\ref{prop:equal-G-F} below we have for all $t\geq0,\,\fF(t)\leq\lambda_*$ and $\fS(t)\leq1.$ Since for $t\geq0,$ $\mathfrak{U}^N(\varphi)(t)\leq\max(\|\varphi^I_0\|_\infty,\|\varphi^I\|_\infty,\|\varphi^V\|_\infty),$ by convergence it follows that: \[\mathfrak{U}(\varphi)(t)\leq\max(\|\varphi^I_0\|_\infty,\|\varphi^I\|_\infty,\|\varphi^V\|_\infty).\] 
				\end{remark}				
				
				\bigskip
				We make the following Assumption introduced in \cite{forien-Zotsa2022stochastic}.
				\begin{assumption}\label{ASS:lbd-g}
					Almost surely,
					\begin{multline*}\label{eqq10}
						\sup\{t\geq0,\,\lambda_0^I(t)>0\}\leq\inf\{t\geq0,\,\gamma_0^I(t)>0\}\text{ and }\\ \sup\{t\geq0,\,\lambda^I(t)>0\}\leq\inf\{t\geq0,\,\gamma^I(t)>0\}.
					\end{multline*}
				\end{assumption}
				Assumption~\ref{ASS:lbd-g} means that, as long as an individual has no recovered, they cannot be reinfected. It also implies that the functions $\lambda^I_0$ and $\lambda^I$ have compact support that is disjoint from those of $\gamma^I_0$ and $\gamma^I$. In this case, the functional law of large numbers satisfies a different set of equations. More precisely, equations~\eqref{eq-inf-F-sf-S} and equation~\eqref{eq-inf-F-sf-F} is replaced by:
				\begin{equation}\label{eq-inf-F-sf-S-no-reinf}
					\begin{aligned}
						&\fS(t)=
						\E\left[\gamma_{0}^I(t)e^{-\nu((\eta_{0},t])}\indic{t>\eta_{0}}\exp\left(-\int_{0}^{t}\gamma_{0}^I(r)\fF(r)dr\right)\right]\\
						&\hspace{1cm}+\int_{0}^{t}\E\left[\gamma^I(t-s)e^{-\nu((s+\eta,t])}\indic{t-s>\eta}\exp\left(-\int_{s}^{t}\gamma^I(r-s)\fF(r)dr\right)\right]\fS(s)\fF(s)ds\\
						&\hspace{1cm}+\int_0^t(1-\bar{I}(v))e^{-\nu((v,t])}\E\left[\gamma^V(t-v)\exp\left(-\int_{v}^{t}\gamma^V(r-v)\fF(r)dr\right)\right]\nu(dv)\end{aligned}
				\end{equation}
				and 
				\begin{equation}\label{eq-inf-F-sf-F-no-reinf}
					\begin{aligned}
						\fF(t)&=\E\left[\lambda_{0}^I(t)\right]+\int_{0}^{t}\E\left[\lambda^I(t-s)\right]\fS(s)\fF(s)ds\\
						&=\overline{I}(0)\overline{\lambda}_0(t)+\int_{0}^{t}\overline{\lambda}(t-s)\fF(s)\fS(s)ds,
					\end{aligned}
				\end{equation}
				where
				\[\bar{I}(0)=\mathbb{P}(\eta_{0}>0),\quad\bar{\lambda}_0(t)=\E\left[\lambda_{0}^I(t)\big|\eta_{0}>0\right]\text{ and }\bar{\lambda}(t)=\E\left[\lambda^I(t)\right].\]
				We note that, in this case, the total force of infection given by equation~\eqref{eq-inf-F-sf-F-no-reinf} is identical to that of the model without periodic vaccination described in \cite{forien-Zotsa2022stochastic}.		
				
				We also note that, under Assumption~\ref{ASS:lbd-g}, equation~\eqref{eq-I} is replaced too by:
				\begin{equation}\label{eq-I-A}
					\bar{I}(t)=\bar{I}(0)F^c_{I,0}(t)+\int_{0}^{t}F^c_I(t-s)\fF(s)\fS(s)ds,
				\end{equation}
				where 
				\[ F^c_{I,0}(t)=\mathbb{P}(\eta_0>t\big|\eta_0>0),\quad F^c_I(t)=\mathbb{P}(\eta>t).\]
				\subsection{A simplified version of the FLLN and its connection with the literature.}\label{sec-FLLN-simply}
				In this section, we present a particular case of Theorem~\ref{Th:FLLN}, referred to as the non-Markovian SIRS model with pulse vaccination. In this specific scenario, an individual becomes immediately susceptible again upon recovering after their immune period. From this result, we recover the classical epidemic model with pulse vaccination. 
				
				We introduce $\varsigma_0$ and $\varsigma$ the random variables representing the duration during which individual is full immunized. We also introduce $\eta^V$ as the random variable representing the duration for which the vaccine provides full immunity.
				We then introduce the following quantities:
				\[V(0)=0,\quad\overline{S}(0)=\mathbb{P}(\eta_0=0,\varsigma_0=0),\quad\overline{I}(0)=\mathbb{P}(\eta_0>0), \text{ and }\bar{R}(0)=\mathbb{P}(\eta_0=0,\varsigma_0>0).\]
				In this specific SIRS model, where the vaccine does not provide permanent immunity, the functions $\lambda^I_0,\, \lambda^I,\,\gamma_{0}^I,\,\gamma^I$ and $\gamma^V$ are defined as follows:
				\[\lambda_0^I(t)=\beta\indic{\eta_0\geq t},\quad\lambda^I(t)=\beta\indic{\eta\geq t},\quad\gamma^I_0(t)=\indic{t>\eta_{0}+\varsigma_0},\quad\gamma^I(t)=\indic{t>\eta+\varsigma}\text{ and }\gamma^V(t)=\indic{t\geq\eta^V}.\]
				We set,
				\[ F^c_{I,0}(t)=\mathbb{P}(\eta_0>t\big|\eta_0>0),\quad F^c_I(t)=\mathbb{P}(\eta>t),\quad F^c_V(t)=\mathbb{P}(\eta^V>t),\]
				\[ G^c_0(t|u)=\mathbb{P}(\varsigma_0>t\big|\varsigma_0>0,\eta_0=u),\qquad G^c(t|u)=\mathbb{P}(\varsigma>t\big|\eta=u).\]
				We recall that $\bar{S}(t)$ is the proportion of susceptible individuals at time $t$, $\bar{R}(t)$ for recovered individuals and $V(t)$ denotes the proportion of vaccinated individuals who remain fully immunized. Consequently,
				\begin{itemize}
					\item for $\varphi^I_0(t)=\indic{\eta_{0}+\varsigma_0\leq t},$ $\varphi^I(t)=\indic{\eta+\varsigma\leq t},$ and $\varphi^V(t)=\indic{\eta^V\leq t},$ we have $\bar{S}(t)=\mathfrak{U}_t(\varphi)=\fS(t)$ 
					\item for $\varphi_0^I(t)=\indic{\eta_{0}\leq t<\eta_{0}+\varsigma_0}$, $\varphi^I(t)=\indic{\eta\leq t<\eta+\varsigma}$, and $\varphi^V(t)=0,$ we have $\bar{R}(t)=\mathfrak{U}_t(\varphi).$
					\item for $\varphi^I_0(t)=\varphi^I(t)=0$ and $\varphi^V(t)=\indic{t\leq\eta^V},$ we have $V(t)=\mathfrak{U}_t(\varphi).$
				\end{itemize}
				We first note that, from equation~\eqref{eq-inf-F-sf-F-no-reinf} and \eqref{eq-I-A} that, $\fF(t)=\beta\bar{I}(t).$
				From equation~\eqref{cor:compartiment}, we then derive following Proposition.
				\begin{prop}\label{prop-FLLN}
					\begin{equation}\label{Pop-eq-S}
						\begin{aligned}
							&\bar{S}(t)=
							\bar{S}(0)e^{-\nu((0,t])}\exp\left(-\beta\int_{0}^{t}\bar{I}(r)dr\right)+\bar{R}(0)\int_{0}^{t}\exp\left(-\beta\int_{y}^{t}\bar{I}(r)dr\right)G_0^c(dy|0)e^{-\nu((0,t])}\\
							&\hspace{2cm}+\bar{I}(0)\int_{0}^{t}\int_0^{t-x} e^{-\nu((x,t])}\exp\left(-\beta\int_{x+u}^{t}\bar{I}(r)dr\right)G_0^c(du|x)F^c_{I,0}(dx)\\
							&\hspace{1cm}+\beta\int_{0}^{t}\int_0^{t-s}\int_{0}^{t-s-x} e^{-\nu((s+x,t])}\exp\left(-\beta\int_{s+x+y}^{t}\bar{I}(r)dr\right)G^c(dy|x)\bar{S}(s)\bar{I}(s)F^c_I(dx)ds\\
							&\hspace{1cm}+\int_0^t\int_{0}^{t-v}\big(1-\bar{I}(v)\big)e^{-\nu((v,t])}\exp\left(-\beta\int_{v+x}^{t}\bar{I}(r)dr\right) F^c_V(dx)\nu(dv),							
						\end{aligned}
					\end{equation}
					\begin{equation}\label{Pop-eq-I}
						\bar{I}(t)=\bar{I}(0)F^c_{I,0}(t)+\int_{0}^{t}F^c_I(t-s)\fF(s)\fS(s)ds,
					\end{equation}
					\begin{equation}\label{Pop-eq-R}
						\begin{aligned}
							&\bar{R}(t)=
							\bar{R}(0)G_0^c(t|0)e^{-\nu((0,t])}+\bar{I}(0)\int_0^t e^{-\nu((x,t])}G^c_0(t-x|x)F_{I,0}^c(dx)\\
							&\hspace{1cm}+\beta\int_{0}^{t}\int_0^{t-s} e^{-\nu((s+x,t])}G^c(t-s-x|x)\bar{S}(s)\bar{I}(s)F^c_I(dx)ds,							
						\end{aligned}
					\end{equation}
					and 
					\begin{equation}\label{Pop-eq-V}
						V(t)=\int_0^t\big(1-\bar{I}(v)\big)e^{-\nu((v,t])}F^c_V(t-v)\nu(dv).
					\end{equation}
				\end{prop}
				\begin{example}\label{Ex-Model-SIRS}
					We consider a scenario where an epidemic occurs, and subsequently, public health authorities develop a vaccine that provides partial (non-permanent) immunity. Typically, at the onset of a campaign, vaccines are administered continuously; however, after some time, health officials may transition to discrete vaccination events to maintain the disease at an endemic level. To model this, we define the vaccination measure $\nu$ as follows. For a sequence of deterministic time $(\tau_n)_n$ and the sequence of non-negative real numbers $(\alpha_n)_n$, we let, 
					\[\nu_1(dv)=w(v)dv,\qquad \nu_2(dv)=\sum_{n} \alpha_n \delta_{\tau_n}(dv)\]
					and define the measure $\nu$  as follows
					\begin{equation}
						\nu(dv) = \nu_1(dv) + \nu_2(dv).
					\end{equation}
					Consequently, for $t\in (\tau_n,\,\tau_{n+1}),\,\nu((v,t])=\int_{v}^{t}w(v)dv+\sum_{k=1}^{n}\alpha_k-\nu_2([0,v]).$
					We then derive that $\partial_t\nu((v,t])=w(t).$
					
					We now consider that: \[G^c_0(t|x)=G^c(t|x)=e^{-\theta t},\quad F_{I,0}^c(x)=F_I^c(x)=e^{-\sigma x},\quad F_{V,}^c(t)=e^{-\kappa t}.\]
					We then derive that, system of equation~\eqref{Pop-eq-S},~\eqref{Pop-eq-I},~\eqref{Pop-eq-R} and~\eqref{Pop-eq-V} reduces to the following ODE model:
					for $t \in (\tau_n, \tau_{n+1}[$, the population proportions follow:
					\begin{equation}\label{cl-ode}
						\left\{\begin{aligned}
							&\frac{d \bar{S}}{dt}(t) = -\beta \bar{S}(t) \bar{I}(t)  + \theta \bar{R}(t) + \kappa V(t)- w(t)\bar{S}(t) \\
							&\frac{d \bar{I}}{dt}(t) = \beta \bar{S}(t) \bar{I}(t) - \sigma \bar{I}(t) \\
							&\frac{d \bar{R}}{dt}(t) = \sigma \bar{I}(t) - (\theta + w(t))\bar{R}(t)\\
							&\frac{dV}{dt}(t) = (1-\bar{I}(t))w(t) - (\kappa + w(t))V(t)\\
							&\bar{S}(t)+\bar{I}(t)+\bar{R}(t)+V(t)=1.
						\end{aligned}\right.
					\end{equation}

					\begin{equation}
						\left\{\begin{aligned}
							\bar{S}(\tau_n^+) &= \bar{S}(\tau_n^-) e^{-\alpha_n} \\
							\bar{R}(\tau_n^+) &= \bar{R}(\tau_n^-) e^{-\alpha_n} \\
							V(\tau_n^+) &= V(\tau_n^-) + \left( \bar{S}(\tau_n^-) + \bar{R}(\tau_n^-) \right)(1 - e^{-\alpha_n})
						\end{aligned}\right.
					\end{equation}
					Detailed settings for this model are provided in \ref{se-set-example}.
				\end{example}
				
				\begin{remark}
					Note that when $w\equiv0$ our SIRS ODE model reduces to a pulse vaccination model with discrete periodical vaccine. Conversely, when $\kappa=0,$ we obtain a model with permanent vaccine. A similar hybrid approach combining continuous vaccination with discrete periodic pulses was studied in \cite{bolatova2024mathematical}.
					In the specific case where  $\kappa=0$ and $w\equiv0$, we retrieve the classical Markovian deterministic ODE SIRS epidemic model with pulse vaccination (see \cite{ma2024stability}) and, subsequently, the classical SIR models studied in \cite{gao2006analysis, gao2007analysis, jin2008pulse, stone2000theoretical}.
				\end{remark}
				
				%
				%
				%
				\section{Long time behavior}\label{sec-endemic-eq}
				
				In this section we will focus on studying the longtime behavior of the system equations~\eqref{eq-inf-F-sf-S}-\eqref{eq-inf-F-sf-F} given by the FLLN in Theorem~\ref{Th:FLLN} under Assumption~\ref{AS-lambda} and~\ref{ASS:lbd-g}. We start to recall the basic reproduction number \[R_0=\int_{0}^{\infty}\bar{\lambda}(s)ds,\]
				which represents the average number of individuals infected by an infectious individual in a fully susceptible population.  
				
				We make the following Assumption on the random infectivity functions.
				\begin{assumption}\label{ass-lbda-gamma}
					\leavevmode
					\begin{enumerate}[label=\textit{(\roman*)},ref=\textit{(\roman*)}]
						\item\label{ass-lbda-gamma-i} The randoms variables $\eta_0$ and $\eta$ defined in \eqref{def:eta} are integrables.
						\item \label{ass-lbda-gamma-ii} The map $t\mapsto\E\left[\lambda_0^I(t)\right]$ is integrable on $\R_+$.
					\end{enumerate}
				\end{assumption}
				Under Assumption~\ref{AS-lambda},~\ref{ASS:lbd-g} and~\ref{ass-lbda-gamma}\ref{ass-lbda-gamma-i} , as $\lambda^I_0(t)\leq\lambda_*$ and $\lambda^I(t)\leq\lambda_*$, it follows that, 
				\[\lim_{t\to\infty}\E\left[\lambda^I_0(t)\right]=0\text{ and }\lim_{t\to\infty}\E\left[\lambda^I(t)\right]=0.\]
				We set
				\[\gamma_*^I=\sup_{t\geq0}\gamma^I(t)\leq1,\quad\gamma_*^V=\sup_{t\geq0}\gamma^V(t)\leq1.\]
				We first establish the following result in Section~\ref{sec-LT-th-1}, which demonstrates that the disease-free steady state is globally asymptotically stable.
				\begin{theorem}\label{Th-LT-free}
					Under Assumption~\ref{AS-lambda},~\ref{ASS:lbd-g} and ~\ref{ass-lbda-gamma},
					\begin{enumerate}[label=\textit{(\roman*)},ref=\textit{(\roman*)}]
						\item \label{Th-lbda-gamma-i} If $R_0<1,$ then $\fF(t)\to0$ as $t\to\infty$.
						\item\label{Th-lbda-gamma-ii} Moreover, if $R_0<\E\left[\frac{1}{\gamma_*^I}\right],$ and $\gamma^V\equiv0,$ then $\fF(t)\to0$ as $t\to\infty.$
					\end{enumerate}
				\end{theorem}
				While Theorem~\ref{Th-LT-free}~\ref{Th-lbda-gamma-ii} suggests that the constant one is not necessarily optimal, Example~\ref{Ex-long} demonstrates that this threshold is indeed optimal for some non-Markovian models, and in particular, for the classical ODE model.
				\begin{remark}\label{Rk-LT}
					From the proof of Theorem~\ref{Th-LT-free} in Section~\ref{sec-LT-th-1}, we note that, if $R_0<1$ (or $R_0<\E\left[\frac{1}{\gamma_*^I}\right],$ and $\gamma^V\equiv0$), we have
					\[\int_{0}^{\infty}\fF(s)ds<\infty.\]
				\end{remark}
				We now define an infection-free solution as one for which $\overline{I}(t)=0$ for all $t\geq0,$ which is equivalent to $\forall t\geq0,\,\fF(t)=0.$ In this case, the average susceptibility is given by:
				\begin{equation}\label{eq-inf-F-sf-free}
					\fS_e(t)=\E\left[\gamma_{0}^I(t)e^{-\nu((\eta_{0},t])}\indic{t>\eta_{0}}\right]+\int_{0}^{t}e^{-\nu((v,t])}\E\left[\gamma^V(t-v)\right]\nu(dv)						
				\end{equation}
				
				By Dominated convergence Theorem, Remark~\ref{Rk-LT} and expression~\eqref{eq-inf-F-sf-S-no-reinf}, we easily derive the following result:
				\begin{coro}\label{coro-endemic-Th-d-free}
					Assume that Assumptions~\ref{AS-lambda}, \ref{ASS:lbd-g}, and \ref{ass-lbda-gamma} hold.
					\begin{enumerate}
						\item \label{cor-V-p} If $\nu([0, t]) \to \infty$ as $t \to \infty$, $\gamma^V \equiv 0$, and either $R_0 < 1$ or $R_0 < \mathbb{E}[1/\gamma_*^I]$, then $\fS(t) \to 0$ as $t \to \infty$.
						\item \label{cor-V-Up} If $\nu(\mathbb{R}_+) < \infty$, $R_0 < 1$, and both $\mathbb{E}[\gamma^I_0(t)] \to 0$ and $\mathbb{E}[\gamma^V(t)] \to 0$ as $t \to \infty$, then $\fS(t) \to 0$ as $t \to \infty$.
					\end{enumerate}
				\end{coro}
				This result implies that, asymptotically, the solution $(\fS(t), \fF(t))$ converges to the disease-free trajectory $(\fS_e(t), 0)$. This convergence is established by Corollary~\ref{coro-endemic-Th-d-free}, specifically under the case where $\gamma^V\equiv0$ (Part~\ref{cor-V-p}) and the case where $\gamma^V$ is not identical null (Part~\ref{cor-V-Up}). In this case, the pair $(\fS_e(t), 0)$ is referred to in the literature as a globally attractive solution. This corollary indicates that, asymptotically, the average susceptibility tends to zero; as a result, the population becomes fully immunized.

				\begin{remark}
					This result generalizes the one in~\cite{gao2007analysis}, where the authors show that, for the classical ODE model with pulse vaccination~\eqref{cl-ode}, the disease-free solution is globally attractive when $R_0=\beta\E\left[\eta\right]<1.$ 
				\end{remark}
				
				We now turn to the study of the persistence of the disease. To this end, we adopt the following assumptions from~\cite{forien-Zotsa2022stochastic}.
				\begin{assumption}\label{ASS2}
					\leavevmode
					\begin{itemize}
						\item We set $\gamma_\ast=\min(\gamma_*^I,\gamma^V_*)$ and assume that $\gamma_*$ is deterministic and positive.
						\item The functions $\gamma^I_0$, $\gamma^I$ and $\gamma^V$ are non-decreasing.
					\end{itemize} 
				\end{assumption}
				Note since the functions $\gamma^I_0$, $\gamma^I$ and $\gamma^V$ are non-decreasing, for any $\delta \in (0,1)$,  there exists a deterministic $t_\delta>0$  such that 
				\begin{equation}
					\gamma_0^I(t_\delta)\wedge\gamma^I(t_\delta)\wedge\gamma^V(t_\delta)\geq(1-\delta)\gamma_\ast\text{ almost surely.}
				\end{equation}
				\begin{assumption}\label{ASS3}
					There exists a positive decreasing function $h$ such that $h(0)=1$ and for all $s,t\in\R_+,\;\overline{\lambda}(s+t)\geq h(s)\overline{\lambda}(t)$. The same holds for $\overline{\lambda}_0$. In addition,  $\overline{\lambda}_0$ is continuous and $\overline{\lambda}$ is of bounded total variation.
				\end{assumption}
				\begin{theorem}\label{end-th-persistence}
					Under Assumption~\ref{AS-lambda},~\ref{ASS:lbd-g},~\ref{ass-lbda-gamma},~\ref{ASS2},~\ref{ASS3}, if $R_0>\E\left[\frac{1}{\gamma_*}\right],\,\fF(0)>0,$ then there exists $c>0$ such that,
					\[\lim_{t\to\infty}\inf\fF(t)>c.\]
				\end{theorem}
				We refer to section~\ref{sec-end-th-persistence} for the proof.
				
				We note that Theorem~\ref{end-th-persistence} still holds when $\gamma^V\equiv0$. The same proof work by using the non-decreasing properties of $\gamma^I_0$ and $\gamma^I$, which ensures that for any $\delta\in(0,1)$ there exists $t_\delta$ such that
				\begin{equation*}
					\gamma_0^I(t_\delta)\wedge\gamma^I(t_\delta)\geq(1-\delta)\gamma_\ast\text{ almost surely.}
				\end{equation*} 
				
				\begin{remark}
					Assumption~\ref{ASS2} ensures that after some time the average susceptibility returns above $\frac{1}{R_0}$ if there are not too many re-infections and Assumption~\ref{ASS3} ensures that the force of infection does not decrease too rapidly.
				\end{remark}
				\begin{example}\label{Ex-long}
					For the non-Markovian model presented in Section~\ref{sec-FLLN-simply}, we have $\gamma^V_* = \gamma^I_* = 1$ and $R_0=\beta\int_0^\infty F^c_I(s)ds$. We then deduce from Theorem~\ref{Th-LT-free} that when $R_0 < 1$, the disease-free steady state is globally asymptotically stable. Moreover, since the functions $\gamma^I_0$, $\gamma^I$, and $\gamma^V$ are non-decreasing, the function $h$ in Assumption~\ref{ASS3} is equal to one, the cumulative function $F_I^c$ has bounded variation, and assuming that $\eta_0$ admits a density or that $F^c_{I,0}$ is continuous, we derive from Theorem~\ref{end-th-persistence} that, for the SIRS model introduced in Section~\ref{sec-FLLN-simply}, the disease persists whenever $R_0 > 1$. Specifically, in the case of our ODE model presented in Example~\ref{Ex-Model-SIRS}, if $R_0 < 1$, the disease-free steady state is asymptotically stable, whereas the disease persists if $R_0 > 1$.
				\end{example}

				\section{Results on stochastic integrals with respect to Poisson measure}\label{sec-poiss-integral}
				We introduce the following easier result.
				\begin{lemma}\label{Law-Tr}
					Let $Q$ be a Poisson random measure on $\R_+$ of intensity $\nu(du)$. For all $r\geq0,$ 
						We set 
						\[\mathbb{T}_r=\inf\{r< s\leq t,\,Q((r,s])=1\}.\]
						Then the Cumulative distribution function $\mathbb{F}_r$ of $\mathbb{T}_r$ is given by: 
						\begin{equation*}
							\mathbb{F}_r(s)=\left\{\begin{aligned}
								&0&\mbox{ if }s\leq r\\&1-\exp\left(-\nu((r,s])\right)&\mbox{ if }r<s\leq t\\&1&\mbox{ if }s> t.
							\end{aligned}\right.
						\end{equation*}
				\end{lemma}
				\begin{remark}
					Note that, if 
					\[H(x)=\indic{\R_+}(x),\text{ and }\delta(a)=\begin{cases}
						0&\mbox{ if }a\neq0\\1&\mbox{ if }a=0.
					\end{cases}\]
					We know that $H'(s)=\delta(s)$ and we derive that,
					\begin{equation*}
						\mathbb{F}_r(s)=(1-e^{-\nu((r,s])})(H(t-s)-H(r-s))+H(s-t)
					\end{equation*}
					Consequently,
					\begin{equation*}
						\begin{aligned}
							\mathbb{F}_r(ds)&=e^{-\nu((r,s])}(H(t-s)-H(r-s))\nu(ds)\\
							&\hspace{2cm}+(1-e^{-\nu((r,s])})(-\delta(t-s)+\delta(r-s))ds+\delta(s-t)ds.
						\end{aligned}
					\end{equation*}
					
					Then, for any differentiable function $\Psi:\R_+\to\R,$
					\begin{align*}
						\int_{\tau}^{t}\Psi(s)\mathbb{F}_\tau(ds)&=\int_{\tau}^{t}\Psi(s)e^{-\nu((\tau,s])}\nu(ds)+\Psi(t)e^{-\nu((\tau,t])}\\
						&=\Psi(\tau)+\int_{\tau}^{t}\Psi'(s)e^{-\nu((\tau,s])}ds.
					\end{align*}
					%
				\end{remark}

				We now establish the following two theorems, which provide a formula for computing the moment exponents of Poisson stochastic integrals. These theorems offer a more practical formula than those presented in \cite{breton2014factorial, decreusefond2014moment, privault2012moments}.
				\begin{theorem}\label{Lem-T0}
					Let $Q$ be a Poisson random measure on $\R_+$ of intensity $\nu(du)$. Let $\varphi:\R_+\to\R$ be a predictable process independent of $Q$,
					then for $\tau\geq0,$
					\begin{equation*}
						\E\left[\exp\left(-\int_{\tau}^{t} \varphi(r)\indic{Q((\tau,r])=0}dr\right)\right]=\E\left[\int_{\tau}^{t}\exp\left(-\int_{\tau}^{u}\varphi(r)dr\right)\mathbb{F}_\tau(du)\right].
					\end{equation*}
				\end{theorem}
				\begin{proof}
					From Lemma~\ref{Law-Tr}, we have,
					\begin{align*}
						\E\left[\exp\left(-\int_{\tau}^{t}\varphi(r)\indic{Q((\tau,r])=0}dr\right)\right]&=\E\left[\exp\left(-\int_{\tau}^{\mathbb{T}_\tau}\varphi(r)dr\right)\right]\\
						&=\E\left[\E\left[\exp\left(-\int_{\tau}^{\mathbb{T}_\tau}\varphi(r)dr\right)\big|\varphi\right]\right]\\
						&=\E\left[\int_{\tau}^{t}\exp\left(-\int_{\tau}^{u}\varphi(r)dr\right)\mathbb{F}_\tau(du)\right].
					\end{align*}
				\end{proof}
				The following theorem give the characteristic function for Poisson integral with non-predictable integrand.
				\begin{theorem}\label{Th-char}
					Let $Q$ be a Poisson random measure on $\R_+$ of intensity $\nu(du)$. Let $h:\R_+\times\R_+\to\R,$ be a predictable function. Then
					\begin{equation*}
						\E\left[\exp\left(-\int_{0}^{t}\int_{r}^{t}\indic{Q((r,s])=0}h(r,s)dsQ(dr)\right)\right]=\E\left[\exp\left(-\int_{0}^{t}\int_{r}^{t}(1-e^{-\int_{r}^{u}h(r,s)ds})\mathbb{F}_r(du)\nu(dr)\right)\right].
					\end{equation*}
				\end{theorem}
				\begin{proof}
					Let $(r_j)$ the jump times of Poisson random measure $Q$. In this case, $\mathbb{T}_{r_j}$ is the elapsed time until the next jump of $Q$ and the next jump occurs according to the intensity measure $\mathbb{F}_{r_j}(du)$ (see Lemma~\ref{Law-Tr}). Thus
					we introduce $\widetilde{Q}$ a standard Poisson random measure of intensity $\nu(dr)\times\mathbb{F}_r(du).$ We note that the family $(r_j,\mathbb{T}_{r_j})_j$ are the points of the measure $\widetilde{Q}$.
					
					Consequently we have, 
					\begin{align*}
						\int_{0}^{t}\int_{r}^{t}\indic{Q([r,s])=0}h(r,s)dsQ(dr)&=\int_{0}^{t}\int_{r}^{\mathbb{T}_r}h(r,s)dsQ(dr)\\
						&=\sum_{j\geq1}\indic{r_j\leq t}\int_{r_j}^{\mathbb{T}_{r_j}}h(r_j,s)ds\\
						&\stackrel{(d)}{=}\int_{[0,t]\times\R_+}\indic{r\leq u\leq t}\int_{r}^{u}h(r,s)ds\widetilde{Q}(dr,du)\\
					\end{align*}

					As a result, from classical result on Poisson random measure~\cite[Theorem~$2.9$, page~$252$]{ccinlar2011probability}
					\begin{align*}
						\E\left[\exp\left(-\int_{0}^{t}\int_{r}^{t}\indic{Q((r,s])=0}h(r,s)dsQ(dr)\right)\right]&=\E\left[\exp\left(-\int_{[0,t]\times\R_+}\indic{r\leq u\leq t}\int_{r}^{u}h(r,s)ds\widetilde{Q}(dr,du)\right)\right]\\
						&=\E\left[\exp\left(-\int_0^t\int_{r}^t[1-e^{-\int_{r}^{u}h(r,s)ds}]\mathbb{F}_r(du)\nu(dr)\right)\right].
					\end{align*}
				\end{proof}
				From Theorem~\ref{Lem-T0} and Theorem~\ref{Th-char} we can easily derive \cite[Theorem~$5.1$]{ngoufack2025functional} that gives the mean of stochastic integral with respect to Poisson measure and with such non-predictable integrand.

				\section{Proof for the FLLN: Theorem~\ref{Th:FLLN}}\label{proof-FLLN}
				\subsection{Proof of Proposition~\ref{Lem:mass-conserv}}\label{sec:mass-conserv}
				We can note that, if we multiply the equation for $x(t)$ by $y(t)$, we obtain
				\begin{equation*}
					\begin{aligned}
						y(t)x(t)&=\E\left[\gamma_0^I(t)y(t)\indic{t\leq\eta_{0}}\exp\left(-\int_{0}^{t}\gamma_{0}^I(r)y(r)dr\right)\right]\\
						&\hspace{1cm}+\E\left[\gamma_0^I(t)y(t)e^{-\nu((\eta_{0},t])}\indic{t>\eta_{0}}\exp\left(-\int_{0}^{t}\gamma_{0}^I(r)y(r)dr\right)\right]\\
						&\hspace{1cm}+\int_{0}^{t}\E\left[y(t)\gamma^I(t-s)\indic{t-s\leq\eta}\exp\left(-\int_{s}^{t}\gamma^I(r-s)y(r)dr\right)\right]x(s)y(s)ds\\
						&\hspace{1cm}+\int_{0}^{t}\E\left[\gamma^I(t-s)y(t)e^{-\nu((s+\eta,t])}\indic{t-s>\eta}\exp\left(-\int_{s}^{t}\gamma^I(r-s)y(r)dr\right)\right]x(s)y(s)ds\\&+\int_0^t\mathcal{K}(x,y)(v)e^{-\nu((v,t])}\E\left[\gamma^V(t-v)y(t)\exp\left(-\int_{v}^{t}\gamma^V(r-v)y(r)dr\right)\right]\nu(dv).							
					\end{aligned}
				\end{equation*}
				Consequently, for $t\geq0,$  by Fubini theorem, we have,
				\begin{align*}
					&\int_{0}^{t}x(s)y(s)ds
					=\int_{0}^{t}\E\left[\gamma_0^I(a)y(a)\indic{a\leq\eta_{0}}\exp\left(-\int_{0}^{a}\gamma_{0}^I(r)y(r)dr\right)\right]da\\
					&\hspace{1cm}+\int_{0}^{t}\E\left[\gamma_0^I(a)y(a)e^{-\nu((\eta_{0},a])}\indic{a>\eta_{0}}\exp\left(-\int_{0}^{a}\gamma_{0}^I(r)y(r)dr\right)\right]da\\
					&\hspace{1cm}+\int_{0}^{t}\int_{s}^{t}\E\left[y(a)\gamma^I(a-s)\indic{a-s\leq\eta}\exp\left(-\int_{s}^{a}\gamma^I(r-s)y(r)dr\right)\right]x(s)y(s)dads\\
					&\hspace{0.5cm}+\int_{0}^{t}\int_{s}^{t}\E\left[\gamma^I(a-s)y(a)e^{-\nu((s+\eta,a])}\indic{a-s>\eta}\exp\left(-\int_{s}^{a}\gamma^I(r-s)y(r)dr\right)\right]x(s)y(s)dads
					\\&+\int_{0}^{t}\int_v^t\mathcal{K}(x,y)(v)e^{-\nu((v,a])}\E\left[\gamma^V(a-v)y(a)\exp\left(-\int_{v}^{a}\gamma^V(r-v)y(r)dr\right)\right]da\nu(dv)
				\end{align*} 
				However,
				\begin{align}\label{eq-cons-2}
					&\int_{0}^{t}\E\left[\gamma_0^I(a)y(a)\indic{a\leq\eta_{0}}\exp\left(-\int_{0}^{a}\gamma_{0}^I(r)y(r)dr\right)\right]da\non\\
					&=\E\left[\indic{t\leq\eta_0}\left\{1-\exp\left(-\int_{0}^{t}\gamma_{0}^I(r)y(r)dr\right)\right\}\right]
					+\E\left[\indic{t>\eta_0}\left\{1-\exp\left(-\int_{0}^{\eta_0}\gamma_{0}^I(r)y(r)dr\right)\right\}\right].
				\end{align}
				and using integration by parts,
				\begin{align}\label{eq-cons-3}
					&\int_{0}^{t}\E\left[\gamma_0^I(a)y(a)e^{-\nu((\eta_{0},a])}\indic{a>\eta_{0}}\exp\left(-\int_{0}^{a}\gamma_{0}^I(r)y(r)dr\right)\right]da\non\\
					&=\E\left[\indic{t>\eta_0}\exp\left(-\int_{0}^{\eta_0}\gamma_{0}^I(r)y(r)dr\right)\right]-\E\left[e^{-\nu((\eta_{0},t])}\indic{t>\eta_0}\exp\left(-\int_{0}^{t}\gamma_{0}^I(r)y(r)dr\right)\right]\non\\
					&\hspace{1cm}-\int_{0}^{t}\E\left[e^{-\nu((\eta_{0},a])}\indic{a>\eta_{0}}\exp\left(-\int_{0}^{a}\gamma_{0}^I(r)y(r)dr\right)\right]\nu(da)
				\end{align}
				Similarly,
				\begin{align}\label{eq-cons-4}
					&\int_{0}^{t}\int_{s}^{t}\E\left[y(a)\gamma^I(a-s)\indic{a-s\leq\eta}\exp\left(-\int_{s}^{a}\gamma^I(r-s)\fF(r)dr\right)\right]x(s)y(s)dads\non\\
					&=\int_{0}^{t}\E\left[\indic{t-s\leq\eta}\right]x(s)y(s)ds-\int_{0}^{t}\E\left[\indic{t-s\leq\eta}\exp\left(-\int_{s}^{t}\gamma^I(r-s)\fF(r)dr\right)\right]x(s)y(s)ds\non\\
					&+\int_{0}^{t}\E\left[\indic{t-s>\eta}\right]x(s)y(s)ds-\int_{0}^{t}\E\left[\indic{t-s>\eta}\exp\left(-\int_{s}^{s+\eta}\gamma^I(r-s)\fF(r)dr\right)\right]x(s)y(s)ds,
				\end{align}
				\begin{align}\label{eq-cons-5}
					&\int_{0}^{t}\int_{s}^{t}\E\left[\gamma^I(a-s)y(a)e^{-\nu((s+\eta,a])}\indic{a-s>\eta}\exp\left(-\int_{s}^{a}\gamma^I(r-s)y(r)dr\right)\right]x(s)y(s)dads\non\\
					&=\int_{0}^{t}\E\left[\indic{t-s>\eta}\exp\left(-\int_{s}^{s+\eta}\gamma^I(r-s)y(r)dr\right)\right]x(s)y(s)ds\non\\
					&\hspace{1cm}-\int_{0}^{t}\E\left[e^{-\nu((s+\eta,t])}\indic{t-s>\eta}\exp\left(-\int_{s}^{t}\gamma^I(r-s)y(r)dr\right)\right]x(s)y(s)ds\non\\
					&\hspace{1cm}-\int_{0}^{t}\int_{s}^{t}\E\left[e^{-\nu((s+\eta,a])}\indic{a-s>\eta}\exp\left(-\int_{s}^{a}\gamma^I(r-s)y(r)dr\right)\right]x(s)y(s)\nu(da)ds.
				\end{align}
				and 
				\begin{align}\label{eq-cons-6}
					&\int_{0}^{t}\int_v^t\mathcal{K}(x,y)(v)e^{-\nu((v,a])}\E\left[\gamma^V(a-v)y(a)\exp\left(-\int_{v}^{a}\gamma^V(r-v)y(r)dr\right)\right]da\nu(dv)\non\\
					&=\int_0^t\mathcal{K}(x,y)(v)\nu(dv)-\int_{0}^{t}\mathcal{K}(x,y)(v)e^{-\nu((v,t])}\E\Big[\exp\left(-\int_{v}^{t}\gamma^V(r-v)y(r)dr\right)\Big]\nu(dv)\non\\
					&-\int_{0}^{t}\int_v^t\mathcal{K}(x,y)(v)e^{-\nu((v,a])}\E\left[\exp\left(-\int_{v}^{a}\gamma^V(r-v)y(r)dr\right)\right]\nu(da)\nu(dv).
				\end{align}
				As a result, Proposition~\ref{Lem:mass-conserv} follows from equation~\eqref{eq-cons-2}, \eqref{eq-cons-3}, \eqref{eq-cons-4}, \eqref{eq-cons-5}, and \eqref{eq-cons-6}.
				\subsection{Proof of Theorem~\ref{Th:exists}}\label{exist}
				In this section, we establish the existence and uniqueness of the system of equations \eqref{eq:F-G-x}–\eqref{eq:F-G-y}. We begin with the proof of uniqueness.
				
				\subsubsection{Uniqueness}
				From Proposition~\ref{Lem:mass-conserv} and Assumption~\ref{AS-lambda}, we first note that if $(x,y)\in D^2_+$ is a solution of equation \eqref{eq:F-G-x}-\eqref{eq:F-G-y}, for all $0\leq t\leq T,$ there exists $C_T>0,$ such that,
				\[x(t)\leq C_T\text{ and }y(t)\leq\lambda_* C_T.\] 
				We suppose now $(x_1,y_1)$ and $(x_2,y_2)$ two solutions of equation \eqref{eq:F-G-x}-\eqref{eq:F-G-y}. Then we have,
				\begin{align*}
					&x_1(t)-x_2(t)=\E\left[\gamma_0^I(t)\indic{t\leq\eta_{0}}\Bigg(\exp\left(-\int_{0}^{t}\gamma_{0}^I(r)y_1(r)dr\right)-\exp\left(-\int_{0}^{t}\gamma_{0}^I(r)y_2(r)dr\right)\Bigg)\right]
					\\&+\E\left[\gamma_0^I(t)e^{-\nu((\eta_{0},t])}\indic{t>\eta_{0}}\Bigg(\exp\left(-\int_{0}^{t}\gamma_{0}^I(r)y_1(r)dr\right)-\exp\left(-\int_{0}^{t}\gamma_{0}^I(r)y_2(r)dr\right)\Bigg)\right]\\
					&+\int_{0}^{t}\E\left[\gamma^I(t-s)\indic{t-s\leq\eta}\Bigg(\exp\left(-\int_{s}^{t}\gamma^I(r-s)y_1(r)dr\right)x_1(s)y_1(s)\right.\\
					&\hspace{3cm}\left.-\exp\left(-\int_{s}^{t}\gamma^I(r-s)y_2(r)dr\right)x_2(s)y_2(s)\Bigg)\right]ds\\
					&+\int_{0}^{t}\E\left[\gamma^I(t-s)e^{-\nu((s+\eta,t])}\indic{t-s>\eta}\Bigg(\exp\left(-\int_{s}^{t}\gamma^I(r-s)y_1(r)dr\right)x_1(s)y_1(s)\right.\\
					&\hspace{3cm}\left.-\exp\left(-\int_{s}^{t}\gamma^I(r-s)y_2(r)dr\right)x_2(s)y_2(s)\Bigg)\right]ds\\
					&+\int_0^t\Big(\mathcal{K}(x_1, y_1)(v) - \mathcal{K}(x_2, y_2)(v)\Big)e^{-\nu((v,t])}\\
					&\hspace{2cm}\times\E\left[\gamma^V(t-v)\exp\left(-\int_{v}^{t}\gamma^V(r-v)y_1(r)dr\right)\right]\nu(dv)\\
					&+\int_0^t \mathcal{K}(x_2, y_2)(v)e^{-\nu((v,t])}\E\Bigg[\gamma^V(t-v)\Bigg(\exp\left(-\int_{v}^{t}\gamma^V(r-v)y_1(r)dr\right)\\
					&\hspace{2cm}-\exp\left(-\int_{v}^{t}\gamma^V(r-v)y_2(r)dr\right)\Bigg)\Bigg]\nu(dv).
				\end{align*}
				Since $|e^{-a}-e^{-b}| \leq |a-b|$ for all $a, b \in \mathbb{R}_+$, and given that $\gamma^I \leq 1$, $\gamma^I_0 \leq 1$, $\gamma^V\leq1$, and $x_i y_i \leq \lambda_* C_T^2$, the following holds. Furthermore, as $\mathbb{F}_s$ and $\mathbb{F}_{s,z}$ are probability measures and $\nu$ is a $\sigma$-finite measure (locally finite on $[0, T]$), following the methodology in \cite[Theorem~$3.1$, p.~$24$]{forien-Zotsa2022stochastic}, there exists a constant $L_T > 0$ such that:
				\[
				|\mathcal{K}(x_1, y_1)(v) - \mathcal{K}(x_2, y_2)(v)| \leq L_T \int_{0}^{T} [|x_1(s) - x_2(s)| + |y_1(s) - y_2(s)|] ds.
				\]
				Consequently, for a given $T > 0$, there exists a constant $C > 0$ such that for all $0 \leq t \leq T$:
				\begin{equation}\label{eq:x}
					|x_1(t) - x_2(t)| \leq C \int_{0}^{t} [|x_1(s) - x_2(s)| + |y_1(s) - y_2(s)|] ds.
				\end{equation}
				By definition of $y_1$ and $y_2$, we deduce that,
				\begin{equation}\label{eq:y}
					\left|y_1(t)-y_2(t)\right|\leq \lambda_*C\int_{0}^{t}[|x_1(s)-x_2(s)|+|y_1(s)-y_2(s)|]ds,
				\end{equation}
				Uniqueness follows from \eqref{eq:x} and \eqref{eq:y} and the Gronwall Lemma.
				
				\subsubsection{Existence}
				
				To establish the existence of a solution on \(\mathbb{R}_+\), it suffices to prove existence on every compact subset of \(\mathbb{R}_+\).
				
				Define the mappings:
				\begin{equation}\label{eq:F-G-x-exist}
					\begin{aligned}
						F(x,y)(t)&=\E\left[\gamma_0^I(t)\indic{t\leq\eta_{0}}\exp\left(-\int_{0}^{t}\gamma_{0}^I(r)y(r)dr\right)\right]\\
						&\hspace{0.5cm}+\E\left[\gamma_0^I(t)e^{-\nu((\eta_{0},t])}\indic{t>\eta_{0}}\exp\left(-\int_{0}^{t}\gamma_{0}^I(r)y(r)dr\right)\right]\\
						&\hspace{0.5cm}+\int_{0}^{t}\E\left[\gamma^I(t-s)\indic{t-s\leq\eta}\exp\left(-\int_{s}^{t}\gamma^I(r-s)y(r)dr\right)\right]x(s)y(s)ds\\
						&\hspace{0.5cm}+\int_{0}^{t}\E\left[\gamma^I(t-s)e^{-\nu((s+\eta,t])}\indic{t-s>\eta}\exp\left(-\int_{s}^{t}\gamma^I(r-s)y(r)dr\right)\right]x(s)y(s)ds\\&\hspace{0.5cm}+\int_0^t\mathcal{K}(x,y)(v)e^{-\nu((v,t])}\E\left[\gamma^V(t-v)\exp\left(-\int_{v}^{t}\gamma^V(r-v)y(r)dr\right)\right]\nu(dv)
					\end{aligned}
				\end{equation}
				and 
				\begin{equation}\label{eq:F-G-y-exist}
					\begin{aligned}
						G(x,y)(t)&=\E\left[\lambda_0^I(t)\indic{t\leq\eta_{0}}\exp\left(-\int_{0}^{t}\gamma_{0}^I(r)y(r)dr\right)\right]\\
						&\hspace{1cm}+\E\left[\lambda_0^I(t)e^{-\nu((\eta_{0},t])}\indic{t>\eta_{0}}\exp\left(-\int_{0}^{t}\gamma_{0}^I(r)y(r)dr\right)\right]\\
						&\hspace{1cm}+\int_{0}^{t}\E\left[\lambda^I(t-s)\indic{t-s\leq\eta}\exp\left(-\int_{s}^{t}\gamma^I(r-s)y(r)dr\right)\right]x(s)y(s)ds\\
						&\hspace{1cm}+\int_{0}^{t}\E\left[\lambda^I(t-s)e^{-\nu((s+\eta,t])}\indic{t-s>\eta}\exp\left(-\int_{s}^{t}\gamma^I(r-s)y(r)dr\right)\right]x(s)y(s)ds.							
					\end{aligned}
				\end{equation}
				
				Let \(H : L^\infty(\mathbb{R}_+,\R^2_+) \to L^\infty(\mathbb{R}_+,\R_+^2)\) be defined by
				\[
				H(x,y)(t) = \left(F(x,y)(t), G(x,y)(t)\right),
				\]
				where \(L^\infty(\mathbb{R}_+,\R_+^2)\) denotes the space of pairs of bounded measurable functions from \(\mathbb{R}_+\) to \(\mathbb{R}_+\).
				
				Recall that \((L^\infty(\mathbb{R}_+,\R^2_+), \|\cdot\|_\infty)\) is a Banach space under the supremum norm:
				\[
				\|(x,y)\|_\infty = \|x\|_\infty + \|y\|_\infty.
				\]
				Fix \(T \geq 0\), and define the truncated norm:
				\[
				\|(x,y)\|_T = \|(x,y)\indic{[0,T]}\|_\infty.
				\]
				
				From the uniqueness setting, \(H\) is continuous from \(L^\infty([0,T],\R^2_+)\) to itself.
				
				Define the set:
				\[
				E = \left\{(x,y) \in L^\infty([0,T],\R^2_+) : \exists \alpha \in [0,1],\ \alpha H(x,y) = (x,y)\right\}.
				\]
				
				Our goal is to show that there exists \((x,y) \in L^\infty([0,T],\R^2_+)\) such that \(H(x,y) = (x,y)\). By Schaefer's fixed-point theorem (see \cite{eberhard1990nonlinear,forster2014leray} or \cite[Chapter~4]{precup2002methods}), it suffices to prove that:
				\(H\) is continuous and compact, and the set \(E\) is bounded.
				
				\medskip
				Since the map \(t\mapsto H(x,y)(t)\) is càdlàg on \([0,T]\), the fixed point \((x,y)\in D([0,T], \mathbb{R}^2_+)\).
				
				\subsubsection*{Compactness in \(L^\infty([0,T], \mathbb{R}^2_+)\)}
				
				Let
				\[
				\mathbb{B} = \left\{(x,y) \in L^\infty([0,T], \mathbb{R}^2_+) : \|(x,y)\|_T \leq 1\right\},
				\text{ and } \mathbb{K} = H(\mathbb{B}).
				\]
				
				By \cite[Theorem, page~4]{cherkas1970compactness}, \(\mathbb{K}\) is compact in \(L^\infty([0,T], \mathbb{R}^2_+)\) if, for every \(\varepsilon > 0\), there exists a finite partition \((A_i)_i\) of \([0,T]\) into disjoint measurable sets such that:
				\begin{equation}\label{eq-equic}
					\sup_i \sup_{s,t \in A_i} |H(x,y)(t) - H(x,y)(s)| \leq \varepsilon,
				\end{equation}
				and
				\[
				\sup_{(x,y) \in \mathbb{B}} \|H(x,y)\|_T < \infty.
				\]
				
				To verify \eqref{eq-equic}, we introduce the modulus of continuity:
				\[
				W'(f,\delta) = \inf_{\{t_i\},\, t_i - t_{i-1} \ge \delta} \max_{1 \leq i \leq r} \sup_{t_{i-1} \leq s \leq t < t_i} |f(t) - f(s)|.
				\]
				It suffices to show:
				\[
				\lim_{\delta \to 0} \sup_{(x,y) \in \mathbb{B}} W'(H(x,y), \delta) = 0.
				\]
				
				Since $|e^{-a}-e^{-b}|\leq|a-b|,\,\forall a,b\in\R_+,\,\gamma_0^I \leq 1$, \(\gamma^I \leq 1\), and \(\|(x,y)\|_T \leq 1\), for $h\geq0,$ we have,
				\begin{align*}
					&\left|\E\left[\gamma_0^I(t+h)\indic{t+h\leq\eta_{0}}\exp\left(-\int_{0}^{t+h}\gamma_{0}^I(r)y(r)dr\right)\right]-\E\left[\gamma_0^I(t)\indic{t\leq\eta_{0}}\exp\left(-\int_{0}^{t}\gamma_{0}^I(r)y(r)dr\right)\right]\right|\\
					&\leq\E\left[\left|\gamma_0^I(t+h)\indic{t+h\leq\eta_{0}}-\gamma_0^I(t)\indic{t\leq\eta_{0}}\right|\right]+h,
				\end{align*}
				
				\begin{align*}
					&\Bigg|\int_{0}^{t+h}\E\left[\gamma^I(t+h-s)\indic{t+h-s\leq\eta}\exp\left(-\int_{s}^{t+h}\gamma^I(r-s)y(r-s)\right)\right]x(s)y(s)ds\\
					&\hspace{2cm}-\int_{0}^{t}\E\left[\gamma^I(t-s)\indic{t-s\leq\eta}\exp\left(-\int_{s}^{t}\gamma^I(r-s)y(r-s)\right)\right]x(s)y(s)ds\Bigg|\\
					&\leq (1+T)h+\int_{0}^{T}\E\left[\left|\gamma^I(s+h)\indic{s+h\leq\eta}-\gamma^I(s)\indic{s\leq\eta}\right|\right]ds.
				\end{align*}
				and there exists $C_T>0,$ such that,
				\begin{align*}
					&\Bigg|\int_0^{t+h} \mathcal{K}(x,y)(v)e^{-\nu((v,t+h])}\E\left[\gamma^V(t+h-v)\exp\left(-\int_{v}^{t+h}\gamma^V(r-v)y(r)dr\right)\right]\nu(dv)\\
					&-\int_0^t\mathcal{K}(x,y)(v) e^{-\nu((v,t])}\E\left[\gamma^V(t-v)\exp\left(-\int_{v}^{t}\gamma^V(r-v)y(r)dr\right)\right]\nu(dv)\Bigg|\\
					&\leq C_T\left(\nu((t, t+h]) + h+\int_{0}^{T} \mathbb{E}\left[ \left| \gamma^V(s+h) - \gamma^V(s) \right| \right] \nu(ds)\right)
				\end{align*}
				We obtain similar bound for the others two terms.
				
				Hence, there exists \(C_T > 0\) such that:
				\begin{align*}
					W'(F(x,y), \delta) &\leq C_T\Bigg[\mathbb{E}[W'(\gamma_0^I\indic{\cdot\leq\eta_0}, \delta)]+\E\left[W'(\gamma_0^Ie^{-\nu((\eta_{0},\cdot])}\indic{\cdot>\eta_{0}},\delta)\right]+W'(Id, \delta)\\
					&\hspace{1cm} + \mathbb{E}[W'(\gamma^I\indic{\cdot\leq\eta}, \delta)]+\mathbb{E}[W'(\gamma^I\indic{\cdot>\eta}, \delta)]+W'(\nu((0,\cdot]), \delta)+\mathbb{E}[W'(\gamma^V, \delta)]\Bigg]
				\end{align*}
				where \(Id(t) = t\).
				
				Similarly, using \(\lambda_0^I, \lambda^I \leq \lambda_*\), we obtain:
				\begin{align*}
					W'(G(x,y), \delta) &\leq \lambda_*\Bigg(\mathbb{E}[W'(\gamma_0^I\indic{\cdot\leq\eta_0}, \delta)]+\E\left[W'(\gamma_0^Ie^{-\nu((\eta_{0},\cdot])}\indic{\cdot>\eta_{0}},\delta)\right]+(3+2T) W'(Id, \delta)\\
					&\hspace{1cm} + T\mathbb{E}[W'(\gamma^I\indic{\cdot\leq\eta}, \delta)]+T\mathbb{E}[W'(\gamma^I\indic{\cdot>\eta}, \delta)]+TW'(\nu((0,\cdot]), \delta)\Bigg).
				\end{align*}
				Since all involved functions belong to \(D([0,T])\) and are bounded, the dominated convergence theorem implies:
				\[
				\lim_{\delta \to 0} \sup_{(x,y) \in \mathbb{B}} W'(H(x,y), \delta) = 0.
				\]
				
				Moreover, it is straightforward to show that, there exists $C_T>0,$
				\[
				\sup_{(x,y) \in \mathbb{B}} \|H(x,y)\|_T \leq C_T.
				\]
				
				Thus, $\mathbb{K}$ is compact.
				
				\subsubsection*{Boundedness of the Set \(E\)}
				
				Let \((x,y) \in E\), so there exists \(\alpha \in [0,1]\) such that $(x,y) = \alpha H(x,y)$. 
				From the results in Section~\ref{sec:mass-conserv} it follows that, 
				\begin{multline*}
					\begin{aligned}
						&\alpha\E\left[\indic{t\leq\eta_{0}}\exp\left(-\int_{0}^{t}\gamma_{0}^I(r)y(r)dr\right)\right]+\alpha\E\left[e^{-\nu((\eta_0,t])}\indic{t>\eta_{0}}\exp\left(-\int_{0}^{t}\gamma_{0}^I(r)y(r)dr\right)\right]\\
						&\hspace{1cm}+\alpha\int_{0}^{t}\E\left[\indic{t-s\leq\eta}\exp\left(-\int_{s}^{t}\gamma^I(r-s)y(r)dr\right)\right]x(s)y(s)ds\\
						&\hspace{1cm}+\alpha\int_{0}^{t}\E\left[e^{-\nu((s+\eta,t])}\indic{t-s\geq\eta}\exp\left(-\int_{s}^{t}\gamma^I(r-s)y(r)dr\right)\right]x(s)y(s)ds\\
						&\hspace{1cm}+\alpha\int_{0}^{t}\E\left[e^{-\nu((\eta_0,a])}\indic{a\geq\eta_{0}}\exp\left(-\int_{0}^{a}\gamma_{0}^I(r)y(r)dr\right)\right]\nu(da)\\
						&\hspace{1cm}+\alpha\int_{0}^{t}\int_{s}^{t}\E\left[e^{-\nu((s+\eta,a])}\indic{a-s\geq\eta}\exp\left(-\int_{s}^{a}\gamma^I(r-s)y(r)dr\right)\right]\nu(da)x(s)y(s)ds\\
						&\hspace{1cm}+\alpha\int_{0}^{t}e^{-\nu((v,t])}\mathcal{K}(x,y)(v)\E\left[\exp\left(-\int_{v}^{t}\gamma^V(r-v)y(r)dr\right)\right]\nu(dv)\\
						&\hspace{1cm}+\alpha\int_{0}^{t}\mathcal{K}(x,y)(v)\int_{v}^{t}e^{-\nu((v,a])}\E\Bigg[\exp\left(-\int_{v}^{a}\gamma^V(r-v)y(r)dr\right)\Bigg]\nu(da)\nu(dv)\\
						&=\alpha+\alpha\int_{0}^{t}\mathcal{K}(x,y)(v)\nu(dv)-(1-\alpha)\int_{0}^{t}x(s)y(s)ds\\
						&\leq \alpha+\alpha\int_{0}^{t}\mathcal{K}(x,y)(v)\nu(dv)
					\end{aligned}
				\end{multline*}
				where the last line follows from the fact that $x$ and $y$ are non-negative.
				
				As a result, since $(x,y) = \alpha H(x,y)$, given by the  expressions \eqref{eq:F-G-x-exist} and \eqref{eq:F-G-y-exist}, and the fact that, $\mathcal{K}(x,y)(v)\leq1$ and $\alpha\leq1$, we deduce that, for all $0\leq t\leq T,$
				\begin{equation*}
					x(t)\leq\alpha\Bigg(1+\int_{0}^{T}\mathcal{K}(x,y)(v)\nu(dv)\Bigg)\leq 1+\nu([0,T])
				\end{equation*}
				and 
				\begin{equation*}
					y(t)\leq\lambda_*\alpha\Bigg(1+\int_{0}^{T}\mathcal{K}(x,y)(v)\nu(dv)\Bigg)\leq \lambda_*\big(1+\nu([0,T])\big).
				\end{equation*}
				Hence $\|(x,y)\|_T\leq (1+\lambda_*)\big(1+\nu([0,T])\big).$
				
				This concludes the proof.

				\subsection{Proof of Theorem~\ref{Th:FLLN}}
				We first construct a system of stochastic equations driven by Poisson random measures (PRMs), and then use an approach of the type of propagation of chaos as in \cite{forien-Zotsa2022stochastic}. 
				
				Let $(\lambda_0^I,\gamma_0^I)$, and $(\lambda^I,\gamma^I)$ be a random variable taking values in $D_+^2$. Let $\gamma^V$ an i.i.d random variables taking values in $D_+$. We assume that $(\lambda_0^I,\gamma_0^I)$, $(\lambda^I,\gamma^I)$ and $\gamma^V$ are independent and satisfies Assumption~\ref{AS-lambda}.
				Also let $Q$ be a standard Poisson random measure on $\R^2_+$, and $Q^V$ be a Poisson standard random measure on $\R_+$, both independent of the previous random variables. We set $\mathcal{N}_t=Q^V([0.t]).$
				We define for $t\geq0$, 
				\begin{equation}
					\left\{
					\begin{aligned}
						A(t)&:=\int_{0}^{t}\int_{0}^{+\infty}\indic{u\leq\Upsilon(r^-)}Q(dr,du)\\[0.5cm]
						\Upsilon(t)&:=\gamma_{A(t),\mathcal{N}(t)}(t)\mathbb{E}\left[\lambda_{A(t),\mathcal{N}(t)}(t)\right],\label{suite_A}
					\end{aligned}\right.
				\end{equation}
				with
				\begin{align*}
					\lambda_{A(t),\mathcal{N}(t)}(t)=\lambda_{A(t)}^I(t-\tau_{A(t)})\indic{ T^V_{\mathcal{N}_t}-\tau_{A(t)}\leq\eta_{A(t)}}.
				\end{align*}
				and 
				\begin{align*}
					\gamma_{A(t),\mathcal{N}(t)}(t)=\gamma_{A(t)}^I(t-\tau_{A(t)})\indic{T_{\mathcal{N}_t}^V-\tau_{A(t)}\leq\eta_{A(t)}}+\gamma^V_{\mathcal{N}_t}(t-T^V_{\mathcal{N}_t})\indic{T_{\mathcal{N}_t}^V-\tau_{A(t)}>\eta_{A(t)}}
				\end{align*}
				where $\tau_{A(t)}$ is the jump time of process $ A $ at time $t$. We recall that $T^V_{\mathcal{N}_t}$ is the vaccination time at time $t$ that corresponds to the jump time of process $\mathcal{N}$ at time $t$.
				
				Note that we construct $A$ by induction on the jumps times. Assumption \ref{AS-lambda}  implies that the rate $\Upsilon(t)$ is bounded almost surely
				by the constant $\lambda_\ast$. Consequently the jump times do not accumulate, and the above induction defines $A(t)$ for all $t\geq0$.
				\begin{prop} \label{prop:equal-G-F}
					Under Assumption~\ref{AS-lambda}, we have the following identity:
					\begin{equation}\label{G-eq-fS}
						\big(\overline{\mathfrak{S}}(t),\overline{\mathfrak{F}}(t)\big)=\big(\mathbb{E}\left[\gamma_{A(t),\mathcal{N}(t)}(t)\right],\mathbb{E}\left[\lambda_{A(t),\mathcal{N}(t)}(t)\right]\big).
					\end{equation}
				\end{prop}
				\begin{proof}
					We define the pair  $(x(t),y(t)):=\big(\mathbb{E}\left[\gamma_{A(t),\mathcal{N}(t)}(t)\right],\mathbb{E}\left[\lambda_{A(t),\mathcal{N}(t)}(t)\right]\big)$. From Section~\ref{sec-exp-F-G}, from Lemma~\ref{Lem:-I-x} and ~\ref{Lem-gam-II} we derive that,
					\begin{equation*}
						\begin{aligned}
							&x(t)=e^{-\nu([0,t])}\Bigg(\E\left[\gamma_{0}^I(t)\exp\left(-\int_{0}^{t}\gamma_{0}^I(r)y(r)dr\right)\right]\\
							&\hspace{1cm}+\int_{0}^{t}\E\left[\gamma^I(t-s)\exp\left(-\int_{s}^{t}\gamma^I(r-s)y(r)dr\right)\right]y(s)x(s)ds\Bigg)\\
							&+\int_{0}^{t}e^{-\nu((v,t])}\E\left[\gamma_{0}^I(t)\indic{v\leq\eta_{0}}\exp\left(-\int_{0}^{t}\gamma_{0}^I(r)y(r)dr\right)\right]\nu(dv)\\
							&\hspace{0.5cm}+\int_{0}^{t}\int_{s}^{t}e^{-\nu((v,t])}\E\left[\gamma^I(t-s)\indic{v-s\leq\eta}\exp\left(-\int_{s}^{t}\gamma^I(r-s)y(r)dr\right)\right]\nu(dv)x(s)y(s)ds\\
							&+\int_{0}^{t}\int_{0}^{s}e^{-\nu((v,t])}\E\left[\gamma^I(t-s)\exp\left(-\int_{s}^{t}\gamma^I(r-s)y(r)dr\right)\right]\nu(dv)x(s)y(s)ds\\
							&+\int_{0}^{t}e^{-\nu((v,t])}\mathcal{K}(x,y)(v)
							\E\left[\gamma^V(t-v)\exp\left(-\int_{v}^{t}\gamma^V(r-v)y(r)dr\right)\right]\nu(dv),
						\end{aligned}
					\end{equation*}
					where $\mathcal{K}(x,y)(v)$ is defined as in~\eqref{eq-Psi-0}:
					\begin{multline*}
						\begin{aligned}
							\mathcal{K}(x,y)(v)&=1-\E\Big[\indic{ v\leq\eta_{0}}\exp\left(-\int_{0}^{v}\gamma_{0}^I(r)y(r)dr\right)\Big]\\
							&\hspace{1cm}-\int_{0}^{v}\E\Big[\indic{v-s\leq\eta}\exp\left(-\int_{s}^{v}\gamma^I(r-s)y(r)dr\right)\Big]x(s)y(s)ds.
						\end{aligned}
					\end{multline*}
					Furthermore, by applying the identity, for $0\leq s\leq t:$
					\begin{equation}\label{eq-id-nu0}
						\int_{0}^{t}e^{-\nu((v,t])}\indic{v\leq\eta_0}\nu(dv)=\indic{t\leq\eta_0}+e^{-\nu((\eta_0,t])}\indic{t>\eta_0}-e^{-\nu((0,t])},
					\end{equation}
					and
					\begin{equation}\label{eq-id-nu}
						\int_{s}^{t}e^{-\nu((v,t])}\indic{v-s\leq\eta}\nu(dv)=\indic{t-s\leq\eta}+e^{-\nu((s+\eta,t])}\indic{t-s>\eta}-e^{-\nu((s,t])},
					\end{equation}
					we show that the expression for $x$ reduce to equation~\eqref{eq:F-G-x}.
					
					Following a similarly approach, by combining Lemma~\ref{Lem:exp-I-y},~\ref{lem:Inf} and~\ref{Lem:exp-II-y} with identity~\eqref{eq-id-nu0}-\eqref{eq-id-nu}, we show that $y$ is given by equation~\eqref{eq:F-G-y}.
					
					Consequently, the pair $(x,y)$ satisfies the system of equations~\eqref{eq:F-G-x}-\eqref{eq:F-G-y}. By the uniqueness of the solution established in Theorem~\ref{Th:exists}, the proof is complete.	
				\end{proof}
				\begin{remark}\label{RQ-bound}
					From Proposition~\ref{prop:equal-G-F} we derive that for all $t\geq0,\,\fF(t)\leq\lambda_*$ and $\fS(t)\leq1.$
				\end{remark}
				
				We next consider the sequence $(Q_k)_{k\geq1}$ and $(Q_k^V)_{k\geq1}$ of Poisson random measures introduced in Section~\ref{sec-model-descr} and for each $k\geq1,$ we define the process $\{A_k(t),\,t\geq0\}$: 
				\[A_k(t)=\int_{0}^{t}\int_{0}^{+\infty}\indic{u\leq\Upsilon_k(r^-)}Q_k(dr,du),\] where
				\[\Upsilon_k(t)=\gamma_{k,A_k(t),\mathcal{N}_t^k}(t)\overline{\mathfrak{F}}(t),\quad\mathcal{N}_t^k=Q^V_k([0,t]),\] 
				\begin{equation}
					\gamma_{k,A_k(t),\mathcal{N}_t^k}(t)=\gamma_{k,A_k(t)}^I(t-\tau_{k,A_k(t)})\indic{ T_{\mathcal{N}_t^k}^V-\tau_{k,A_k(t)}\leq\eta_{k,A_k(t)}}+\gamma_{k,\mathcal{N}^k_t}^V(t-\tau_{k,A_k(t)})\indic{T_{\mathcal{N}_t^k}^V-\tau_{k,A_k(t)}>\eta_{k,A_k(t)}}
				\end{equation}
				and 
				\begin{equation}
					\lambda_{k,A_k(t),\mathcal{N}_t^k}(t)=\lambda_{k,A_k(t)}^I(t-\tau_{k,A_k(t)})\indic{ T_{\mathcal{N}_t^k}^V-\tau_{k,A_k(t)}\leq\eta_{k,A_k(t)}},
				\end{equation}
				with $(\tau_{k,i})_i$ the jump times of process $A_k$.

				In this definition we use the same $(\lambda_{k,i}^I,\gamma_{k,i}^I,\gamma^V_{k,i},Q_k,Q^V_k)$ as in the definition of the model in Section~\ref{sec-model-descr}. 
				Moreover, since  $\left((\lambda_{k,i}^I)_i,(\gamma_{k,i}^I)_i,(\gamma^V_{k,i})_i,Q_k,Q^V_k\right)_{k\geq1}$ are i.i.d,\\ $\left(\lambda_{k,A_k(\cdot),\mathcal{N}_\cdot^k},\gamma_{k,A_k(\cdot),\mathcal{N}_\cdot^k},Q_k\right)_{k\geq1}$ are also i.i.d, and then $(A_k)_{k\geq1}$ are also i.i.d.

				Since the sequences $\left( (\lambda_{k, i, \mathcal{N}^k})_i, (\gamma_{k, i, \mathcal{N}^k})_i, Q_k \right)_{k \geq 1}$ are i.i.d., and given the uniform bounds $\lambda_{k, i, \mathcal{N}^k} \leq \lambda_*$ and $\gamma_{k, i, \mathcal{N}^k} \leq 1$ from Assumption~\ref{AS-lambda}, as well as $\fF \leq \lambda_*$ and $\fS \leq 1$ from Remark~\ref{RQ-bound}, we can apply the same computation as in \cite[Lemma~$6.2$]{forien-Zotsa2022stochastic} to obtain the following result.
				
				\begin{lemma}\label{lem_inq}Assumption~\ref{AS-lambda}, for $k\in\mathbb{N}$ and $T\geq0$, 
					\begin{equation*}
						\mathbb{E}\left[\sup_{t\in[0,T]}\left|A^N_k(t)-A_k(t)\right|\right]\leq\int_{0}^{T}\mathbb{E}\Big[\left|\Upsilon^N_k(t)-\Upsilon_k(t)\right|\Big]dt=:\delta^N(T)\label{eqA}
					\end{equation*}
					and 
					\begin{equation*}
						\mathbb{E}\left[\sup_{t\in[0,T]}\left|\tau^N_{k,A^N_k(t)}-\tau_{k,A_k(t)}\right|\right]\leq T\delta^N(T).
					\end{equation*}
					Moreover, 
					\begin{equation}\delta^N(T)\leq\frac{\lambda^*}{\sqrt{N}}T\exp(2\lambda^*T).\label{eqdelta}\end{equation}
				\end{lemma}
				From Lemma~\ref{lem_inq} or \cite[Lemma~$6.3$]{forien-Zotsa2022stochastic}, we also have the following Lemma.
				\begin{lemma}
					For $k\in\mathbb{N}$ and $T\geq0$ we have 
					\begin{equation}
						\E\left[\sup_{t\in[0,T]}\left|\gamma_{k,A_k^N(t),\mathcal{N}_t^k}(t)-\gamma_{k,A_k(t),\mathcal{N}_t^k}(t)\right|\right]\leq\frac{\lambda_*}{\sqrt{N}}T\exp(2\lambda_*T),\label{eqgam}
					\end{equation}
					\begin{equation}
						\E\left[\sup_{t\in[0,T]}\left|\lambda_{k,A_k^N(t),\mathcal{N}_t^k}(t)-\lambda_{k,A_k(t),\mathcal{N}_t^k}(t)\right|\right]\leq\frac{\lambda_*^{2}}{\sqrt{N}}T\exp(2\lambda_*T),\label{eqlam}
					\end{equation}
				\end{lemma}
				
				\medskip	
				\begin{proof}[Completing the proof of Theorem~\ref{Th:FLLN}]
					For $t\geq0$, we have
					\begin{align*}
						\overline{\mathfrak{F}}^N(t)&=\frac{1}{N}\sum_{k=1}^{N}\lambda_{k,A_k^N(t),\mathcal{N}_t^k}(t)\\
						&=\frac{1}{N}\sum_{k=1}^{N}\left(\lambda_{k,A_k^N(t),\mathcal{N}_t^k}(t)-\lambda_{k,A_k(t),\mathcal{N}_t^k}(t)\right)+\frac{1}{N}\sum_{k=1}^{N}\lambda_{k,A_k(t),\mathcal{N}_t^k}(t),
					\end{align*}
					and 
					\begin{align*}
						\overline{\mathfrak{S}}^N(t)&=\frac{1}{N}\sum_{k=1}^{N}\gamma_{k,A_k^N(t),\mathcal{N}_t^k}(t)\\
						&=\frac{1}{N}\sum_{k=1}^{N}\left(\gamma_{k,A_k^N(t),\mathcal{N}_t^k}(t)-\gamma_{k,A_k(t),\mathcal{N}_t^k}(t)\right)+\frac{1}{N}\sum_{k=1}^{N}\gamma_{k,A_k(t),\mathcal{N}_t^k}(t).
					\end{align*}
					From \eqref{eqgam} and \eqref{eqlam}, we have 
					\begin{equation*}\left\{
						\begin{aligned}
							\E\left[\frac{1}{N}\sum_{k=1}^{N}\sup_{t \in [0,T]}\left|\lambda_{k,A_k^N(t),\mathcal{N}_t^k}(t)-\lambda_{k,A_k(t),\mathcal{N}_t^k}(t)\right|\right]\leq\frac{\lambda^{*2}}{\sqrt{N}}T\exp(2\lambda^*T),\\
							\E\left[\frac{1}{N}\sum_{k=1}^{N}\sup_{t \in [0,T]}\left|\gamma_{k,A_k^N(t),\mathcal{N}_t^k}(t)-\gamma_{k,A_k(t),\mathcal{N}_t^k}(t)\right|\right]\leq\frac{\lambda^{*}}{\sqrt{N}}T\exp(2\lambda^*T).
						\end{aligned}\right.
					\end{equation*}
					Hence, 
					\begin{multline*}
						\left(\frac{1}{N}\sum_{k=1}^{N}\left(\gamma_{k,A_k^N(t),\mathcal{N}_t^k}(t)-\gamma_{k,A_k(t),\mathcal{N}_t^k}(t)\right), \right. \\ \left. \frac{1}{N}\sum_{k=1}^{N}\left(\lambda_{k,A_k^N(t),\mathcal{N}_t^k}(t)-\lambda_{k,A_k(t),\mathcal{N}_t^k}(t)\right)\right) \xrightarrow[N\to+\infty]{}(0,0)
					\end{multline*}
					locally uniformly in $t$. 
					
					Moreover, as $\left(\gamma_{k,A_k(\cdot),\mathcal{N}_\cdot^k}(\cdot),\lambda_{k,A_k(\cdot),\mathcal{N}_\cdot^k}(\cdot)\right)_k$ is a collection of i.i.d. random variables in $D^2$, by the law of large numbers in $D^2$ \cite[Theorem~$1$]{rao1963law},
					\begin{multline*}\left(\frac{1}{N}\sum_{k=1}^{N}\gamma_{k,A_k(\cdot),\mathcal{N}_\cdot^k}(\cdot),\frac{1}{N}\sum_{k=1}^{N}\lambda_{k,A_k(\cdot),\mathcal{N}_\cdot^k}(\cdot)\right)\\\xrightarrow[N\to+\infty]{\mathbb{P}}\left(\E\left[\gamma_{1,A_1(\cdot),\mathcal{N}_\cdot^1}(\cdot)\right],\E\left[\lambda_{1,A_1(\cdot),\mathcal{N}_\cdot^1}(\cdot)\right]\right)\text{ in }D^2.\end{multline*}
					
					In view of Proposition~\ref{prop:equal-G-F}, the proof of the convergence of the pair $(\fS^N,\fF^N)$ is now complete. To conclude the proof of
					Theorem~\ref{Th:FLLN}, it remains to establish the convergence of $\mathfrak{U}^N(\varphi)$. Since the random function $\varphi$ has the
					same properties as the function $\gamma$ or $\lambda$, the proof of convergence of $\mathfrak{U}^N(\varphi)$ is easily deduced from that of $\fS^N$ or $\fF^N$.

				\end{proof}

				\section{Proof for the Longtime behavior}\label{sec-LT}
				\subsection{Proof of Theorem~\ref{Th-LT-free}}\label{sec-LT-th-1}
				\subsubsection{proof of Theorem~\ref{Th-LT-free}\ref{Th-lbda-gamma-i}}
				We recall from the system of equations~\eqref{eq-inf-F-sf-S-no-reinf}-\eqref{eq-inf-F-sf-F-no-reinf}, we have for $t\geq0,$
				\begin{equation*}
					\fF(t)=
					\E\left[\lambda_{0}^I(t)\right]+\int_{0}^{t}\E\left[\lambda^I(t-s)\right]\fS(s)\fF(s)ds.							
				\end{equation*}
				This implies that,
				\begin{align}
					\int_{0}^{t}\fF(u)du&=\int_{0}^{t}\E\left[\lambda_0(u)\right]du+\int_{0}^{t}\left(\int_{0}^{t-s}\E\left[\lambda(u)\right]du\right)\fS(s)\fF(s)ds\label{equiv-end-F-inf}\\
					&\leq \int_{0}^{t}\E\left[\lambda_0(u)\right]du+R_0\int_{0}^{t}\fF(s)ds,\non
				\end{align}
				where the last line follows from the fact that, $\fS(s)\leq1$.
				
				Hence since $R_0<1$, we obtain,
				\[\int_{0}^{t}\fF(u)du\leq\frac{1}{1-R_0}\int_{0}^{t}\E\left[\lambda_0(u)\right]du.\]
				From Assumption~\ref{ass-lbda-gamma}, it follows that,
				\[\int_{0}^{\infty}\fF(u)du<\infty.\]
				As a result, since $\fS(t)\leq1,\,\lambda^I(t)\leq\lambda_*,\E\left[\lambda^I_0(t)\right]\to0$ and $\E\left[\lambda^I(t)\right]\to0$ as $t\to\infty,$ by dominated convergence theorem, we conclude that,
				\begin{equation*}
					\lim_{t\to\infty}\fF(t)=
					\lim_{t\to\infty}\int_{0}^{t}\E\left[\lambda^I(t-s)\right]\fS(s)\fF(s)ds=0.							
				\end{equation*}
				\subsubsection{proof of Theorem~\ref{Th-LT-free}\ref{Th-lbda-gamma-ii}}
				We now establish the proof for general case. We adapt the proof in \cite{forien-Zotsa2022stochastic}. This part is proved in two cases: $\mathbb{P}(\gamma_*^I=0)>0$ and $\mathbb{P}(\gamma_*^I=0)=0.$
				\subsubsection*{Case 1:$\mathbb{P}(\gamma_*^I=0)>0$}
				Using the fact that
				\[\int_{s}^{t}e^{-\nu((s+\eta,a])}\indic{a-s>\eta}\nu(da)=\indic{t-s>\eta}\left(1-e^{-\nu((s+\eta,t])}\right),\]
				we derive that,
				\begin{align*}
					\mathbb{P}(\gamma_*^I=0)\int_{0}^{t}\fF(s)\fS(s)ds&=	\int_{0}^{t}\E\left[\indic{t-s\leq\eta}\indic{\gamma_*^I=0}\right]\fF(s)\fS(s)ds\\
					&\hspace{1cm}+\int_{0}^{t}\E\left[e^{-\nu((s+\eta,t])}\indic{t-s>\eta}\indic{\gamma_*^I=0}\right]\fS(s)\fF(s)ds\\
					&\hspace{1cm}+\int_{0}^{t}\int_{s}^{t}\E\left[e^{-\nu((s+\eta,a])}\indic{a-s>\eta}\indic{\gamma_*^I=0}\right]\nu(da)\fS(s)\fF(s)ds\\
					&\leq1.
				\end{align*}
				where the last line follows from Proposition~\ref{Lem:mass-conserv} and the fact that $\gamma_*^I=0$ implies $\gamma^I(t)=0$ for all $t$. 
				
				Hence, for all $t\geq0,$
				\[\int_{0}^{t}\fF(s)\fS(s)ds\leq\frac{1}{\mathbb{P}(\gamma_*^I=0)}.\]
				Therefore
				\[\int_{0}^{\infty}\fF(s)\fS(s)ds<\infty.\]
				As a result, since $\lambda^I(t)\leq\lambda_*,\E\left[\lambda^I_0(t)\right]\to0$ and $\E\left[\lambda^I(t)\right]\to0$ as $t\to\infty,$ by dominated convergence theorem, we conclude that, as $t\to\infty,\,\fF(t)\to0$.
				
				\subsubsection*{Case 2:$\mathbb{P}(\gamma_*^I=0)=0$}As in \cite{forien-Zotsa2022stochastic} from \eqref{equiv-end-F-inf} and Assumption~\ref{ass-lbda-gamma}\ref{ass-lbda-gamma-ii}, we first note that,
				\begin{equation}\label{eq_equil1-a}
					\int_{0}^{\infty}\fF(s)ds<\infty\Longleftrightarrow \int_{0}^{\infty}\fF(s)\fS(s)ds<\infty.
				\end{equation}
				Thus from the proof of Case 1 and \eqref{eq_equil1-a}, it suffices to show that  $\int_{0}^{+\infty}\overline{\mathfrak{F}}(s)ds<+\infty$.
				We prove this claim by contradiction. Suppose that 
				\begin{equation}
					\int_{0}^{+\infty}\overline{\mathfrak{F}}(u)du=+\infty.\label{eq_lim_F}
				\end{equation}
				From  \eqref{equiv-end-F-inf}, using Fubuni's theorem, we obtain 
				\begin{align*}\label{eq_IF}
					\int_{0}^{t}\overline{\mathfrak{F}}(u)du
					&=\overline{I}(0)\int_{0}^{t}\overline{\lambda}_0(u)du+R_0\int_{0}^{t}\overline{\mathfrak{S}}(u)\overline{\mathfrak F}(u)du-\int_{0}^{t}\left(\int_{t-u}^{+\infty}\overline{\lambda}(s)ds\right)\overline{\mathfrak{S}}(u)\overline{\mathfrak F}(u)du.
				\end{align*}
				Consequently, 
				\begin{equation}\label{frac_lim}
					\frac{\int_{0}^{t}\overline{\mathfrak{S}}(u)\overline{\mathfrak F}(u)du}{\int_{0}^{t}\overline{\mathfrak{F}}(u)du}=\frac{1}{R_0}+\frac{\int_{0}^{t}\left(\int_{t-u}^{+\infty}\overline{\lambda}(s)ds\right)\overline{\mathfrak{S}}(u)\overline{\mathfrak F}(u)du}{R_0\int_{0}^{t}\overline{\mathfrak{F}}(u)du}-\frac{\overline{I}(0)\int_{0}^{t}\overline{\lambda}_0(u)du}{R_0\int_{0}^{t}\overline{\mathfrak{F}}(u)du}.
				\end{equation}
				From Assumption~\ref{ass-lbda-gamma}\ref{ass-lbda-gamma-ii} and \eqref{eq_lim_F}, we have
				\begin{equation}\label{eq_F_1}
					\frac{\int_{0}^{t}\overline{\lambda}_0(u)du}{\int_{0}^{t}\overline{\mathfrak{F}}(u)du}\xrightarrow[t\to+\infty]{}0.
				\end{equation}
				In addition, since $\int_{t}^{+\infty}\overline{\lambda}(s)ds\to0$ as $t\to+\infty$,  for $\epsilon>0$ there exists $T_\epsilon>0$ such that $\int_{T_\epsilon}^{+\infty}\overline{\lambda}(s)ds < \epsilon$. Hence, for $t\geq T_\epsilon$,
				\begin{align*}
					& \int_{0}^{t}\left(\int_{t-u}^{+\infty}\overline{\lambda}(s)ds\right)\overline{\mathfrak{S}}(u)\overline{\mathfrak F}(u)du \\
					&=\int_{0}^{t}\left(\int_{u}^{+\infty}\overline{\lambda}(s)ds\right)\overline{\mathfrak{S}}(t-u)\overline{\mathfrak F}(t-u)du\\
					&\leq\int_{0}^{T_\epsilon}\left(\int_{u}^{+\infty}\overline{\lambda}(s)ds\right)\overline{\mathfrak{S}}(t-u)\overline{\mathfrak F}(t-u)du+\epsilon\int_{T_\epsilon}^{t}\overline{\mathfrak{S}}(t-u)\overline{\mathfrak F}(t-u)du\\
					&\leq R_0\lambda_* T_\epsilon+\epsilon\int_{0}^{t}\overline{\mathfrak F}(u)du.
				\end{align*}
				Thus by \eqref{eq_lim_F}, we have 
				\begin{equation}\label{eq_F_2}
					\frac{\int_{0}^{t}\left(\int_{t-u}^{+\infty}\overline{\lambda}(s)ds\right)\overline{\mathfrak{S}}(u)\overline{\mathfrak F}(u)du}{\int_{0}^{t}\overline{\mathfrak{F}}(u)du}\xrightarrow[t\to+\infty]{}0.
				\end{equation}
				Hence under the assumption \eqref{eq_lim_F}, from \eqref{frac_lim}, \eqref{eq_F_1} and \eqref{eq_F_2}, we obtain 
				\begin{equation}
					\frac{\int_{0}^{t}\overline{\mathfrak{S}}(u)\overline{\mathfrak F}(u)du}{\int_{0}^{t}\overline{\mathfrak{F}}(u)du}\xrightarrow[t\to+\infty]{}\frac{1}{R_0}.\label{eq_frac_lim}
				\end{equation}
				On the other hand, from Proposition~\ref{Lem:mass-conserv} and the fact that $\gamma^I\leq\gamma_*^I$ we have
				\begin{multline}\label{eq_in_G}
					\int_{0}^{u} \E \left[\indic{u-s\leq\eta} \exp \left( - \gamma_\ast^I\int_{s}^{u}\overline{\mathfrak F}(r) dr \right) \right] \overline{\mathfrak S}(s) \overline{\mathfrak F}(s) ds\\
					+\int_{0}^{u}\E\left[e^{-\nu((s+\eta,u])}\indic{u-s\geq\eta}\exp\left(-\gamma_*^I\int_{s}^{u}\fF(r)dr\right)\right]\fF(s)\fS(s)ds<1.
				\end{multline}
				Next, multiplying by $\overline{\mathfrak F}(u)$ and integrating from $0$ to $t$ both sides of \eqref{eq_in_G},   we have
				\begin{multline*}
					\int_{0}^{t}\left(\int_{0}^{u} \E \left[ \indic{u-s\leq\eta}\overline{\mathfrak F}(u)\exp \left( - \gamma_\ast^I\int_{s}^{u}\overline{\mathfrak F}(r) dr \right) \right] \overline{\mathfrak S}(s) \overline{\mathfrak F}(s) ds\right)du\\
					+\int_{0}^{t}\left(\int_{0}^{u}\E\left[e^{-\nu((s+\eta,u])}\indic{u-s\geq\eta}\fF(u)\exp\left(-\gamma_*^I\int_{s}^{u}\fF(r)dr\right)\right]\fF(s)\fS(s)ds\right)du<\int_{0}^{t}\overline{\mathfrak F}(u)du,
				\end{multline*}
				and by Fubuni's theorem,
				\begin{multline*}
					\int_{0}^{t}\E \left[ \int_{s}^{t}\indic{u-s\leq\eta}\overline{\mathfrak F}(u)\exp \left( - \gamma_\ast^I\int_{s}^{u}\overline{\mathfrak F}(r) dr \right) du\right]\overline{\mathfrak S}(s) \overline{\mathfrak F}(s)ds\\
					+\int_{0}^{t}\E\left[\int_{s}^{t}e^{-\nu((s+\eta,u])}\indic{u-s\geq\eta}\fF(u)\exp\left(-\gamma_*^I\int_{s}^{u}\fF(r)dr\right)du\right]\fF(s)\fS(s)ds<\int_{0}^{t}\overline{\mathfrak F}(u)du,
				\end{multline*}
				Hence for $0<\epsilon\leq1$, since $\indic{\gamma_\ast\geq\epsilon}\leq1$,
				\begin{multline*}
					\int_{0}^{t}\E \left[\indic{\gamma_*^I\geq\epsilon} \int_{s}^{t}\indic{u-s\leq\eta}\overline{\mathfrak F}(u)\exp \left( - \gamma_\ast^I\int_{s}^{u}\overline{\mathfrak F}(r) dr \right) du\right]\overline{\mathfrak S}(s) \overline{\mathfrak F}(s)ds\\
					+\int_{0}^{t}\E\left[\indic{\gamma_*^I\geq\epsilon}\int_{s}^{t}e^{-\nu((s+\eta,u])}\indic{u-s\geq\eta}\fF(u)\exp\left(-\gamma_*^I\int_{s}^{u}\fF(r)dr\right)du\right]\fF(s)\fS(s)ds<\int_{0}^{t}\overline{\mathfrak F}(u)du,
				\end{multline*}
				from which, by integration by parts, we obtain 
				\begin{multline*}
					\E\left[\frac{1}{\gamma_*^I}\indic{\gamma_*^I\geq\epsilon}\right]\int_{0}^{t}\fF(s)\fS(s)ds-\int_{0}^{t}\E \left[\frac{1}{\gamma_*^I}\indic{\gamma_*^I\geq\epsilon}\indic{t-s\leq\eta}\exp \left( - \gamma_\ast^I\int_{s}^{t}\overline{\mathfrak F}(r) dr \right)\right]\overline{\mathfrak S}(s) \overline{\mathfrak F}(s)ds\\
					-\int_{0}^{t}\E \left[\frac{1}{\gamma_*^I}\indic{\gamma_*^I\geq\epsilon}e^{-\nu((s+\eta,t])}\indic{t-s>\eta}\exp \left( - \gamma_\ast^I\int_{s}^{t}\overline{\mathfrak F}(r) dr \right)\right]\overline{\mathfrak S}(s) \overline{\mathfrak F}(s)ds\\
					-\int_{0}^{t}\int_{s}^{t}\E \left[\frac{1}{\gamma_*^I}\indic{\gamma_*^I\geq\epsilon}e^{-\nu((s+\eta,u])}\indic{u-s>\eta}\exp \left( - \gamma_\ast^I\int_{s}^{u}\overline{\mathfrak F}(r) dr \right)\right]\nu(du)\overline{\mathfrak S}(s) \overline{\mathfrak F}(s)ds<\int_{0}^{t}\overline{\mathfrak F}(u)du.
				\end{multline*}
				Thus,
				\begin{multline}\label{eq_d}
					\E \left[\frac{1}{\gamma_*^I}\indic{\gamma_\ast^I\geq\epsilon}\right]\int_{0}^{t}\overline{\mathfrak S}(s) \overline{\mathfrak F}(s)ds<\int_{0}^{t}\overline{\mathfrak F}(u)du\\+\int_{0}^{t}\E \left[\frac{1}{\gamma_*^I}\indic{\gamma_*^I\geq\epsilon}\indic{t-s\leq\eta}\exp \left( - \gamma_\ast^I\int_{s}^{t}\overline{\mathfrak F}(r) dr \right)\right]\overline{\mathfrak S}(s) \overline{\mathfrak F}(s)ds
					\\+\int_{0}^{t}\E \left[\frac{1}{\gamma_*^I}\indic{\gamma_*^I\geq\epsilon}e^{-\nu((s+\eta,t])}\indic{t-s>\eta}\exp \left( - \gamma_\ast^I\int_{s}^{s+\eta}\overline{\mathfrak F}(r) dr \right)\right]\overline{\mathfrak S}(s) \overline{\mathfrak F}(s)ds\\
					+\int_{0}^{t}\int_{s}^{t}\E \left[\frac{1}{\gamma_*^I}\indic{\gamma_*^I\geq\epsilon}e^{-\nu((s+\eta,u])}\indic{u-s>\eta}\exp \left( - \gamma_\ast^I\int_{s}^{u}\overline{\mathfrak F}(r) dr \right)\right]\nu(du)\overline{\mathfrak S}(s) \overline{\mathfrak F}(s)ds
				\end{multline}
				Moreover, from Proposition~\ref{Lem:mass-conserv} using $\gamma^I\leq\gamma_*^I,$ as in \eqref{eq_in_G}, we obtain,
				\begin{multline*}
					\int_{0}^{t}\E \left[\frac{1}{\gamma_*^I}\indic{\gamma_*^I\geq\epsilon}\indic{t-s\leq\eta}\exp \left( - \gamma_\ast^I\int_{s}^{t}\overline{\mathfrak F}(r) dr \right)\right]\overline{\mathfrak S}(s) \overline{\mathfrak F}(s)ds
					\\+\int_{0}^{t}\E \left[\frac{1}{\gamma_*^I}\indic{\gamma_*^I\geq\epsilon}e^{-\nu((s+\eta,t])}\indic{t-s>\eta}\exp \left( - \gamma_\ast^I\int_{s}^{s+\eta}\overline{\mathfrak F}(r) dr \right)\right]\overline{\mathfrak S}(s) \overline{\mathfrak F}(s)ds\\
					+\int_{0}^{t}\int_{s}^{t}\E \left[\frac{1}{\gamma_*^I}\indic{\gamma_*^I\geq\epsilon}e^{-\nu((s+\eta,u])}\indic{u-s>\eta}\exp \left( - \gamma_\ast^I\int_{s}^{u}\overline{\mathfrak F}(r) dr \right)\right]\nu(du)\overline{\mathfrak S}(s) \overline{\mathfrak F}(s)ds\leq\frac{1}{\epsilon}
				\end{multline*}
				Consequently, under the assumption \eqref{eq_lim_F}, we have
				\begin{multline*}
					\frac{\int_{0}^{t}\E \left[\frac{1}{\gamma_*^I}\indic{\gamma_*^I\geq\epsilon}\indic{t-s\leq\eta}\exp \left( - \gamma_\ast^I\int_{s}^{t}\overline{\mathfrak F}(r) dr \right)\right]\overline{\mathfrak S}(s) \overline{\mathfrak F}(s)ds}{\int_{0}^{t}\overline{\mathfrak F}(u)du}
					\\+\frac{\int_{0}^{t}\E \left[\frac{1}{\gamma_*^I}\indic{\gamma_*^I\geq\epsilon}e^{-\nu((s+\eta,t])}\indic{t-s>\eta}\exp \left( - \gamma_\ast^I\int_{s}^{s+\eta}\overline{\mathfrak F}(r) dr \right)\right]\overline{\mathfrak S}(s) \overline{\mathfrak F}(s)ds}{\int_{0}^{t}\overline{\mathfrak F}(u)du}\\
					+\frac{\int_{0}^{t}\int_{s}^{t}\E \left[\frac{1}{\gamma_*^I}\indic{\gamma_*^I\geq\epsilon}e^{-\nu((s+\eta,u])}\indic{u-s>\eta}\exp \left( - \gamma_\ast^I\int_{s}^{u}\overline{\mathfrak F}(r) dr \right)\right]\nu(du)\overline{\mathfrak S}(s) \overline{\mathfrak F}(s)ds}{\int_{0}^{t}\overline{\mathfrak F}(u)du}\xrightarrow[t\to+\infty]{}0,
				\end{multline*}
				This implies that, by \eqref{eq_d}, for all $0<\epsilon\leq1$, 
				\begin{equation*}
					\limsup_{t\to+\infty}\frac{\int_{0}^{t}\overline{\mathfrak{S}}(u)\overline{\mathfrak F}(u)du}{\int_{0}^{t}\overline{\mathfrak{F}}(u)du}\leq\left(\E\left[\frac{1}{\gamma_\ast^I}\indic{\gamma_\ast^I\geq\epsilon}\right]\right)^{-1}. 
				\end{equation*}
				Since  $\mathbb{P}\left(\gamma_\ast^I=0\right)=0,$ we deduce by the monotone convergence theorem that
				\begin{equation*}
					\limsup_{t\to+\infty}\frac{\int_{0}^{t}\overline{\mathfrak{S}}(u)\overline{\mathfrak F}(u)du}{\int_{0}^{t}\overline{\mathfrak{F}}(u)du}\leq\left(\E\left[\frac{1}{\gamma_\ast^I}\right]\right)^{-1},
				\end{equation*}
				However, this contradicts \eqref{eq_frac_lim} since $R_0<\E\left[\frac{1}{\gamma_\ast^I}\right]$ by the assumption of Theorem~\ref{Th-LT-free}\ref{Th-lbda-gamma-ii}. 
				
				This completes the proof of the second case.
				
				\subsection{Proof of Theorem~\ref{end-th-persistence}}\label{sec-end-th-persistence}
				The goal of Theorem~\ref{end-th-persistence} is to prove that, if $\gamma_\ast$ is deterministic and positive, $R_0>\frac{1}{\gamma_\ast}$, the functions $\gamma^I_0$, $\gamma^I$ and $\gamma^V$ are non-decreasing and if there exists a positive decreasing function $h$ such that for all $0\leq s, t,\,\overline{\lambda}(t+s)\geq h(t)\overline{\lambda}(s),$ and $\overline{\mathfrak F}(0)>0,$ then there exists $c>0$ such that for all $t\geq0,\,\overline{\mathfrak F}(t)\geq c$. We adapt the proof of \cite{forien-Zotsa2022stochastic}.
				
				Let $\delta>0$ be such that $(1-\delta)\gamma_\ast>\frac{1}{R_0}$.
				From Assumption~\ref{ASS2}, there exists $s_1\geq0$ deterministic such that $\gamma_{0}^I(s_1)\wedge\gamma^I(s_1)\wedge\gamma^V(s_1)\geq(1-\delta)\gamma_\ast$ a.s.
				Let $\epsilon>0$ be such that $(1-\delta)\gamma_\ast>\frac{1+\epsilon}{R_0}$.
				Let $x$ be the solution of the following Volterra equation
				\begin{equation}\label{V1}
					x(t)= h(s_1+t)+(1+\epsilon)\int_{0}^{t}p(t-s)x(s)ds,\;t\geq0, 
				\end{equation}
				with
				\begin{align*}
					p(t)=\frac{\overline{\lambda}(t)}{R_0}.
				\end{align*}
				As 
				\[(1+\epsilon)\int_{0}^{+\infty}p(t)dt=1+\epsilon>1,\qquad\int_{0}^{+\infty}h(t)dt<+\infty\] where the integrability of $h$ results from Assumption~\ref{ASS3} and the integrability of $\overline\lambda$. Moreover, since $h$ and $p$ are bounded and non-negative, by \cite[Remark following Theorem~$4$, page~$253$]{feller2015integral}, $x(t)\to+\infty$ as $t\to+\infty$, hence there exists $s_2\geq0$ such that $x(s_2)>2$.
				
				Let $c_2>0$ be such that 
				\begin{equation}\label{inR_0}
					(1-\delta)c_3\gamma_\ast \exp(-c_2(s_1+s_2))\geq\frac{1+\epsilon}{R_0},
				\end{equation}
				for some deterministic constant $0< c_3<1$ to be chosen letter.
				
				Let \[c_1=\frac{1}{2}\min(c_2,\overline{\mathfrak F}(0)),\, \quad \text{ and } \quad c_0=\frac{c_1}{2}h(s_1+s_2),\]
				\[t_0=\inf\{t\geq0,\,\overline{\mathfrak F}(t)\leq c_0\},\quad \text{ and } \quad t_1=\sup\{t\leq t_0,\,\overline{\mathfrak F}(t)\geq c_1\}.\]  
				Since $h$ is decreasing and $h(0)=1$ we have $c_0<c_1$.
				We want to show that $t_0=\infty$, and will prove it by contradiction. 
				Let us thus suppose that $t_0<+\infty$, which implies that $t_1<+\infty$. From Assumption~\ref{ASS3}, by the continuity of $\overline{\mathfrak F}$ and the definition of $t_1$, for all $t\geq t_1$, we obtain
				\begin{align*}
					\overline{\mathfrak F}(t)&=\overline{\lambda}_0(t)\overline{I}(0)+\int_{0}^{t}\overline{\lambda}(t-s)\overline{\mathfrak F}(s)\overline{\mathfrak S}(s)ds\\
					&\geq\overline{\lambda}_0(t-t_1+t_1)\overline{I}(0)+\int_{0}^{t_1}\overline{\lambda}(t-t_1+t_1-s)\overline{\mathfrak F}(s)\overline{\mathfrak S}(s)ds\\
					&\geq h(t-t_1)\left(\overline{\lambda}_0(t_1)\overline{I}(0)+\int_{0}^{t_1}\overline{\lambda}(t_1-s)\overline{\mathfrak F}(s)\overline{\mathfrak S}(s)ds\right)\\
					&= h(t-t_1)\overline{\mathfrak F}(t_1)\\
					&\geq c_1 h(t-t_1). 
				\end{align*}
				The definition of $t_0$ and the continuity of $\overline{\mathfrak F},$ implies that $\overline{\mathfrak F}(t_0)\leq c_0.$ Combining with the last inequality evaluated at $t=t_0$, we have $c_0\geq c_1 h(t_0-t_1)$. Hence, by the definition of $c_0$ and the fact that $h$ is decreasing, we deduce that $t_0-t_1> s_1+s_2$. So $t_0>t_1+s_1+s_2$ and for all $t\in[t_1,t_0]$
				\begin{equation}\overline{\mathfrak F}(t)\leq c_1< c_2.\label{e-q}\end{equation}
				
				On the other hand, as $\gamma_0^I(t)\leq1$, $\gamma^I(t)\leq1$ and $\gamma^V(t)\leq1$,  we have, for all $t\geq t_1$,
				\begin{align*}
					\overline{\mathfrak S}(t)&=\E\left[\gamma_{0}^I(t)e^{-\nu((\eta_0,t])}\indic{t>\eta_0}\exp\left(-\int_{0}^{t}\gamma_{0}^I(r)\overline{\mathfrak F}(r)dr\right)\right] \\
					& \hspace{1.5cm} +\int_{0}^{t}\E\left[\gamma^I(t-s)e^{-\nu((s+\eta,t])}\indic{t-s>\eta}\exp\left(-\int_{s}^{t}\gamma^I(r-s)\overline{\mathfrak F}(r)dr\right)\right]\overline{\mathfrak F}(s)\overline{\mathfrak S}(s)ds\\
					&\hspace{1cm}+\int_{0}^{t}\big(1-\bar{I}(v)\big)e^{-\nu((v,t])}\E\left[\gamma^V(t-v)\exp\left(-\int_{v}^{t}\gamma^V(r-v)\fF(r)dr\right)\right]\nu(dv)\\
					&\geq e^{-\nu((t_1,t])}\exp\left(-\int_{t_1}^{t}\overline{\mathfrak F}(r)dr\right)\Bigg(\E\left[\gamma_{0}^I(t)e^{-\nu((\eta_0,t_1])}\indic{t_1>\eta_0}\exp\left(-\int_{0}^{t_1}\gamma_{0}^I(r)\overline{\mathfrak F}(r)dr\right)\right]\\
					&\hspace{1.5cm}+\int_{0}^{t_1}\E\left[\gamma^I(t-s)e^{-\nu((s+\eta,t_1])}\indic{t_1-s>\eta}\exp\left(-\int_{s}^{t_1}\gamma^I(r-s)\overline{\mathfrak F}(r)dr\right)\right]\overline{\mathfrak F}(s)\overline{\mathfrak S}(s)ds\\
					&\hspace{1cm}+\int_{0}^{t_1}\big(1-\bar{I}(v)\big)e^{-\nu((v,t_1])}\E\left[\gamma^V(t-v)\exp\left(-\int_{v}^{t_1}\gamma^V(r-v)\fF(r)dr\right)\right]\nu(dv)\Bigg).
				\end{align*} 
				But, as for $t\in[t_1+s_1,t_1+s_1+s_2]$ and $s\in[0,t_1],$
				\[\gamma_{0}^I(t)\wedge\gamma^I(t-s)\wedge\gamma^V(t-s)\geq\gamma_{0}^I(s_1)\wedge\gamma^I(s_1)\wedge\gamma^V(s_1)\geq(1-\delta)\gamma_\ast,\] we deduce that, for all $t\in[t_1+s_1,t_1+s_1+s_2]$,
				
				\begin{align*}
					\overline{\mathfrak{S}}(t)&\geq(1-\delta)c_3\gamma_\ast \exp\left(-\int_{t_1}^{t}\overline{\mathfrak F}(r)dr\right),
				\end{align*} 
				where we take
				\begin{align*}
					0<c_3&=e^{-\nu((t_1,t_1+s_1+s_2])}\Bigg(\E\left[e^{-\nu((\eta_0,t_1])}\indic{t_1>\eta_0}\exp\left(-\int_{0}^{t_1}\gamma_{0}^I(r)\overline{\mathfrak F}(r)dr\right)\right]\\
					&\hspace{1.5cm}+\int_{0}^{t_1}\E\left[e^{-\nu((s+\eta,t_1])}\indic{t_1-s>\eta}\exp\left(-\int_{s}^{t_1}\gamma^I(r-s)\overline{\mathfrak F}(r)dr\right)\right]\overline{\mathfrak F}(s)\overline{\mathfrak S}(s)ds\\
					&\hspace{1cm}+\int_{0}^{t_1}\big(1-\bar{I}(v)\big)e^{-\nu((v,t_1])}\E\left[\exp\left(-\int_{v}^{t_1}\gamma^V(r-v)\fF(r)dr\right)\right]\nu(dv)\Bigg)<1,
				\end{align*}
				where the inequality follows from Remark~\ref{RQ-bound-phi} for appropriate function $\varphi$. 
				
				Moreover, since from \eqref{e-q} for $t\in[t_1+s_1,t_1+s_1+s_2],\,\overline{\mathfrak F}(t)\leq c_2$,
				\begin{align*}
					\overline{\mathfrak{S}}(t)&\geq(1-\delta)c_3\gamma_\ast \exp(-c_2(s_2+s_1)).
				\end{align*} 
				
				Then from \eqref{inR_0}
				\begin{equation}
					\forall t\in[t_1+s_1,t_1+s_1+s_2],\; \quad \overline{\mathfrak{S}}(t)\geq\frac{1+\epsilon}{R_0}.\label{e-q1}
				\end{equation}
				Let $y(t)=\overline{\mathfrak F}(t+t_1+s_1)$ and define $g$ as follows:
				\begin{equation*}
					g(t)=\overline{I}(0)\overline{\lambda}_0(t_1+s_1+t)+\int_{0}^{t_1+s_1}\overline{\lambda}(t_1+s_1+t-s)\overline{\mathfrak F}(s)\overline{\mathfrak S}(s)ds,
				\end{equation*}
				where we recall that
				\begin{align*}
					p(t)=\frac{\overline{\lambda}(t)}{R_0}.
				\end{align*}
				Then using \eqref{e-q1} for any $t\geq0$,
				\begin{align*}
					y(t)&\geq g(t)+(1+\epsilon)\int_{t_1+s_1}^{t_1+s_1+t}p(t_1+s_1+t-s)y(s-t_1-s_1)ds\\
					&= g(t)+(1+\epsilon)\int_{0}^{t}p(t-s)y(s)ds.
				\end{align*}
				However, from Assumption~\ref{ASS3} we deduce that
				\begin{equation*}
					g(t)\geq \overline{\mathfrak F}(t_1) h(s_1+t)
				\end{equation*}
				and as $\overline{\mathfrak F}(t_1)= c_1$ by continuity, we deduce that 
				\begin{equation}\label{eq-last}
					y(t)\geq c_1 h(s_1+t)+(1+\epsilon)\int_{0}^{t}p(t-s)y(s)ds.
				\end{equation}
				Thus, applying Lemma~9.8.2 in \cite{gripenberg_volterra_1990} to $ -y(t) $ in \eqref{eq-last} with the convolution kernel $ k(s) = (1+\varepsilon)p(s) $, we obtain
				\begin{equation*}
					y(t)\geq c_1 x(t)
				\end{equation*}
				where $x$ is given by \eqref{V1}.
				However $x(s_2)>2$. Hence $\overline{\mathfrak F}(t_1+s_1+s_2)>2c_1>c_1$ and $t_0\geq t_1+s_1+s_2$, this contradicts the definition of $t_1$. Hence $t_0=+\infty$. This concludes the proof. 
				
				\section{Alternative expression of $\big(\mathbb{E}\left[\gamma_{A(t),\mathcal{N}(t)}(t)\right],\mathbb{E}\left[\lambda_{A(t),\mathcal{N}(t)}(t)\right]\big)$}\label{sec-exp-F-G}
				In this section we establish the following Lemmas, which will be used to derive that the pair $(x(t),y(t)):=\big(\mathbb{E}\left[\gamma_{A(t),\mathcal{N}(t)}(t)\right],\mathbb{E}\left[\lambda_{A(t),\mathcal{N}(t)}(t)\right]\big)$ satisfies the set of equations~\eqref{eq:F-G-x}-\eqref{eq:F-G-y}.
				
				We recall that $(T^V_j)_j$ is the jump times of Poisson measure $Q^V$ with intensity measure $\nu$. For $t\geq0,$ $\mathcal{N}_t=Q^V([0,t]).$ 
				\subsection{Case of $y$}\label{sec-equiv-fF}
				We have
				\begin{align*}
					y(t)&:=\E[\lambda_{A(t),\mathcal{N}_t}(t)]=\E\Big[\lambda_{A(t)}^I(t-\tau_{A(t)})\indic{ T^V_{\mathcal{N}_t}-\tau_{A(t)}\leq\eta_{A(t)}}\Big]\\
					&=\E\Big[\indic{\mathcal{N}_t=0}\lambda_{A(t)}^I(t-\tau_{A(t)})\indic{ T^V_{\mathcal{N}_t}-\tau_{A(t)}\leq\eta_{A(t)}}\Big]+\E\Big[\sum_{j\geq1}\indic{\mathcal{N}_t=j}\lambda_{A(t)}^I(t-\tau_{A(t)})\indic{ T^V_{j}-\tau_{A(t)}\leq\eta_{A(t)}}\Big]\\
					&=\E\Big[\indic{\mathcal{N}_t=0}\lambda_{A(t)}^I(t-\tau_{A(t)})\Big]+\E\Big[\sum_{j\geq1}\indic{\mathcal{N}_t=j}\lambda_{A(t)}^I(t-\tau_{A(t)})\indic{\tau_{A(t)}\leq T^V_j\leq t}\indic{ T^V_{j}-\tau_{A(t)}\leq\eta_{A(t)}}\Big]\\
					&\hspace{2cm}+\E\Big[\sum_{j\geq1}\indic{\mathcal{N}_t=j}\lambda_{A(t)}^I(t-\tau_{A(t)})\indic{T^V_j\leq \tau_{A(t)}\leq t}\indic{ T^V_{j}-\tau_{A(t)}\leq\eta_{A(t)}}\Big].
				\end{align*}

				We introduce the following lemmas, the proofs of which are provided in Sections~\ref{Lem:proofs-I}, \ref{Lem:proofs-II}, and \ref{Lem:proofs-III}, respectively.
				
				\begin{lemma}\label{Lem:exp-I-y}
					\begin{align*}
						&\E\left[\indic{\mathcal{N}_t=0}\lambda_{A(t)}^I(t-\tau_{A(t)})\right]\\
						&=\exp\left(-\nu([0,t])\right)\Bigg(\E\left[\gamma_{0}^I(t)\exp\left(-\int_{0}^{t}\gamma_0^I(r)y(r)dr\right)\right]\\
						&\hspace{1cm}+\int_{0}^{t}\E\left[\gamma^I(t-s)\exp\left(-\int_{s}^{t}\gamma^I(r-s)y(r)dr\right)\right]x(s)y(s)ds\Bigg)
					\end{align*}
				\end{lemma}
				
				\begin{lemma}\label{lem:Inf}
					\begin{multline*}
						\E\Big[\sum_{j\geq1}\indic{\mathcal{N}_t=j}\lambda_{A(t)}^I(t-\tau_{A(t)})\indic{\tau_{A(t)}\leq T^V_j\leq t}\indic{ T^V_{j}-\tau_{A(t)}\leq\eta_{A(t)}}\Big]\\
						\begin{aligned}
							&=\int_{0}^{t}e^{-\nu((v,t])}\E\left[\lambda_{0}^I(t)\indic{v\leq\eta_{0}}\exp\left(-\int_{0}^{t}\gamma_{0}^I(r)y(r)dr\right)\right]\nu(dv)\\
							&\hspace{1cm}+\int_{0}^{t}\int_{w}^{t}e^{-\nu((v,t])}\E\left[\lambda^I(t-w)\indic{v-w\leq\eta}\exp\left(-\int_{w}^{t}\gamma^I(r-w)y(r)dr\right)\right]\nu(dv)x(w)y(w)dw\\
						\end{aligned}
					\end{multline*}
				\end{lemma}

				\begin{lemma}\label{Lem:exp-II-y}
					\begin{multline*}
						\E\Big[\sum_{j\geq1}\indic{\mathcal{N}_t=j}\lambda_{A(t)}^I(t-\tau_{A(t)})\indic{T^V_j\leq \tau_{A(t)}\leq t}\indic{ T^V_{j}-\tau_{A(t)}\leq\eta_{A(t)}}\Big]\\
						\begin{aligned}
							&=\int_{0}^{t}\int_{0}^{w}e^{-\nu((v,t])}\E\left[\lambda^I(t-w)\exp\left(-\int_{w}^{t}\gamma^I(r-w)y(r)dr\right)\right]\nu(dv)x(w)y(w)dw\\
						\end{aligned}
					\end{multline*}
				\end{lemma}

				\subsection{Case of $x$}
				We have
				\begin{align*}
					x(t)&:=\E[\gamma_{A(t),\mathcal{N}_t}(t)]=\E\Big[\gamma_{A(t)}^I(t-\tau_{A(t)})\indic{ T^V_{\mathcal{N}_t}-\tau_{A(t)}\leq\eta_{A(t)}}\Big]+\E\Big[\gamma^V_{\mathcal{N}_t}(t-T^V_{\mathcal{N}_t})\indic{T_{\mathcal{N}_t}^V-\tau_{A(t)}>\eta_{A(t)}}\Big]\\
					&=\E\Big[\gamma_{A(t)}^I(t-\tau_{A(t)})\indic{ T^V_{\mathcal{N}_t}-\tau_{A(t)}\leq\eta_{A(t)}}\Big]+\E\Big[\sum_{j\geq1}\indic{\mathcal{N}_t=j}\gamma_{j}^V(t-T^V_j)\indic{ T^V_{j}-\tau_{A(t)}>\eta_{A(t)}}\Big].
				\end{align*}
				
				We first note that, from Section~\ref{sec-equiv-fF} we easily derive the following result.
				\begin{lemma}\label{Lem:-I-x}
					\begin{equation*}
						\begin{aligned}
							&\E\Big[\gamma_{A(t)}^I(t-\tau_{A(t)})\indic{ T^V_{\mathcal{N}_t}-\tau_{A(t)}\leq\eta_{A(t)}}\Big]\\
							&=e^{-\nu([0,t])}\Bigg(\E\left[\gamma_{0}^I(t)\exp\left(-\int_{0}^{t}\gamma_{0}^I(r)y(r)dr\right)\right]\\
							&\hspace{1cm}+\int_{0}^{t}\E\left[\gamma^I(t-s)\exp\left(-\int_{s}^{t}\gamma^I(r-s)y(r)dr\right)\right]y(s)x(s)ds\Bigg)\\
							&+\int_{0}^{t}e^{-\nu((v,t])}\E\left[\gamma_{0}^I(t)\indic{v\leq\eta_{0}}\exp\left(-\int_{0}^{t}\gamma_{0}^I(r)y(r)dr\right)\right]\nu(dv)\\
							&\hspace{0.5cm}+\int_{0}^{t}\int_{s}^{t}e^{-\nu((v,t])}\E\left[\gamma^I(t-s)\indic{v-s\leq\eta}\exp\left(-\int_{s}^{t}\gamma^I(r-s)y(r)dr\right)\right]\nu(dv)x(s)y(s)ds\\
							&+\int_{0}^{t}\int_{0}^{s}e^{-\nu((v,t])}\E\left[\gamma^I(t-s)\exp\left(-\int_{s}^{t}\gamma^I(r-s)y(r)dr\right)\right]\nu(dv)x(s)y(s)ds.
						\end{aligned}
					\end{equation*}
				\end{lemma}
				It remains to compute 
				\begin{align}\label{eq-gV}
					&\E\Big[\sum_{j\geq1}\indic{\mathcal{N}_t=j}\gamma_{j}^V(t-T^V_j)\indic{ T^V_{j}-\tau_{A(t)}>\eta_{A(t)}}\Big]\non\\
					&=\E\Big[\indic{A(t)=0}\sum_{j\geq1}\indic{\mathcal{N}_t=j}\gamma_{j}^V(t-T^V_j)\indic{ T^V_{j}>\eta_{0}}\Big]+\sum_{i\geq1}\E\Big[\indic{A(t)=i}\sum_{j\geq1}\indic{\mathcal{N}_t=j}\gamma_{j}^V(t-T^V_j)\indic{ T^V_{j}-\tau_{i}>\eta_{i}}\Big]
				\end{align}
				We present the following results, the proofs of which are provided in Section~\ref{Lem:proofs-IV} and Section~\ref{Lem:proofs-V}, respectively.
				
				\begin{lemma}\label{Lem:-II-x}
					\begin{multline}
						\E\Big[\indic{A(t)=0}\sum_{j\geq1}\indic{\mathcal{N}_t=j}\gamma_{j}^V(t-T^V_j)\indic{ T^V_{j}>\eta_{0}}\Big]\\
						\begin{aligned}
							&=\int_0^te^{-\nu((v,t])}\E\Big[\indic{A(v)=0}\indic{ v>\eta_{0}}\Big]\E\left[\gamma^V(t-v)\exp\left(-\int_{v}^{t}\gamma^V(r-v)y(r)dr\right)\right]\nu(dv)
						\end{aligned}
					\end{multline}
				\end{lemma}
				\begin{lemma}\label{Lem:-III-x}
					\begin{multline}
						\sum_{i\geq1}\E\Big[\indic{A(t)=i}\sum_{j\geq1}\indic{\mathcal{N}_t=j}\gamma_{j}^V(t-\tau_{i})\indic{ T^V_{j}-\tau_{i}>\eta_{i}}\Big]\\
						\begin{aligned}
							&=\int_{0}^{t}e^{-\nu((v,t])}\E\Big[\sum_{i\geq1}\indic{A(v)=A(\tau_i)}\indic{ v-\tau_i>\eta_{i}}\Big]
							\E\left[\gamma^V(t-v)\exp\left(-\int_{v}^{t}\gamma^V(r-v)y(r)dr\right)\right]\nu(dv).
						\end{aligned}
					\end{multline}
				\end{lemma}
				The challenge lies to compute the following two expectations: 
				\[\E\Big[\indic{A(v)=0}\indic{ v>\eta_{0}}\Big] \text{ and } 
				\E\Big[\sum_{i\geq1}\indic{A(v)=A(\tau_i)}\indic{ v-\tau_i>\eta_{i}}\Big].\]
				Evaluating these terms separately is quite difficult and would require the results on Poisson integrals with non-predictable integrands established in Section~\ref{sec-poiss-integral}. However, to obtain a more tractable expression, we note that:  
				\begin{align}\label{eq-x-EE}
					&\E\Big[\indic{A(v)=0}\indic{ v>\eta_{0}}\Big]+\E\Big[\sum_{i\geq1}\indic{A(v)=A(\tau_i)}\indic{ v-\tau_i>\eta_{i}}\Big]\non\\
					&=\E\Big[\indic{v-\tau_{A(v)}\geq\eta_{A(v)}}\Big]\non\\
					&=1-\E\Big[\indic{v-\tau_{A(v)}\leq\eta_{A(v)}}\Big]\non\\
					&=1-\E\Big[\indic{v\leq\eta_{0}}\exp\left(-\int_{0}^{v}\gamma_0(r)y(r)dr\right)\Big]\non\\
					&\hspace{2cm}-\int_{0}^{v}\E\Big[\indic{v-s\leq\eta}\exp\left(-\int_{s}^{v}\gamma(r-s)y(r)dr\right)\Big]x(s)y(s)ds\non\\
					&:=\mathcal{K}(x,y)(v)
				\end{align}
				where the penultimate equality follows directly from the framework established in Section~\ref{sec-equiv-fF}.
				
				By combining Lemma~\ref{Lem:-II-x} and ~\ref{Lem:-III-x} and expression~\eqref{eq-x-EE}, from~\eqref{eq-gV} we derive the following Lemma.
				\begin{lemma}\label{Lem-gam-II}
					\begin{equation*}
						\begin{aligned}
							&\E\Big[\sum_{j\geq1}\indic{\mathcal{N}_t=j}\gamma_{j}^V(t-T^V_j)\indic{ T^V_{j}-\tau_{A(t)}>\eta_{A(t)}}\Big]\\
							&=\int_{0}^{t}e^{-\nu((v,t])}\mathcal{K}(x,y)(v)
							\E\left[\gamma^V(t-v)\exp\left(-\int_{v}^{t}\gamma^V(r-v)y(r)dr\right)\right]\nu(dv).
						\end{aligned}
					\end{equation*}
				\end{lemma}
				
				\subsection{Proof of the Lemmas}\label{Lem:proofs}
				In this section, we introduce the following notation. We set,
				\begin{equation}
					\mathcal{F}_t = \sigma\left( (\lambda_i^I, \gamma_i^I)_{0 \leq i \leq A(t')}, A(t'), t' \leq t \right) \quad \text{and} \quad \mathcal{G}_t = \sigma\left( \mathcal{N}_v,\,(\gamma^V_{j})_{1\leq j<\mathcal{N}_v},\,v \leq t \right).
				\end{equation}
				We denote the conditional expectation with respect to the $i$-th infection time by $\mathbb{E}_i[\cdot] := \mathbb{E}\left[ \cdot \, \big| \, \mathcal{F}_{\tau_{i}} \right]$. Furthermore, for $0 \leq s \leq t$ and a measurable function $\phi$, we set:
				\begin{equation}
					P(s, t, \phi) = \int_{s}^{t} \int_{0}^{\phi(r-s)y(r)} Q(dr, du).
				\end{equation}
				
				\subsubsection{proof of Lemma~\ref{Lem:exp-I-y}}\label{Lem:proofs-I}
				
				We denote by $(\tau_i)_i$ the jump time of process $A$.
				
				We have,
				\begin{align*}
					&\E\left[\lambda_{A(t)}^I(t-\tau_{A(t)})\indic{\mathcal{N}_t=0}\right]\\
					&=\E\left[\indic{\mathcal{N}_t=0}\indic{A(t)=0}\lambda_{0}^I(t)\right]+\sum_{i\geq1}\E\left[\indic{\mathcal{N}_t=0}\indic{\tau_i\leq t}\indic{A(t)=i}\lambda_{i}^I(t-\tau_{i})\right].
				\end{align*}
				By definition, for all $t\geq0,$ when $\mathcal{N}_t=0,\,\forall r\in[0,t],\,\gamma_{A(r),\mathcal{N}_r}(r)=\gamma_{i,0}(r)=\gamma^I_{i}(r-\tau_{i}).$ Consequently, by definition of the process $A,$ and the fact that the pair $(Q,\gamma_i^I)$ and $Q^V$ are independent, it follows that,
				\begin{align*}
					&\mathbb{P}\left(A(t)=i,\,Q^V([0,t])=0\big|\mathcal{F}_{\tau_i}\right)\indic{\tau_i\leq t}\\
					&=\mathbb{P}\left(Q\left((s,u)\in\R^2_+,\,\tau_i\leq s\leq t,\,\gamma_i^I(r-\tau_i)y(r)\geq u\right)=0,\,Q^V([0,t])=0\big|\mathcal{F}_{\tau_i}\right)\indic{\tau_i\leq t}\\
					&=e^{-\nu([0,t])}\mathbb{P}\left(Q\left((s,u)\in\R^2_+,\,\tau_i\leq s\leq t,\,\gamma_i^I(r-\tau_i)y(r)\geq u\right)=0\big|\mathcal{F}_{\tau_i}\right)\indic{\tau_i\leq t}\\
					&=e^{-\nu([0,t])}\exp\left(-\int_{\tau_i}^{t}\gamma_i^I(r-\tau_i)y(r)dr\right)\indic{\tau_i\leq t},
				\end{align*}
				since $Q\big|_{(\tau_i,t]}$ and $\mathcal{F}_{\tau_i}$ are independent.
				
				Thus, as $(\lambda_0^I,\gamma_0^I)$ and $Q$ are independent, it follows that,
				\begin{align*}
					\E\left[\lambda_{A(t)}^I(t-\tau_{A(t)})\indic{\mathcal{N}_t=0}\right]
					&=e^{-\nu([0,t])}\E\left[\lambda_{0}^I(t)\exp\left(-\int_{0}^{t}\gamma^I_0(r)y(r)dr\right)\right]\\
					&\hspace{0.5cm}+e^{-\nu([0,t])}\sum_{i\geq1}\E\left[\indic{\tau_i\leq t}\lambda_{i}^I(t-\tau_{i})\exp\left(-\int_{\tau_i}^{t}\gamma^I_i(r-\tau_i)y(r)dr\right)\right].
				\end{align*}
				Since $(\lambda^I_i,\gamma^I_i)$ and $\tau_i$ are independent, recalling that the law of $(\lambda^I,\gamma^I)$ is denoted by $\mu^I$, we obtain
				\begin{align*}
					&\E\left[\lambda_{A(t)}^I(t-\tau_{A(t)})\indic{\mathcal{N}_t=0}\right]=e^{-\nu([0,t])}\E\left[\lambda_{0}^I(t)\exp\left(-\int_{0}^{t}\gamma^I_0(r)y(r)dr\right)\right]\\&+e^{-\nu([0,t])}\sum_{i\geq1}\E\left[\indic{\tau_i\leq t}\int_{D^2}\lambda^I(t-\tau_{i})\exp\left(-\int_{\tau_i}^{t}\gamma^I(r-\tau_i)y(r)dr\right)\mu^I(d\lambda^I,d\gamma^I)\right]
				\end{align*}
				In addition as $(\tau_i)$ are the jump time of process $A,$ by Fubuni's Theorem's we have,
				\begin{multline*}
					\E\left[\lambda_{A(t)}^I(t-\tau_{A(t)})\indic{\mathcal{N}_t=0}\right]\\
					\begin{aligned}
						&=e^{-\nu([0,t])}\E\left[\lambda_{0}^I(t)\exp\left(-\int_{0}^{t}\gamma^I_0(r)y(r)dr\right)\right]\\&\quad+e^{-\nu([0,t])}\int_{D^2}\E\left[\int_{0}^{t}\lambda^I(t-s)\exp\left(-\int_{s}^{t}\gamma^I(r-s)y(r)dr\right)A(ds)\right]\mu^I(d\lambda^I,d\gamma^I)\\
						&=e^{-\nu([0,t])}\E\left[\lambda_{0}^I(t)\exp\left(-\int_{0}^{t}\gamma^I_0(r)y(r)dr\right)\right]\\&\quad+e^{-\nu([0,t])}\int_{D^2}\E\left[\int_{0}^{t}\lambda^I(t-s)\exp\left(-\int_{s}^{t}\gamma^I(r-s)y(r)dr\right)\gamma_{A(s),\mathcal{N}_s}(s)y(s)ds\right]\mu^I(d\lambda^I,d\gamma^I)\\
						&=e^{-\nu([0,t])}\E\left[\lambda_{0}^I(t)\exp\left(-\int_{0}^{t}\gamma^I_0(r)y(r)dr\right)\right]\\&\quad+e^{-\nu([0,t])}\int_{0}^{t}\int_{D^2}\lambda^I(t-s)\exp\left(-\int_{s}^{t}\gamma^I(r-s)y(r)dr\right)\mu^I(d\lambda^I,d\gamma^I)\E\left[\gamma_{A(s),\mathcal{N}_s}(s)\right]y(s)ds\\
					\end{aligned}
				\end{multline*}
				\subsubsection{proof of Lemma~\ref{lem:Inf}}\label{Lem:proofs-II}
				We have
				\begin{align*}
					&\E\Big[\sum_{j\geq1}\indic{\mathcal{N}_t=j}\lambda_{A(t)}^I(t-\tau_{A(t)})\indic{\tau_{A(t)}\leq T^V_j\leq t}\indic{ T^V_{j}-\tau_{A(t)}\leq\eta_{A(t)}}\Big]\\
						&=\sum_{j\geq1}\E\Big[\indic{A(t)=0}\indic{\mathcal{N}_t=j}\lambda_{0}^I(t)\indic{ T^V_{j}\leq\eta_0}\Big]\\
						&\hspace{1cm}+\sum_{j\geq1}\sum_{i\geq1}\E\left[\indic{A(t)=i}\indic{\mathcal{N}_t=j}\lambda_{i}^I(t-\tau_{i})\indic{\tau_i\leq t}\indic{T^V_{j}-\tau_{i}\leq\eta_{i}}\right]\\		
						&=\sum_{j\geq1}\E\left[\lambda_{0}^I(t)\mathbb{P}\big(A(t)=0,\,Q^V((T^V_j,t])=0\big|\mathcal{F}_0\vee\mathcal{G}_{T^V_j}\big) \indic{T^V_j\leq\eta_{0}}\indic{T^V_j\leq t}\right]\\
						&\hspace{1cm}+\sum_{i\geq1}\sum_{j\geq1}\E\left[\lambda_{i}^I(t-\tau_i)\mathbb{P}\big(A(t)=i,\,Q^V((T^V_j,t])=0\big|\mathcal{F}_{\tau_i}\vee\mathcal{G}_{T^V_j}\big)\indic{T^V_j-\tau_i\leq\eta_{i}}\indic{\tau_i\leq T^V_j\leq t}\right]
				\end{align*}
				Note that as $Q^V((T^V_j,t])=0$ and $T^V_j-\tau_i\leq\eta_{i},$ by definition of $\forall r\in[\tau_i,t],\,\gamma_{A(r),\mathcal{N}_r}(r)=\gamma_{i,j}(r)=\gamma_i^I(r-\tau_i)$. Then, by definition of the process $A,$ and the fact that $Q$ and $Q^V$ are independent, 
				\begin{align*}
					&\mathbb{P}\big(A(t)=i,\,Q^V((T^V_j,t])=0\big|\mathcal{F}_{\tau_i}\vee\mathcal{G}_{T^V_j}\big)\indic{T^V_j-\tau_i\leq \eta_i}\indic{\tau_i\leq T^V_j\leq t}\\
					&=\mathbb{P}\left(Q\left((s,u)\in\R^2_+,\,\tau_i\leq s\leq t,\,\gamma_i^I(r-\tau_i)y(r)\geq u\right)=0,\,Q^V((T^V_j,t])=0\big|\mathcal{F}_{\tau_i}\vee\mathcal{G}_{T^V_j}\right)\\
					&\hspace{3cm}\times\indic{T^V_j-\tau_i\leq \eta_i}\indic{\tau_i\leq T^V_j\leq t}\\
					&=e^{-\nu((T^V_j,t])}\mathbb{P}\left(Q\left((s,u)\in\R^2_+,\,\tau_i\leq s\leq t,\,\gamma_i^I(r-\tau_i)y(r)\geq u\right)=0\big|\mathcal{F}_{\tau_i}\right)\indic{T^V_j-\tau_i\leq \eta_i}\indic{\tau_i\leq T^V_j\leq t}\\
					&=e^{-\nu((T^V_j,t])}\exp\left(-\int_{\tau_i}^{t}\gamma_i^I(r-\tau_i)y(r)dr\right)\indic{T^V_j-\tau_i\leq \eta_i}\indic{\tau_i\leq T^V_j\leq t},
				\end{align*}
				since $Q\big|_{(\tau_i,t]}$ and $\mathcal{F}_{\tau_i}$ are independent.
				
				As a result, we have the following equality,
				\begin{multline*}
					\begin{aligned}
						&=\sum_{j\geq1}\E\left[e^{-\nu((T^V_j,t])}\indic{T^V_j\leq\eta_{0}}\indic{T^V_j\leq t}\lambda_{0}^I(t)\exp\left(-\int_{0}^{t}\gamma_0^I(r)y(r)dr\right)\right]\\
						&\hspace{1cm}+\sum_{i\geq1}\sum_{j\geq1}\E\left[e^{-\nu((T^V_j,t])}\indic{T^V_j-\tau_i\leq\eta_{i}}\indic{\tau_{i}\leq T^V_j\leq t}\lambda_{i}^I(t-\tau_i)\exp\left(-\int_{\tau_i}^{t}\gamma_i^I(r-\tau_i)y(r)dr\right)\right]\\
						&=\E\left[\int_{0}^t e^{-\nu((v,t])}\indic{v\leq\eta_{0}}\lambda_{0}^I(t)\exp\left(-\int_{0}^{t}\gamma_0^I(r)y(r)dr\right)Q^V(dv)\right]\\
						&\hspace{1cm}+\sum_{i\geq1}\E\left[\int_{\tau_{i}}^{t}e^{-\nu((v,t])}\lambda_{i}^I(t-\tau_i)\indic{v-\tau_i\leq\eta_{i}}\exp\left(-\int_{\tau_i}^{t}\gamma_i^I(r-\tau_i)y(r)dr\right)Q^V(dv)\right]\\
				&=\int_{0}^{t}e^{-\nu((v,t])}\E\left[\lambda_{0}^I(t)\indic{v\leq\eta_{0}}\exp\left(-\int_{0}^{t}\gamma_{0}^I(r)y(r)dr\right)\right]\nu(dv)\\
				&\hspace{1cm}+\sum_{i\geq1}\E\left[\indic{\tau_{i}\leq t}\int_{\tau_{i}}^{t}e^{-\nu((v,t])}\lambda_{i}^I(t-\tau_i)\indic{v-\tau_i\leq\eta_{i}}\exp\left(-\int_{\tau_i}^{t}\gamma_{i}^I(r-\tau_i)y(r)dr\right)\nu(dv)\right].	
				\end{aligned}
			\end{multline*}
		Since $(\lambda_i^I,\gamma_i^I)$ and $\tau_i$ are independent, recalling that the law of $(\lambda^I,\gamma^I)$ is denoted by $\mu^I$, we obtain
	\begin{multline*}
	\begin{aligned}
	&=\int_{0}^{t}e^{-\nu((v,t])}\E\left[\lambda_{0}^I(t)\indic{v\leq\eta_{0}}\exp\left(-\int_{0}^{t}\gamma_{0}^I(r)y(r)dr\right)\right]\nu(dv)\\
	&\hspace{0.2cm}+\sum_{i\geq1}\int_{D^2}\E\Bigg[\indic{\tau_{i}\leq t}\int_{\tau_{i}}^{t}e^{-\nu((v,t])}\lambda^I(t-\tau_i)\indic{v-\tau_i\leq\eta}\\
	&\hspace{3cm}\times\exp\left(-\int_{\tau_i}^{t}\gamma^I(r-\tau_i)y(r)dr\right)\nu(dv)\Bigg]\mu^I(d\lambda^I,d\gamma^I)
	\end{aligned}
\end{multline*}
Moreover, as $(\tau_i)_i$ is the jump time of process $A$, we have,
\begin{multline*}
\begin{aligned}
&=\int_{0}^{t}e^{-\nu((v,t])}\E\left[\lambda_{0}^I(t)\indic{v\leq\eta_{0}}\exp\left(-\int_{0}^{t}\gamma_{0}^I(r)y(r)dr\right)\right]\nu(dv)\\
&\hspace{0.2cm}+\int_{D^2}\E\Bigg[\int_{0}^{t}\int_{w}^{t}e^{-\nu((v,t])}\lambda^I(t-w)\indic{v-w\leq\eta}\\
&\hspace{3cm}\times\exp\left(-\int_{w}^{t}\gamma^I(r-w)y(r)dr\right)\nu(dv)A(dw)\Bigg]\mu^I(d\lambda^I,d\gamma^I).
\end{aligned}
\end{multline*}
By a martingale argument as in \cite[Section~$6$, page~$26$]{forien-Zotsa2022stochastic} we obtain,
\begin{multline*}
\begin{aligned}
&=\int_{0}^{t}e^{-\nu((v,t])}\E\left[\lambda_{0}^I(t)\indic{v\leq\eta_{0}}\exp\left(-\int_{0}^{t}\gamma_{0}^I(r)y(r)dr\right)\right]\nu(dv)\\
&\hspace{0cm}+\int_{D^2}\E\left[\int_{0}^{t}\int_{w}^{t}e^{-\nu((v,t])}\lambda^I(t-w)\indic{v-w\leq\eta}\exp\left(-\int_{w}^{t}\gamma^I(r-w)y(r)dr\right)\nu(dv)\right.\\&\hspace{8cm}\left.\gamma_{A(w),\mathcal{N}_w}(w)y(w)dw\right]\mu^I(d\lambda^I,d\gamma^I)\\
&=\int_{0}^{t}e^{-\nu((v,t])}\E\left[\lambda_{0}^I(t)\indic{v\leq\eta_{0}}\exp\left(-\int_{0}^{t}\gamma_{0}^I(r)y(r)dr\right)\right]\nu(dv)\\
&\hspace{1cm}+\int_{0}^{t}\int_{w}^{t}e^{-\nu((v,t])}\E\left[\lambda^I(t-w)\indic{v-w\leq\eta}\exp\left(-\int_{w}^{t}\gamma^I(r-w)y(r)dr\right)\right]\nu(dv)x(w)y(w)dw.
\end{aligned}
\end{multline*}
\subsubsection{proof of Lemma~\ref{Lem:exp-II-y}}\label{Lem:proofs-III}
We have
\begin{align*}
&\E\Big[\sum_{j\geq1}\indic{\mathcal{N}_t=j}\lambda_{A(t)}^I(t-\tau_{A(t)})\indic{T^V_j\leq \tau_{A(t)}\leq t}\indic{ T^V_{j}-\tau_{A(t)}\leq\eta_{A(t)}}\Big]\\
&\hspace{2cm}=\sum_{i\geq1}\sum_{j\geq1}\E\left[\indic{A(t)=i}\indic{T^V_j\leq\tau_i\leq t}\indic{\mathcal{N}_t=j}\lambda_{i}^I(t-\tau
_{i})\right].
\end{align*}
Note that as $Q^V([T^V_j,t])=0$ and $T^V_j\leq\tau_i,$ by definition $\forall r\in[\tau_i,t],\,\gamma_{A(r),\mathcal{N}_r}(r)=\gamma_{i,j}(r)=\gamma_i^I(r-\tau_i)$. Consequently, by definition of the process $A$ and the fact that $Q$ and $Q^V$ are independent, we have,

\begin{align*}
&\mathbb{P}(A(t)=A(\tau_i),Q^V((T^V_j,t])=0\big|\mathcal{F}_{\tau_i}\vee\mathcal{G}_{T^V_j})\indic{T^V_j\leq\tau_i\leq t}\\
&=e^{-\nu((T^V_j,t])}\mathbb{P}\left(Q\left((s,u)\in\R^2_+,\,\tau_i\leq s\leq t,\,\gamma_i^I(r-\tau_i)y(r)\geq u\right)=0\big|\mathcal{F}_{\tau_i}\right)\indic{T^V_j\leq\tau_i\leq t}\\
&=e^{-\nu((T^V_j,t])}\exp\left(-\int_{\tau_i}^{t}\gamma_i^I(r-\tau_i)y(r)dr\right)\indic{T^V_j\leq\tau_i\leq t},
\end{align*}
since $Q\big|_{(\tau_i,t]}$ and $\mathcal{F}_{\tau_i}$ are independent.

As a result, as in the proof of Lemma~\ref{lem:Inf}
\begin{multline*}
\E\Big[\sum_{j\geq1}\indic{\mathcal{N}_t=j}\lambda_{A(t)}^I(t-\tau_{A(t)})\indic{T^V_j\leq \tau_{A(t)}\leq t}\indic{ T^V_{j}-\tau_{A(t)}\leq\eta_{A(t)}}\Big]\\
\begin{aligned}
&=\sum_{i\geq1}\sum_{j\geq1}\E\left[\indic{T^V_j\leq\tau_{i}\leq t}\lambda_{i}^I(t-\tau
_{i})\mathbb{P}(A(t)=A(\tau_i),Q^V((T^V_j,t])=0\big|\mathcal{F}_{\tau_i}\vee\mathcal{G}_{T^V_j})\right]\\
&=\sum_{i\geq1}\E\left[\indic{\tau_{i}\leq t}\int_0^{\tau_{i}}e^{-\nu((v,t])}\lambda_{i}^I(t-\tau_i)\exp\left(-\int_{\tau_i}^{t}\gamma_{i}^I(r-\tau_i)y(r)dr\right)Q^V(dv)\right]\\
&=\sum_{i\geq1}\E\left[\indic{\tau_{i}\leq t}\int_0^{\tau_{i}}e^{-\nu((v,t])}\lambda_{i}^I(t-\tau_i)\exp\left(-\int_{\tau_i}^{t}\gamma_{i}^I(r-\tau_i)y(r)dr\right)\nu(dv)\right]\\
&=\sum_{i\geq1}\int_{D^2}\E\left[\indic{\tau_{i}\leq t}\int_0^{\tau_{i}}e^{-\nu((v,t])}\lambda^I(t-\tau_i)\exp\left(-\int_{\tau_i}^{t}\gamma^I(r-\tau_i)y(r)dr\right)\nu(dv)\right]\mu^I(d\lambda^I,d\gamma^I),
\end{aligned}
\end{multline*}
where the last line follows from the fact that $(\lambda_i^I,\gamma_i^I)$ and $\tau_i$ are independent.

Moreover, as $(\tau_i)_i$ is the jump time of process $A$, we have the following equality,
\begin{multline*}
\begin{aligned}
&=\int_{D^2}\E\left[\int_{0}^{t}\int_{0}^{w}e^{-\nu((v,t])}\lambda^I(t-w)\exp\left(-\int_{w}^{t}\gamma^I(r-w)y(r)dr\right)\nu(dv)A(dw)\right]\mu^I(d\lambda^I,d\gamma^I).
\end{aligned}
\end{multline*}
By a martingale argument as in \cite[Section~$6$, page~$26$]{forien-Zotsa2022stochastic} we obtain,
\begin{multline*}
\begin{aligned}
&=\int_{D^2}\E\left[\int_{0}^{t}\int_{0}^{w}e^{-\nu((v,t])}\lambda^I(t-w)\exp\left(-\int_{w}^{t}\gamma^I(r-w)y(r)dr\right)\nu(dv)\right.\\&\hspace{8cm}\left.\gamma_{A(w),\mathcal{N}_w}(w)y(w)dw\right]\mu^I(d\lambda^I,d\gamma^I)\\
&=\int_{0}^{t}\int_{0}^{w}e^{-\nu((v,t])}\E\left[\lambda^I(t-w)\exp\left(-\int_{w}^{t}\gamma^I(r-w)y(r)dr\right)\right]\nu(dv)x(w)y(w)dw.
\end{aligned}
\end{multline*}

\subsubsection{proof of Lemma~\ref{Lem:-II-x}}\label{Lem:proofs-IV}
Note that, as $\mathcal{N}_t=j$ and $T^V_j>\eta_0,$ by definition for all $r\in[T^V_j,t],\,\gamma_{A(r),\mathcal{N}_r}(r)=\gamma_{i,j}(r)=\gamma^V_j(r-T^V_j).$ As a result, as $Q^V$ and $Q$ are independent and $\gamma^V_j$ and $T^V_j$ are also independent, we have
\begin{multline*}
\E\Big[\indic{A(t)=0}\sum_{j\geq1}\indic{\mathcal{N}_t=j}\gamma_{j}^V(t-T^V_j)\indic{ T^V_{j}>\eta_{0}}\Big]\\
\begin{aligned}
&=\E\Big[\sum_{j\geq1}\indic{P(0,T^V_j,\gamma)=0}\indic{P(T^V_j,t,\gamma^V_j)=0}\indic{\mathcal{N}_t=j}\gamma_{j}^V(t-T^V_j)\indic{ T^V_{j}>\eta_{0}}\Big]\\
&=\E\Big[\sum_{j\geq1}\indic{P(0,T^V_j,\gamma)=0}\E\Big[\indic{\mathcal{N}_t=j}\indic{P(T^V_j,t,\gamma^V_j)=0}\gamma_{j}^V(t-T^V_j)\big|\mathcal{F}_{T^V_j-}\vee\mathcal{G}_{T^V_j-}\Big]\indic{ T^V_{j}>\eta_{0}}\Big]\\
&=\E\Big[\sum_{j\geq1}\indic{P(0,T^V_j,\gamma)=0}\gamma^V_j(t-T^V_j)e^{-\nu((T^V_j,t])}\exp\left(-\int_{T^V_j}^{t}\gamma^V_j(r-T^V_j)y(r)dr\right)\indic{ T^V_{j}>\eta_{0}}\Big]\\
&=\E\Big[\sum_{j\geq1}\indic{P(0,T^V_j,\gamma)=0}\E\Big[\gamma^V_j(t-T^V_j)\exp\left(-\int_{T^V_j}^{t}\gamma^V_j(r-T^V_j)y(r)dr\right)\big|\mathcal{F}_{T^V_j}\Big]\\
&\hspace{6cm}\times e^{-\nu((T^V_j,t])}\indic{ T^V_{j}>\eta_{0}}\Big]\\
&=\E\Big[\sum_{j\geq1}\indic{P(0,T^V_j,\gamma)=0}\int_{D}\gamma^V(t-T^V_j)\exp\left(-\int_{T^V_j}^{t}\gamma^V(r-T^V_j)y(r)dr\right)\mu^V(d\gamma^V)\\&\hspace{6cm}\times e^{-\nu((T^V_j,t])}\indic{ T^V_{j}>\eta_{0}}\Big]
\end{aligned}
\end{multline*}
Consequently, by definition of the Poisson random measure $Q^V$, we obtain
\begin{multline*}
\begin{aligned}
&=\E\Big[\int_0^t\indic{P(0,v,\gamma)=0}\int_{D}\gamma^V(t-v)\exp\left(-\int_{v}^{t}\gamma^V(r-v)y(r)dr\right)\mu^V(d\gamma^V)e^{-\nu((v,t])}\indic{ v>\eta_{0}}Q^V(dv)\Big]\\
&=\E\Big[\int_0^t\indic{P(0,v,\gamma)=0}\int_{D}\gamma^V(t-v)\exp\left(-\int_{v}^{t}\gamma^V(r-v)y(r)dr\right)\mu^V(d\gamma^V)e^{-\nu((v,t])}\indic{ v>\eta_{0}}\nu(dv)\Big]\\
&=\int_0^te^{-\nu((v,t])}\E\Big[\indic{P(0,v,\gamma)=0}\indic{ v>\eta_{0}}\Big]\int_{D}\gamma^V(t-v)\exp\left(-\int_{v}^{t}\gamma^V(r-v)y(r)dr\right)\mu^V(d\gamma^V)\nu(dv)
\end{aligned}
\end{multline*} 
This concludes the proof of the Lemma.

\subsubsection{proof of Lemma~\ref{Lem:-III-x}}\label{Lem:proofs-V}
Note that, as $\mathcal{N}_t=j$ and $T^V_j-\tau_i>\eta_i,$ by definition of $\gamma$, for all $r\in[T^V_j,t],\,\gamma_{i,j}(r)=\gamma^V_j(r-T^V_j).$ As a result, as $Q^V$ and $Q$ are independent and $\gamma^V_j$ and $T^V_j$ are also independent, we have
\begin{multline*}
\sum_{i\geq1}\E\Big[\indic{A(t)=i}\sum_{j\geq1}\indic{\mathcal{N}_t=j}\gamma_{j}^V(t-\tau_{i})\indic{ T^V_{j}-\tau_{i}>\eta_{i}}\Big]\\
\begin{aligned}
&=\sum_{i\geq1}\sum_{j\geq1}\E\Big[\indic{P(\tau_i,T^V_j,\gamma)=0}\indic{P(T^V_j,t,\gamma^V_j)=0}\indic{\mathcal{N}_t=j}\gamma_{j}^V(t-\tau_{i})\indic{ T^V_{j}-\tau_{i}>\eta_{i}}\Big]\\
&=\sum_{i\geq1}\sum_{j\geq1}\E\Big[\indic{P(\tau_i,T^V_j,\gamma)=0}\E\Big[\indic{\mathcal{N}_t=j}\indic{P(T^V_j,t,\gamma^V_j)=0}\gamma_{j}^V(t-T^V_j)\big|\mathcal{F}_{T^V_j-}\vee\mathcal{G}_{T^V_j-}\Big]\indic{ T^V_{j}-\tau_i>\eta_{i}}\Big]\\
&=\sum_{i\geq1}\sum_{j\geq1}\E\Bigg[\indic{P(\tau_{i},T^V_j,\gamma)=0}\E\Big[\gamma^V_j(t-T^V_j)\exp\left(-\int_{T^V_j}^{t}\gamma^V_j(r-T^V_j)y(r)dr\right)\big|\mathcal{F}_{T^V_j-}\Big]\\
&\hspace{3cm}\times e^{-\nu((T^V_j,t])}\indic{ T^V_{j}-\tau_i>\eta_{i}}\Bigg]\\
&=\sum_{i\geq1}\sum_{j\geq1}\E\Big[\indic{P(\tau_i,T^V_j,\gamma)=0}\int_{D}\gamma^V(t-T^V_j)\exp\left(-\int_{T^V_j}^{t}\gamma^V(r-T^V_j)y(r)dr\right)\mu^V(d\gamma^V)\\
&\hspace{5cm}\times e^{-\nu((T^V_j,t])}\indic{ T^V_{j}-\tau_i>\eta_{i}}\Big]\\
&=\sum_{i\geq1}\E\Big[\int_{0}^{t}\indic{P(\tau_i,v,\gamma)=0}\indic{ v-\tau_i>\eta_{i}}\int_{D}\gamma^V(t-v)\exp\left(-\int_{v}^{t}\gamma^V(r-v)y(r)dr\right)\mu^V(d\gamma^V)\\
&\hspace{6cm}\times e^{-\nu((v,t])}Q^V(dv)\Big]\\
&=\sum_{i\geq1}\E\Bigg[\int_{0}^{t}\indic{P(\tau_i,v,\gamma)=0}\indic{ v-\tau_i>\eta_{i}}\int_{D}\gamma^V(t-v)\exp\left(-\int_{v}^{t}\gamma^V(r-v)y(r)dr\right)\mu^V(d\gamma^V)\\
&\hspace{4cm}\times e^{-\nu((v,t])}\nu(dv)\Bigg]\\
&=\int_{0}^{t}\E\Big[\sum_{i\geq1}\indic{P(\tau_i,v,\gamma)=0}\indic{ v-\tau_i>\eta_{i}}\Big]\int_{D}\gamma^V(t-v)\exp\left(-\int_{v}^{t}\gamma^V(r-v)y(r)dr\right)\mu^V(d\gamma^V)\\
&\hspace{7cm}\times e^{-\nu((v,t])}\nu(dv).
\end{aligned}
\end{multline*}
This concludes the proof.

\bibliographystyle{plain}
\bibliography{biblio.bib}

@article{bolatova2024mathematical,
	title={Mathematical modeling of infectious diseases and the impact of vaccination strategies},
	author={Bolatova, Diana and Kadyrov, Shirali and Kashkynbayev, Ardak},
	journal={Mathematical Biosciences and Engineering},
	volume={21},
	number={9},
	pages={7103},
	year={2024},
	publisher={American Institute of Mathematical Sciences}
}

@article{ma2024stability,
	title={Stability analysis of SIRS model considering pulse vaccination and elimination disturbance},
	author={Ma, Yanli and Zuo, Xuewu},
	journal={Journal of Mathematics},
	volume={2024},
	number={1},
	pages={6617911},
	year={2024},
	publisher={Wiley Online Library}
}

@article{breton2014factorial,
	title={Factorial moments of point processes},
	author={Breton, Jean-Christophe and Privault, Nicolas},
	journal={Stochastic Processes and their Applications},
	volume={124},
	number={10},
	pages={3412--3428},
	year={2014},
	publisher={Elsevier}
}

@article{privault2012moments,
	title={Moments of Poisson stochastic integrals with random integrands},
	author={Privault, Nicolas},
	journal={arXiv preprint arXiv:1204.4854},
	year={2012}
}

@article{decreusefond2014moment,
	title={Moment formulae for general point processes},
	author={Decreusefond, Laurent and Flint, Ian},
	journal={Journal of Functional Analysis},
	volume={267},
	number={2},
	pages={452--476},
	year={2014},
	publisher={Elsevier}
}

@article{hadjipanayis2019compliance,
	title={Compliance with vaccination schedules},
	author={Hadjipanayis, Adamos},
	journal={Human Vaccines \& Immunotherapeutics},
	volume={15},
	number={4},
	pages={1003--1004},
	year={2019},
	publisher={Taylor \& Francis},
	doi={10.1080/21645515.2018.1556078}
}

@misc{who2016grisp,
	title={Global Routine Immunization Strategies and Practices (GRISP)},
	author={World Health Organization},
	year={2016},
	howpublished={\url{https://apps.who.int/iris/handle/10665/204500}},
	note={Accessed November 2025}
}

@misc{eberhard1990nonlinear,
	title={Nonlinear functional analysis and its applications},
	author={Eberhard, Zeidler},
	year={1990},
	publisher={Springer Verlag}
}

@book{precup2002methods,
	title={Methods in nonlinear integral equations},
	author={Precup, Radu},
	year={2002},
	publisher={Springer Science \& Business Media}
}

@article{forster2014leray,
	title={Leray-Schauder Existence Theory for Quasilinear Elliptic Equations},
	author={Forster, Sam and Gleeson, Eavan and Hoffmann, Franca},
	journal={Lecture notes},
	year={2014}
}

@article{cherkas1970compactness,
	title={Compactness in $L^\infty$ spaces},
	author={Cherkas, Barry M},
	journal={Proceedings of the American Mathematical Society},
	volume={25},
	number={2},
	pages={347--350},
	year={1970},
	publisher={JSTOR}
}

@article{jin2008pulse,
	title={Pulse vaccination in the periodic infection rate SIR epidemic model},
	author={Jin, Zhen and Haque, Mainul and Liu, Quanxing},
	journal={International Journal of Biomathematics},
	volume={1},
	number={04},
	pages={409--432},
	year={2008},
	publisher={World Scientific}
}

@article{shulgin1998pulse,
	title={Pulse vaccination strategy in the SIR epidemic model},
	author={Shulgin, Boris and Stone, Lewi and Agur, Zvia},
	journal={Bulletin of mathematical biology},
	volume={60},
	number={6},
	pages={1123--1148},
	year={1998},
	publisher={Springer}
}

@article{meng2010delay,
	title={A delay SIR epidemic model with pulse vaccination and incubation times},
	author={Meng, Xinzhu and Chen, Lansun and Wu, Bo},
	journal={Nonlinear analysis: real world applications},
	volume={11},
	number={1},
	pages={88--98},
	year={2010},
	publisher={Elsevier}
}

@phdthesis{nagy2011epidemic,
	title={Epidemic models with pulse vaccination and time delay},
	author={Nagy, Lisa},
	year={2011},
	school={University of Waterloo}
}

@incollection{liu2017pulse,
	title={Pulse Control Strategies},
	author={Liu, Xinzhi and Stechlinski, Peter},
	booktitle={Infectious Disease Modeling: A Hybrid System Approach},
	pages={179--226},
	year={2017},
	publisher={Springer}
}

@article{wang2014pulse,
	title={On pulse vaccine strategy in a periodic stochastic SIR epidemic model},
	author={Wang, Fengyan and Wang, Xiaoyi and Zhang, Shuwen and Ding, Changming},
	journal={Chaos, Solitons \& Fractals},
	volume={66},
	pages={127--135},
	year={2014},
	publisher={Elsevier}
}

@article{gao2006analysis,
	title={Analysis of a delayed epidemic model with pulse vaccination and saturation incidence},
	author={Gao, Shujing and Chen, Lansun and Nieto, Juan J and Torres, Angela},
	journal={Vaccine},
	volume={24},
	number={35-36},
	pages={6037--6045},
	year={2006},
	publisher={Elsevier}
}

@article{stone2000theoretical,
	title={Theoretical examination of the pulse vaccination policy in the SIR epidemic model},
	author={Stone, L and Shulgin, B and Agur, Zvia},
	journal={Mathematical and computer modelling},
	volume={31},
	number={4-5},
	pages={207--215},
	year={2000},
	publisher={Elsevier}
}

@article{gao2007analysis,
	title={Analysis of an SIR epidemic model with pulse vaccination and distributed time delay},
	author={Gao, Shujing and Teng, Zhidong and Nieto, Juan J and Torres, Angela},
	journal={BioMed Research International},
	volume={2007},
	number={1},
	pages={064870},
	year={2007},
	publisher={Wiley Online Library}
}

@article{CLT_zotsa2025stochastic,
	author = {Zotsa Ngoufack, Arsene Brice},
	title = {Functional central limit theorems and SPDE for epidemic model with memory on the last infection and waning immunity},
	journal = {Arxiv preprint arXiv:2505.15617},
	year = 2025
}

@article{ngoufack2025functional,
	title={Functional Central Limit Theorems for epidemic models with varying infectivity and waning immunity},
	author={Zotsa Ngoufack, Arsene Brice},
	journal={ESAIM: Probability and Statistics},
	volume={29},
	pages={45--112},
	year={2025},
	publisher={EDP Sciences}
}

@book{gripenberg_volterra_1990,
	title = {Volterra Integral and Functional Equations},
	author = {Gripenberg, Gustaf and Londen, Stig-Olof and Staffans, Olof},
	year = {1990},
	number = {34},
	publisher = {Cambridge University Press}
}

@book{billingsley1999convergence,
  title={Convergence of probability measures},
  author={Billingsley, Patrick},
  year={1999},
  publisher={John Wiley \& Sons},
}

@article{forien2026stochastic,
	title={Stochastic epidemic model with varying infectivity and waning immunity: the law of large numbers with unbounded infectivity},
	author={Forien, Rapha{\"e}l and Pardoux, {\'E}tienne},
	journal={arXiv preprint arXiv:2606.11845},
	year={2026}
}

@article{baghdadi2026stochastic,
	title={Stochastic epidemic model for malaria: The law of large numbers},
	author={Baghdadi, Othmane and Pardoux, {\'E}tienne},
	journal={Discrete and Continuous Dynamical Systems},
	volume={58},
	pages={323--355},
	year={2026},
	publisher={Discrete and Continuous Dynamical Systems}
}

@article{pang2022functional,
title={Functional limit theorems for non-Markovian epidemic models},
author={Pang, Guodong and Pardoux, {\'E}tienne},
journal={The Annals of Applied Probability},
volume={32},
number={3},
pages={1615--1665},
year={2022},
publisher={Institute of Mathematical Statistics}
}

@book{ccinlar2011probability,
  title={Probability and {S}tochastics},
  author={{\c{C}}{\i}nlar, Erhan},
  volume={261},
  year={2011},
  publisher={Springer Science \& Business Media},
}

@book{britton2019stochastic,
  title={Stochastic epidemic models with inference},
  author={Britton, Tom and Pardoux, Etienne and Ball, Franck and Laredo, Catherine and Sirl, David and Tran, Viet Chi},
  volume={132},
  year={2019},
  publisher={Springer},
}

@article{forien_epidemic_2021,
  title={Epidemic models with varying infectivity},
  author={Forien, Rapha{\"e}l and Pang, Guodong and Pardoux, {\'E}tienne},
  journal={SIAM Journal on Applied Mathematics},
  volume={81},
  number={5},
  pages={1893--1930},
  year={2021},
  publisher={SIAM},
}

@article{forien-Zotsa2022stochastic,
	title={Stochastic epidemic models with varying infectivity and waning immunity},
	author={Forien, Rapha{\"e}l and Pang, Guodong and Pardoux, Etienne and Zotsa-Ngoufack,Arsene-Brice},
	journal={The Annals of Applied Probability},
	volume={35},
	number={3},
	pages={2175--2216},
	year={2025},
	publisher={Institute of Mathematical Statistics}
}

@article{rao1963law,
  title={The Law of Large Numbers for {$D[0,1]$}-Valued Random Variables},
  author={Rao, R Ranga},
  journal={Theory of Probability \& Its Applications},
  volume={8},
  number={1},
  pages={70--74},
  year={1963},
  publisher={SIAM},
}

@article{foutel2025optimal,
title={Optimal vaccination policy to prevent endemicity: a stochastic model},
author={Foutel-Rodier, F{\'e}lix and Charpentier, Arthur and Gu{\'e}rin, H{\'e}l{\`e}ne},
journal={Journal of Mathematical Biology},
volume={90},
number={1},
pages={1--55},
year={2025},
publisher={Springer}
}

@article{guerin2025stochastic,
title={A stochastic epidemic model with memory of the last infection and waning immunity},
author={Gu{\'e}rin, H{\'e}l{\`e}ne and Zotsa-Ngoufack, Arsene Brice},
journal={arXiv preprint arXiv:2505.00601},
year={2025}
}

@incollection{feller2015integral,
	title={On the integral equation of renewal theory},
	author={Feller, William},
	booktitle={Selected Papers I},
	pages={567--591},
	year={2015},
	publisher={Springer},
}

\appendix
\section{Computation to recover the ODE in Example~\ref{Ex-Model-SIRS}}\label{se-set-example}
\subsection{Porportion of infected individuals $\bar{I}$}
From~\eqref{Pop-eq-I} and the fact that 
\[F_{I,0}^c(x)=F_I^c(x)=e^{-\sigma x}\]
we derive that, for $t\geq0,$
\begin{equation}\label{A-I}
\bar{I}(t)=\bar{I}(0)\,e^{-\sigma t}+\beta\int_0^t e^{-(t-s)\sigma}\bar{S}(s)\bar{I}(s)ds.
\end{equation}
Therefore by derivation, we obtain
\[\frac{d \bar{I}}{dt}(t) = \beta \bar{S}(t) \bar{I}(t) - \sigma \bar{I}(t) .\]
\subsection{Proportion of recovered individuals $\bar{R}$}
From~\eqref{Pop-eq-R}, and the fact that
\[G^c_0(t|x)=G^c(t|x)=e^{-\theta t},\] 
\begin{equation}\label{A-R}
\begin{aligned}
\bar{R}(t)
&=\bar{R}(0) e^{-\theta t}e^{-\nu((0,t])}+\bar{I}(0)\sigma\int_{0}^{t} e^{-\nu((x,t])}e^{-\theta(t-x)} e^{-\sigma x}dx\\
&\hspace{1cm}+\beta\sigma\int_{0}^{t}\int_0^{t-s} e^{-\nu((s+x,t])}e^{-\theta(t-s-x)}e^{-\sigma x}dx\bar{S}(s)\bar{I}(s)ds					
\end{aligned}
\end{equation}
For $t\in (\tau_n,\,\tau_{n+1}),$ we recall that $\partial_t\nu((v,t])=w(t).$ Taking the differential of \eqref{A-R}, it follows that,
\begin{equation*}
\begin{aligned}
&\frac{d\bar{R}}{dt}(t)=-(\theta+w(t))\bar{R}(t)+\bar{I}(0)\sigma \,e^{-\sigma t}
+\beta\sigma\int_0^t e^{-(t-s)\sigma}\bar{S}(s)\bar{I}(s)ds.						
\end{aligned}
\end{equation*}

As a result, from~\eqref{A-I}, it follows that,
\begin{equation*}
\partial_t\bar{R}(t)=-(\theta+w(t))\bar{R}(t)+\sigma\bar{I}(t).						
\end{equation*}
Noting that, for any $v < \tau_n,\,\nu((v, \tau_n^+])=\alpha_n+\nu((v, \tau_n^-]),$ from~\eqref{A-R}, we derive that, 
\begin{equation*}
\bar{R}(\tau_n^+) =  e^{-\alpha_n}\bar{R}(\tau_n^-).
\end{equation*}
\subsection{Proportion of vaccinated individuals $V$}
From~\eqref{Pop-eq-V} and the fact that, 
\[F_{V,}^c(t)=e^{-\kappa t},\]
we derive that,
\begin{equation}\label{A-V}
V(t)=\int_0^t\big(1-\bar{I}(v)\big)e^{-\nu((v,t])}e^{-(t-v)\kappa}\nu(dv),
\end{equation}
and by derivation, for $t\in (\tau_n,\,\tau_{n+1}),$ as $\partial_t\nu((v,t])=w(t),$ we obtain that,
\begin{equation*}
\frac{dV}{dt}(t)=\big(1-\bar{I}(t)\big)w(t)-(\kappa+w(t)) V(t)
\end{equation*}
Moreover from~\eqref{A-V} we also deduce that,
\begin{equation*}
V(\tau_n^+) =  e^{-\alpha_n}V(\tau_n^-).
\end{equation*}
\subsection{Proportion of susceptible individuals}
From~\eqref{Pop-eq-S} 
\begin{equation}
\begin{aligned}
&\bar{S}(t)=
\bar{S}(0)e^{-\nu((0,t])}\exp\left(-\beta\int_{0}^{t}\bar{I}(r)dr\right)+\bar{R}(0)\int_{0}^{t}\exp\left(-\beta\int_{y}^{t}\bar{I}(r)dr\right)G^c_0(dy|0)\,e^{-\nu((0,t])}\\
&\hspace{2cm}+\bar{I}(0)\int_{0}^{t}\int_0^{t-x} e^{-\nu((x,t])}\exp\left(-\beta\int_{x+u}^{t}\bar{I}(r)dr\right)G_0^c(du|x)F^c_{I,0}(dx)\\
&\hspace{1cm}+\beta\int_{0}^{t}\int_0^{t-s}\int_{0}^{t-s-x} e^{-\nu((s+x,t])}\exp\left(-\beta\int_{s+x+y}^{t}\bar{I}(r)dr\right)G^c(dy|x)\bar{S}(s)\bar{I}(s)F^c_I(dx)ds\\
&\hspace{1cm}+\int_0^t\int_{0}^{t-v}\big(1-\bar{I}(v)\big)e^{-\nu((v,t])}\exp\left(-\beta\int_{v+x}^{t}\bar{I}(r)dr\right) F^c_V(dx)\nu(dv)\\
&:=B_1(t)+B_2(t)+B_3(t)+B_4(t).							
\end{aligned}
\end{equation}
We recall that,
\[F_{I,0}^c(x)=F_I^c(x)=e^{-\sigma x},\quad G^c_0(t|x)=G^c(t|x)=e^{-\theta t},\quad F_{V,}^c(t)=e^{-\kappa t},\]
and for $t\in (\tau_n,\,\tau_{n+1}[,$ $\partial_t\nu((v,t])=w(t).$ 

Consequently, for $t\in (\tau_n,\,\tau_{n+1}[,$ we obtain
\begin{equation*}
\partial_t B_1(t)=-(w(t)+\beta\bar{I}(t))B_1(t),
\end{equation*}
\begin{align*}
\partial_t B_2(t)&=-(w(t)+\beta\bar{I}(t))B_2(t)\non\\
&\quad+\theta\Big(\bar{R}(0)G_0^c(t|0)e^{-\nu((0,t])}+\int_{0}^{t}e^{-\nu((x,t])}G^c_0(t-x|x)F_{I,0}^c(dx)\Big),
\end{align*}
\begin{align*}
\partial_t	B_3(t)&=\beta \theta\int_{0}^{t}\int_{0}^{t-s}e^{-\nu((s+x,t])}G^c(t-s-x|x)F^c_I(dx)\bar{S}(s)\bar{I}(s)ds-(w(t)+\beta\bar{I}(t))B_3(t),
\end{align*}
and 
\begin{align*}
\partial_t B_4(t)=\kappa\int_{0}^{t}\big(1-\bar{I}(v)\big) e^{-\nu((v,t])}F^c_V(t-v)\nu(dv)-(w(t)+\beta\bar{I}(t))B_4(t).
\end{align*}
As a result, it follows that, for $t\in (\tau_n,\,\tau_{n+1}[,$
\begin{align*}
\frac{d\bar{S}}{dt}(t)&=-(w(t)+\beta\bar{I}(t))\bar{S}(t)\\ 
&\quad+ \theta \Bigg( \bar{R}(0) G^c_0(t|0) e^{-\nu((0,t])} + \int_0^t e^{-\nu((x,t])}G^c_0(t-x|x)  F_{I,0}^c(dx)\\&\hspace{3cm}+\int_{0}^{t}\int_{0}^{t-s}e^{-\nu((s+x,t])}G^c(t-s-x|x)F^c_I(dx)\bar{S}(s)\bar{I}(s)ds \Bigg)\\&+\kappa\int_{0}^{t}\big(1-\bar{I}(v)\big) e^{-\nu((v,t])}F^c_V(t-v)\nu(dv)\\
&=-(w(t)+\beta\bar{I}(t))\bar{S}(t)+ \theta \bar{R}(t) +\kappa V(t),
\end{align*}
where the last line follows from the expressions of $\bar{I},\,\bar{R}$ and $V.$

Moreover, we also note that,
\begin{equation*}
\bar{S}(\tau_n^+) =  e^{-\alpha_n}\bar{S}(\tau_n^-).
\end{equation*}
On the other hand, as for $t\geq0$, $\bar{S}(t)+\bar{I}(t)+\bar{R}(t)+V(t)=1$, we derive that,
\begin{align*}
V(\tau_n^+) &= 1 - \bar{I}(\tau_n) - \bar{S}(\tau_n^+) - \bar{R}(\tau_n^+) \\
&= 1 - \bar{I}(\tau_n) - (\bar{S}(\tau_n^-) + \bar{R}(\tau_n^-))e^{-\alpha_n}.
\end{align*}
Substituting $1 - \bar{I}(\tau_n) = \bar{S}(\tau_n^-) + \bar{R}(\tau_n^-) + V(\tau_n^-)$, we recover:
\begin{equation*}
V(\tau_n^+) = V(\tau_n^-) + (\bar{S}(\tau_n^-) + \bar{R}(\tau_n^-))(1 - e^{-\alpha_n}).
\end{equation*}
This concludes the proof.
\end{document}